\documentclass[reqno]{amsart}
\usepackage{amsfonts}
\usepackage{amsmath}
\usepackage{amssymb}
\usepackage{mathtools}
\usepackage[hidelinks]{hyperref}

\newtheorem{theorem}{Theorem}[section]
\theoremstyle{plain}

\newtheorem{corollary}[theorem]{Corollary}

\newtheorem{definition}[theorem]{Definition}
\newtheorem{example}[theorem]{Example}

\newtheorem{lemma}[theorem]{Lemma}

\newtheorem{question}[theorem]{Question}
\newtheorem{proposition}[theorem]{Proposition}
\newtheorem{remark}[theorem]{Remark}

\numberwithin{equation}{section}
\newcommand{\func}[1]{\operatorname{#1}}
\newtheorem{maintheorem}{Theorem}

\begin{document}
\title[Abelian maximal pattern complexity and extremal words]{Abelian
maximal pattern complexity and extremal words}
\author{Qingcheng Zeng}
\address{School of Mathematical Sciences, Beihang University, Beijing
100083, P. R. China}
\email{qczeng@buaa.edu.cn}
\author{Yumei Xue}
\address{School of Mathematical Sciences, Beihang University, Beijing
100083, P. R. China}
\email{yxue@buaa.edu.cn}
\author{Cheng Zeng}
\address{School of Mathematics and Information Science, Shandong Technology
and Business University, Yantai, Shandong Province 264003, P. R. China}
\email{czeng@sdtbu.edu.cn}

\begin{abstract}
We study the Abelian maximal pattern complexity $p_{\alpha }^{\ast \mathrm{ab%
}}(k)$ of infinite words over finite alphabets. For recurrent aperiodic
words using exactly $\ell\geq 2$ letters, we obtain a sharp lower bound $%
p_{\alpha }^{\ast \mathrm{ab}}(k)\geq k+\ell-1$ and give a structural
characterization of those words attaining equality for every $k\geq 1$.

In the general case, we prove that an infinite word $\alpha $ is
aperiodic if and only if $\binom{p_{\alpha }^{\ast \mathrm{ab}}(k)}{2}\geq k$
for every positive integer $k$. For aperiodic words over $\ell\geq 2$
letters, each occurring infinitely often, we further prove that $p_{\alpha
}^{\ast \mathrm{ab}}(k)\geq m$ whenever $\binom{m}{2}\leq(\ell-1)(k-\ell+2)$%
, for all positive integers $m,k$. Together with a matching construction,
this shows that the minimum Abelian maximal pattern complexity in this class
is $\sqrt{2(\ell-1)k}+O_{\ell}(1)$.

We call a word an Abelian pattern Sturmian word if, for every positive
integer $k$, its Abelian maximal pattern complexity is the least positive
integer $m$ satisfying $\binom{m}{2}\geq k$. We show that a word is Abelian
pattern Sturmian if and only if, after relabeling its alphabet, it is the
characteristic word of an infinite set $E\subset\mathbb{N}_0$ whose
bipartite incidence graph is a forest. This graph has two disjoint copies of 
$\mathbb{N}_0$ as its vertex classes, with a left vertex $r$ adjacent to a
right vertex $s$ exactly when $r+s\in E$.
\end{abstract}

\subjclass[2020]{Primary 37B10; Secondary 68R15.}
\thanks{Yumei Xue is the corresponding author. }
\keywords{Abelian maximal pattern complexity; infinite words; eventual
periodicity; Abelian pattern Sturmian words}
\maketitle


\section{Introduction}

Topological entropy is one of the fundamental invariants of dynamical
systems, measuring the exponential growth of distinguishable orbit segments.
For systems of zero entropy, this exponential scale no longer distinguishes
the remaining orbit complexity, and finer growth functions become natural.
In symbolic dynamics, one such fundamental invariant is factor complexity.
Let $\mathbb{A}$ be a finite alphabet. An element $\alpha =\alpha _{0}\alpha
_{1}\alpha _{2}\cdots \in \mathbb{A}^{\mathbb{N}_{0}},$\ where $\mathbb{N}%
_{0}=\{0,1,2,\ldots \},$ is called a (right) \emph{infinite word}. For an
infinite word $\alpha \in \mathbb{A}^{\mathbb{N}_{0}}$ and a positive
integer $n\in \mathbb{N}$, let $\mathcal{F}_{\alpha }(n)=\{\alpha _{i}\alpha
_{i+1}\cdots \alpha _{i+n-1}:i\in \mathbb{N}_{0}\}$ denote the set of all 
\emph{factors }(\emph{blocks}) of $\alpha $ of length $n$, and define the 
\emph{factor complexity} of $\alpha $ by $p_{\alpha }(n)=\sharp \mathcal{F}%
_{\alpha }(n),$ where $\sharp E$ denotes the cardinality of a finite set $E.$
If $\mathcal{O}(\alpha )$ denotes the orbit closure of $\alpha $ under the
left shift, then 
\begin{equation*}
  h_{\mathrm{top}}(\mathcal{O}(\alpha ))=\lim_{n\rightarrow \infty }\frac{1}{n}
\log p_{\alpha }(n).
\end{equation*}
Thus, in the zero-entropy regime, the growth of $p_{\alpha }(n)$ provides a
finer measure of the combinatorial complexity that is invisible at the
entropy scale.

A basic problem is to determine the minimal complexity of aperiodic infinite
words. Recall that an infinite word $\alpha$ is called eventually periodic
if there exist $N\geq 0$ and $p\geq 1$ such that $\alpha_{n+p}=\alpha_n$ for
all $n\geq N$, and aperiodic otherwise. The classical theorem of Morse and
Hedlund \cite{MH1940} gives an answer: $\alpha$ is aperiodic if and only if $%
p_\alpha(n)\geq n+1$ for every $n\in\mathbb{N}$. Words satisfying $%
p_\alpha(n)=n+1$ for every $n\in\mathbb{N}$ are called Sturmian words. 

A coarser complexity is obtained by ignoring the order of letters inside a
factor. For a finite word $u\in \mathbb{A}^{\ast }$, let $\Psi
(u)=(|u|_{a})_{a\in \mathbb{A}}$ denote its \emph{Parikh vector}, where $%
|u|_{a}$ denotes the number of occurrences of the letter $a$ in $u$. Two
finite words $u,v\in \mathbb{A}^{\ast }$ are said to be \emph{Abelian
equivalent}, denoted by $u\sim _{\mathrm{ab}}v,$ if $\Psi (u)=\Psi (v).$
Clearly $\sim _{\mathrm{ab}}$ defines an equivalence relation on $\mathbb{A}%
^{\ast }.$ The \emph{Abelian complexity} of $\alpha $ is defined by $%
p_{\alpha }^{\mathrm{ab}}(n)=\sharp \left( \mathcal{F}_{\alpha }(n)/\sim _{%
\mathrm{ab}}\right) .$ For background on Abelian complexity and related
notions, see \cite{FP2023,RSZ2011}. 

Like factor complexity, Abelian complexity gives characterizations of
periodic bi-infinite words and of Sturmian words \cite{CH1973}. For infinite
words, however, Abelian complexity does not characterize eventual
periodicity: every Sturmian word and the eventually periodic word $%
01^{\infty }=0111\cdots $ have Abelian complexity $2$ at every length.

Kamae and Zamboni \cite{KZ2002a,KZ2002b} introduced a different notion of
complexity of an infinite word, called maximal pattern complexity. For $%
k\geq 1,$ let $\Sigma _{k}(\mathbb{N}_{0}):=\{S\subset \mathbb{N}_{0}:\sharp
S=k\}.$ An element $S=\{s_{1}<s_{2}<\cdots <s_{k}\}\in \Sigma _{k}(\mathbb{N}%
_{0})$ is called a $k$-\emph{pattern}. Define $\alpha \lbrack S]=\alpha
_{s_{1}}\alpha _{s_{2}}\cdots \alpha _{s_{k}}\in \mathbb{A}^{k}.$ For each $%
n\in \mathbb{N}_{0}$, we call the word $\alpha \lbrack n+S]$ an $S$-\emph{%
factor} of $\alpha $, where $n+S:=\{n+s_{1},\ldots ,n+s_{k}\}.$ We denote
the set of all $S$-factors of $\alpha $ by $\mathcal{F}_{\alpha }(S).$ The 
\emph{pattern complexity} $p_{\alpha }(S)$ is defined by $p_{\alpha
}(S):=\sharp \mathcal{F}_{\alpha }(S),$ and the \emph{maximal pattern
complexity} of $\alpha $ is $p_{\alpha }^{\ast }(k):=\sup_{S\in \Sigma _{k}(%
\mathbb{N}_{0})}p_{\alpha }(S).$

Kamae and Zamboni \cite{KZ2002a} proved the maximal-pattern analogue of the
Morse--Hedlund theorem: 
\begin{equation}
\alpha \text{ is aperiodic}\Longleftrightarrow p_{\alpha }^{\ast }(k)\geq 2k%
\text{\ for all }k\in \mathbb{N}.  \label{KZ2002}
\end{equation}%
An aperiodic word is called a \emph{pattern Sturmian word} if $p_{\alpha
}^{\ast }(k)=2k$ for all $k$. Kamae, Rao, Tan and Xue \cite{KRTX2006}
studied the language structure of recurrent pattern Sturmian words and
proposed a classification in terms of their primitive structures. More
recently, Le, Pavlov and Schlortt \cite{LPS} obtained a complete
classification of recurrent pattern Sturmian words and structural results in
the nonrecurrent case.

The relation in (\ref{KZ2002}) was generalized by Kamae and Rao in \cite%
{KR2006}. An infinite word $\alpha $ is called \emph{periodic by projection }%
if there exists a set $\varnothing \neq B\subsetneq \mathbb{A}$ such that 
\begin{equation*}
\pi _{B}(\alpha ):=\mathbf{1}_{B}(\alpha _{0})\mathbf{1}_{B}(\alpha _{1})%
\mathbf{1}_{B}(\alpha _{2})\cdots \in \{0,1\}^{\mathbb{N}_{0}}
\end{equation*}%
is eventually periodic (where $\mathbf{1}_{B}$ denotes the characteristic
function of $B$). We say $\alpha $ is \emph{aperiodic by projection} if $%
\alpha $ is not periodic by projection. Kamae and Rao proved the following
connection between maximal pattern complexity and the above notion of
periodicity.

\begin{theorem}[\protect\cite{KR2006}]
\label{thm:KR2006} Assume $\sharp \mathbb{A}=\ell \geq 2$ and let $\alpha
\in \mathbb{A}^{\mathbb{N}_{0}}$ be aperiodic by projection. Then $p_{\alpha
}^{\ast }(k)\geq \ell k$ for all $k\in \mathbb{N}.$
\end{theorem}

For $\alpha \in \mathbb{A}^{\mathbb{N}_{0}},$ write $\mathbb{A}_{\alpha
}=\{a\in \mathbb{A}:a$ occurs in $\alpha \}$. An infinite word $\alpha $ is
called \emph{periodic} if there exists $q\geq 1$ such that $\alpha
_{i+q}=\alpha _{i}$ for all $i\in \mathbb{N}_{0}.$ An infinite word is
called \emph{recurrent} if its every factor occurs infinitely often.

One of the main tools introduced by Kamae and Rao in \cite{KR2006} is the 
\emph{singular decomposition}. For a recurrent word $\beta $, a letter $a\in 
\mathbb{A}_{\beta }$ is called singular if the indicator word $\pi
_{\{a\}}(\beta )$ is periodic. If $m$ is the least common period of the
indicator words of all singular letters (with $m=1$ when there are no
singular letters), then the subwords 
\begin{equation*}
\beta ^{(j)}:=(\beta _{j+nm})_{n\in \mathbb{N}_{0}},\text{\ }0\leq j<m,
\end{equation*}%
form the singular decomposition of $\beta $. For a non-recurrent word, the
corresponding decomposition is defined through an auxiliary recurrent word
in its orbit closure; we refer to \cite{KR2006} for the precise formulation.

Motivated by this decomposition, for an arbitrary $m\geq 1$ we call 
\begin{equation*}
\alpha ^{(j)}:=(\alpha _{j+nm})_{n\in \mathbb{N}_{0}},\text{\ }0\leq j<m,
\end{equation*}%
the $j$th \emph{residue word} of $\alpha $ modulo $m$. We say that $\alpha $
is an $m$\emph{-interleaving} if, for some $p\in \{0,\ldots ,m-1\}$ and
letters $c_{j}$ $(j\neq p)$, $\alpha ^{(j)}=c_{j}^{\infty }$\ $(j\neq p).$
Writing $x:=\alpha ^{(p)}$, we also say that $\alpha $ is an $m$%
-interleaving of $x$. Such an interleaving is called \emph{separated} if the
letters $c_{j}$ $(j\neq p)$ are pairwise distinct and $\{c_{j}:j\neq p\}\cap 
\mathbb{A}_{x}=\varnothing .$

More recently, Schlortt \cite{Schlortt2026} proved that if $\alpha $ is
aperiodic and $\sharp \mathbb{A}_{\alpha }=\ell \geq 2,$ then $p_{\alpha
}^{\ast }(k)\geq 2k+\ell -2$\ for all $k\in \mathbb{N},$ and this bound is
sharp. Schlortt further determined the structure of the words attaining this
minimum for every $k$.

For $\ell \geq 2,$ let $\mathcal{A}_{\ell }$ denote the set of aperiodic
infinite words over an $\ell $-letter\ alphabet in which every letter occurs
infinitely often.

\begin{theorem}[\protect\cite{Schlortt2026}]
\label{thm:Schlortt} Let $\alpha \in \mathcal{A}_{\ell }$ with $\ell \geq 3$%
. If $p_{\alpha }^{\ast }(k)=2k+\ell -2$ for all $k\in \mathbb{N},$ then $%
\alpha $ is an $m$-interleaving of a binary pattern Sturmian word for some $%
m\geq 2$.
\end{theorem}

Kamae, Widmer and Zamboni \cite{KWZ2015} introduced the Abelian counterpart
of maximal pattern complexity. For a $k$-pattern $S\in \Sigma _{k}(\mathbb{N}%
_{0})$, define $p_{\alpha }^{\mathrm{ab}}(S):=\sharp (\mathcal{F}_{\alpha
}(S)/\sim _{\mathrm{ab}}),$ and the \emph{Abelian maximal pattern complexity}
by 
\begin{equation*}
p_{\alpha }^{\ast \mathrm{ab}}(k):=\sup_{S\in \Sigma _{k}(\mathbb{N}%
_{0})}p_{\alpha }^{\mathrm{ab}}(S).
\end{equation*}%
Clearly $p_{\alpha }^{\ast \mathrm{ab}}(k)\leq p_{\alpha }^{\ast }(k).$ They
proved the following lower bound.

\begin{theorem}[\protect\cite{KWZ2015}]
\label{thm:KWZ2015}Let $\sharp \mathbb{A}=\ell \geq 2.$ If $\alpha \in 
\mathbb{A}^{\mathbb{N}_{0}}$ is recurrent and aperiodic by projection, then $%
p_{\alpha }^{\ast \mathrm{ab}}(k)\geq (\ell -1)k+1$ for all\ $k\in \mathbb{N}%
.$ When $\ell =2$, equality always holds.
\end{theorem}

For general $\ell \geq 3$ and $k\geq 3$, it was left open whether this lower
bound can be attained. Theorem \ref{thm:KWZ2015} is not true without the
assumption of recurrence. In fact, 
\begin{equation}
x:=10^{3}10^{3^{2}}10^{3^{3}}1\cdots \in \{0,1\}^{\mathbb{N}_{0}}
\label{ex:alpha-3}
\end{equation}%
is aperiodic and $p_{x}^{\ast \mathrm{ab}}(3)=3.$

\medskip \noindent\textbf{Main results.} Our first result provides a lower
bound for the Abelian maximal pattern complexity of recurrent aperiodic
words.

\begin{maintheorem}
\label{thm:main-lower-bound}Let $\alpha \in \mathbb{A}^{\mathbb{N}_{0}}$ be
a recurrent aperiodic word with $\sharp \mathbb{A}_{\alpha }=\ell \geq 2$.
Then $p_{\alpha }^{\ast \mathrm{ab}}(k)\geq k+\ell -1$ for all $k\in \mathbb{%
N}.$
\end{maintheorem}

For $\ell =2$, Theorem \ref{thm:main-lower-bound} reduces to the binary case
of Theorem \ref{thm:KWZ2015}, since aperiodicity and aperiodicity by
projection are equivalent in the binary case. Moreover, the word in (\ref%
{ex:alpha-3}) also shows that the recurrence assumption in Theorem \ref%
{thm:main-lower-bound} is essential.

We give a structural characterization of the recurrent aperiodic words
attaining equality in Theorem \ref{thm:main-lower-bound}.

\begin{maintheorem}
\label{thm:main-interleaving}Let $\alpha $ be a recurrent aperiodic word
with $\sharp \mathbb{A}_{\alpha }=\ell \geq 2.$ Then $p_{\alpha }^{\ast 
\mathrm{ab}}(k)=k+\ell -1$ for all $k\in \mathbb{N}$ if and only if $\alpha $
is a separated $(\ell -1)$-interleaving of a recurrent aperiodic binary word.
\end{maintheorem}

Without assuming recurrence, we obtain the following analogue of the
Morse--Hedlund theorem.

\begin{maintheorem}
\label{thm:threshold} Let $\alpha \in \mathbb{A}^{\mathbb{N}_{0}}$. Then the
following conditions are equivalent:

\begin{enumerate}
\item $\alpha $ is eventually periodic;

\item $\sup_{k\in \mathbb{N}}p_{\alpha }^{\ast \mathrm{ab}}(k)<\infty $;

\item $\binom{p_{\alpha }^{\ast \mathrm{ab}}(k)}{2}<k$ for some $k\in 
\mathbb{N}$.
\end{enumerate}
\end{maintheorem}

Define $\mathcal{L}_{\ell }^{\mathrm{ab}}(k):=\inf_{\alpha \in \mathcal{A}%
_{\ell }}p_{\alpha }^{\ast \mathrm{ab}}(k).$ Theorem \ref{thm:threshold} and
Corollary\ \ref{cor:sparse-upper} give%
\begin{equation*}
\mathcal{L}_{2}^{\mathrm{ab}}(k)=\min \left\{ m\in \mathbb{N}:\binom{m}{2}%
\geq k\right\} \text{ for every }k\in \mathbb{N}.
\end{equation*}

An infinite word $\alpha $ is called an \emph{Abelian} \emph{pattern
Sturmian word} if $p_{\alpha }^{\ast \mathrm{ab}}(k)=\mathcal{L}_{2}^{%
\mathrm{ab}}(k)$ for all $k\in \mathbb{N}$.

For an infinite set $D\subset \mathbb{N}_{0}$, let $\mathbf{1}_{D}\in
\{0,1\}^{\mathbb{N}_{0}}$ denote its characteristic word, where $(\mathbf{1}%
_{D})_{n}=1$ if and only if $n\in D$. Define a simple bipartite graph $%
\Gamma _{D}$ with vertex classes $\mathbb{N}_{0}\times \{\mathrm{L}\}$ and $%
\mathbb{N}_{0}\times \{\mathrm{R}\}$ by 
\begin{equation}
\{(r,\mathrm{L}),(s,\mathrm{R})\}\in E(\Gamma _{D})\Longleftrightarrow
r+s\in D.  \label{eq:graph}
\end{equation}%
The graph-theoretic notation and conventions used below are recalled in
Section \ref{sec:def-preliminary}.

\begin{maintheorem}
\label{thm:main-APS} An infinite word $\alpha $ is Abelian pattern Sturmian
if and only if it uses exactly two letters and, after identifying them with $%
0$ and $1$, one has $\alpha =\mathbf{1}_{D}$\ for an infinite set $D\subset 
\mathbb{N}_{0}$ such that $\Gamma _{D}$ is a forest$.$
\end{maintheorem}

For general alphabets, we obtain the following lower bound.

\begin{maintheorem}
\label{thm:finite}Let $\ell \geq 2$ and $\alpha \in \mathcal{A}_{\ell }$.
Then for every $m,k\in \mathbb{N}$, 
\begin{equation*}
k\geq \frac{m(m-1)}{2(\ell -1)}+\ell -2\Longrightarrow p_{\alpha }^{\ast 
\mathrm{ab}}(k)\geq m.
\end{equation*}
\end{maintheorem}

Together with an upper bound obtained in Subsection \ref{subsec:square-root}%
, Theorem \ref{thm:finite} implies the following limit 
\begin{equation*}
\lim_{k\rightarrow \infty }\frac{\mathcal{L}_{\ell }^{\mathrm{ab}}(k)}{\sqrt{%
k}}=\sqrt{2(\ell -1)}.
\end{equation*}

The paper is organized as follows. Section \ref{sec:def-preliminary}
collects the definitions and preliminary results used throughout. Sections %
\ref{sec:lower-bound-RAW} and \ref{sec:stru-RAW} treat recurrent aperiodic
words: Section \ref{sec:lower-bound-RAW} establishes the lower bound in
Theorem \ref{thm:main-lower-bound}, while Section \ref{sec:stru-RAW}
characterizes the minimal case in Theorem \ref{thm:main-interleaving}.
Section \ref{sec:basic-lower-AW} proves the periodicity criterion in Theorem %
\ref{thm:threshold}. Section \ref{sec:APSW} characterizes Abelian pattern
Sturmian words in terms of incidence graphs. Finally, Section \ref%
{sec:general-lower-AW} proves Theorem \ref{thm:finite} and the extremal
square-root law.

\section{Definitions and preliminaries}

\label{sec:def-preliminary}

We use the following standard graph-theoretic terminology. For a graph $G,$
write $E(G)$ for its edge set. An undirected graph is \emph{simple} if it
has no loops or multiple edges. A \emph{path} is a finite sequence of
vertices in which consecutive vertices are joined by an edge. A \emph{cycle}
is a path of length at least $3$ whose initial and terminal vertices
coincide and whose remaining vertices are distinct. A graph is \emph{%
connected} if any two of its vertices can be joined by a path, and a \emph{%
connected component} is a maximal connected subgraph. A \emph{bipartite graph%
} has two disjoint vertex classes, with every edge joining a vertex in one
class to a vertex in the other. The \emph{degree} of a vertex is the number
of edges incident to it. A \emph{tree} is a connected undirected graph
without cycles, and a \emph{forest} is an undirected graph without cycles. A 
\emph{spanning tree} of a connected graph is a tree containing all of its
vertices, and a \emph{leaf} is a vertex of degree one.

For the bipartite graph $\Gamma _{D}$ defined by (\ref{eq:graph}), the left
and right vertex classes are regarded as disjoint copies of $\mathbb{N}_{0}$%
. For brevity, we write $(r,s)\in E(\Gamma _{D})$ to mean $\{(r,\mathrm{L}%
),(s,\mathrm{R})\}\in E(\Gamma _{D}).$ For finite $R,S\subset \mathbb{N}_{0}$%
, let $\Gamma _{D}[R,S]$ denote the bipartite subgraph induced by the
corresponding left vertices indexed by $R$ and right vertices indexed by $S$%
. Since every cycle is finite, $\Gamma _{D}$ is a forest if and only if $%
\Gamma _{D}[R,S]$ is a forest for every finite $R,S\subset \mathbb{N}_{0}$.

Throughout this paper, we endow $\mathbb{A}$ with the discrete topology and $%
\mathbb{A}^{\mathbb{N}_{0}}$ with the product topology. We write $a^{\infty
}=aaa\cdots $ for the constant word with value $a$. For a finite word $u\in 
\mathbb{A}^{\ast }$, let $\left\vert u\right\vert $ denote its length. For $%
a\in \mathbb{A}$, let $\mathbf{e}_{a}$ denote the standard basis vector of $%
\mathbb{R}^{\mathbb{A}}$ corresponding to $a$. For a set $\mathcal{U}\subset 
\mathbb{A}^{\ast }$ of finite words, write $\Psi (\mathcal{U}):=\{\Psi
(u):u\in \mathcal{U}\}.$

For an infinite set $\mathcal{I}=\{i_{0}<i_{1}<\cdots \}\subset \mathbb{N}%
_{0}\ $and $\beta \in \mathbb{A}^{\mathbb{N}_{0}},$ we define $\beta \lbrack 
\mathcal{I}]\in \mathbb{A}^{\mathbb{N}_{0}}$ by $(\beta \lbrack \mathcal{I}%
])_{n}:=\beta _{i_{n}}$ $(n\in \mathbb{N}_{0}).$ For a finite or infinite
set $\mathcal{I}\subset \mathbb{N}_{0}$ and $\Omega \subset \mathbb{A}^{%
\mathbb{N}_{0}},$ we define $\Omega \lbrack \mathcal{I}]:=\{\beta \lbrack 
\mathcal{I}]:\beta \in \Omega \}.$

Let $\mathsf{T}:\mathbb{A}^{\mathbb{N}_{0}}\rightarrow \mathbb{A}^{\mathbb{N}%
_{0}}$ be the \emph{shift map,} defined by $(\mathsf{T}\alpha )_{n}=\alpha
_{n+1}$ for $n\in \mathbb{N}_{0}$. The \emph{orbit closure} of $\alpha $ is 
\begin{equation*}
\mathcal{O}(\alpha ):=\overline{\{\mathsf{T}^{n}\alpha :n\in \mathbb{N}_{0}\}%
}.
\end{equation*}

If $P\subset \mathbb{N}_{0}$ is finite, then 
\begin{equation}
\mathcal{O}(\alpha )[P]=\mathcal{F}_{\alpha }(P)=\{\alpha \lbrack n+P]:n\in 
\mathbb{N}_{0}\}.  \label{eq:rel}
\end{equation}%
Indeed, restriction to the coordinates in $P$ is continuous and takes values
in a finite discrete space, so passing to the orbit closure produces no
additional $P$-factors.

For simplicity, we write $\pi _{a}(\alpha )$ for $\pi _{\{a\}}(\alpha )$
when $a$ is a letter. More generally, a map of alphabets $\varphi :\mathbb{A}%
\rightarrow \mathbb{B}$ is understood to act coordinatewise on both finite
and infinite words. Thus, for $u=u_{0}\cdots u_{r-1}\in \mathbb{A}^{\ast },$ 
$\varphi (u):=\varphi (u_{0})\cdots \varphi (u_{r-1}).$

\begin{lemma}
\label{lem:projection} If $\alpha \in \mathbb{A}^{\mathbb{N}_{0}}$ is
aperiodic, then there exists $a\in \mathbb{A}$ such that $\pi _{a}(\alpha )$
is aperiodic. If $\alpha $ is recurrent, then $\pi _{a}(\alpha )$ is
recurrent for every $a\in \mathbb{A}$.
\end{lemma}

\begin{proof}
Suppose that $\pi _{a}(\alpha )$ is eventually periodic for every $a\in 
\mathbb{A}$. Then for each $a\in \mathbb{A}$, there exist $q_{a}\geq 1$ and $%
N_{a}\geq 0$ such that $\mathbf{1}_{\{a\}}(\alpha _{n+q_{a}})=\mathbf{1}%
_{\{a\}}(\alpha _{n})$ for all $n\geq N_{a}.$ Let $q:=\func{lcm}_{a\in 
\mathbb{A}}q_{a}$ and $N:=\max_{a\in \mathbb{A}}N_{a}.$ It follows that $%
\mathbf{1}_{\{a\}}(\alpha _{n+q})=\mathbf{1}_{\{a\}}(\alpha _{n})$ for every 
$a\in \mathbb{A}$ and every $n\geq N$. Hence $\alpha _{n+q}=\alpha _{n}$ for
all $n\geq N,$ which contradicts the aperiodicity of $\alpha $.

Now assume that $\alpha $ is recurrent and fix $a\in \mathbb{A}$. Let $v$ be
a factor of $\pi _{a}(\alpha )$. Then there exists a factor $w$ of $\alpha $
such that $v=\pi _{a}(w).$ Since $w$ occurs infinitely often in $\alpha $,
the factor $v$ occurs infinitely often in $\pi _{a}(\alpha )$. Thus $\pi
_{a}(\alpha )$ is recurrent.
\end{proof}

\begin{lemma}
\label{lem:transfer}Let $\alpha \in \mathbb{A}^{\mathbb{N}_{0}}.$ For every $%
k\in \mathbb{N}$, the following statements hold.

\begin{enumerate}
\item If $y\in \mathcal{O}(\alpha )$, then $p_{y}^{\ast \mathrm{ab}}(k)\leq
p_{\alpha }^{\ast \mathrm{ab}}(k)$.

\item For $N\geq 0$, $p_{\mathsf{T}^{N}\alpha }^{\ast \mathrm{ab}}(k)\leq
p_{\alpha }^{\ast \mathrm{ab}}(k)\leq p_{\mathsf{T}^{N}\alpha }^{\ast 
\mathrm{ab}}(k)+N.$

\item If $\varphi :\mathbb{A}\rightarrow \mathbb{B}$ is a map of alphabets,
then $p_{\varphi (\alpha )}^{\ast \mathrm{ab}}(k)\leq p_{\alpha }^{\ast 
\mathrm{ab}}(k)$.

\item If two words agree at all positions $n\geq N$, their Abelian maximal
pattern complexities differ by at most $N$ for every $k$.
\end{enumerate}
\end{lemma}

\begin{proof}
For (1), (\ref{eq:rel}) shows directly that every finite pattern occurring
in $y$ also occurs in $\alpha $. Hence $p_{y}^{\mathrm{ab}}(S)\leq p_{\alpha
}^{\mathrm{ab}}(S)$ for every $S\in \Sigma _{k}(\mathbb{N}_{0})$, and the
claim follows. For (2), if $n\geq N$, then $\alpha \lbrack n+S]=(\mathsf{T}%
^{N}\alpha )[n-N+S].$ Hence the $S$-factors of $\alpha $ based at positions $%
n\geq N$ are precisely the $S$-factors of $\mathsf{T}^{N}\alpha $. The
remaining $N$ starting positions $0,\ldots ,N-1$ contribute at most $N$
additional Parikh vectors. Therefore one has $p_{\mathsf{T}^{N}\alpha }^{%
\mathrm{ab}}(S)\leq p_{\alpha }^{\mathrm{ab}}(S)\leq p_{\mathsf{T}^{N}\alpha
}^{\mathrm{ab}}(S)+N,$ and taking the supremum over all $k$-patterns proves
(2).\ For (3), define $Q_{\varphi }:\mathbb{R}^{\mathbb{A}}\rightarrow 
\mathbb{R}^{\mathbb{B}}$ by $Q_{\varphi }(\mathbf{e}_{a}):=\mathbf{e}%
_{\varphi (a)}.$ Then, for every $S\in \Sigma _{k}(\mathbb{N}_{0})$ and $%
n\in \mathbb{N}_{0}$, we have $\Psi (\varphi (\alpha )[n+S])=Q_{\varphi
}(\Psi (\alpha \lbrack n+S])).$ Hence $p_{\varphi (\alpha )}^{\mathrm{ab}%
}(S)\leq p_{\alpha }^{\mathrm{ab}}(S),$ and taking the supremum over all $k$%
-patterns proves (3). For (4), if $\alpha _{n}=\beta _{n}$ for all $n\geq N$%
, then $\mathsf{T}^{N}\alpha =\mathsf{T}^{N}\beta .$ Applying (2) to both
words gives $p_{\mathsf{T}^{N}\alpha }^{\ast \mathrm{ab}}(k)\leq p_{\alpha
}^{\ast \mathrm{ab}}(k)\leq p_{\mathsf{T}^{N}\alpha }^{\ast \mathrm{ab}%
}(k)+N $ and the same inequalities for $\beta $. Therefore one has $%
|p_{\alpha }^{\ast \mathrm{ab}}(k)-p_{\beta }^{\ast \mathrm{ab}}(k)|\leq N.$
\end{proof}

A set $Z\subset \mathbb{A}^{\mathbb{N}_{0}}$ is called \emph{shift-invariant}
if $\mathsf{T}(Z)\subset Z$. A nonempty closed shift-invariant set is \emph{%
minimal} if it contains no proper nonempty closed shift-invariant subset. A
word $\alpha \in \mathbb{A}^{\mathbb{N}_{0}}$ is \emph{uniformly recurrent}
if, for every nonempty finite factor $u$ of $\alpha $, there exists an
integer $L\geq |u|$ such that every factor of $\alpha $ of length $L$
contains $u$.

We shall use the following standard facts about minimal dynamical systems;
see, for example, \cite{Bruin2022}.

\begin{lemma}
\label{lem:minimal}Every nonempty closed shift-invariant set contains a
minimal closed shift-invariant subset. Each point of a minimal subset is
uniformly recurrent, and hence recurrent. A finite minimal set is a periodic
orbit, while no point of an infinite minimal set is eventually periodic.
\end{lemma}

\begin{lemma}
\label{lem:noperiodic}Let $\alpha \in \mathbb{A}^{\mathbb{N}_{0}}.$ If $%
\mathcal{O}(\alpha )$ contains no periodic point, then $p_{\alpha }^{\ast 
\mathrm{ab}}(k)\geq k+1$ for every $k\in \mathbb{N}$.
\end{lemma}

\begin{proof}
By Lemma \ref{lem:minimal}, $\mathcal{O}(\alpha )$\ contains a minimal
closed shift-invariant subset $M.$ Let $y\in M.$ Then $M$ is infinite since $%
\mathcal{O}(\alpha )$ contains no periodic point. Hence, again by Lemma \ref%
{lem:minimal}, $y$ is recurrent and aperiodic. By Lemma \ref{lem:projection}%
, $\pi _{a}(y)$ is recurrent and aperiodic for some $a\in \mathbb{A}$. Hence
Theorem \ref{thm:KWZ2015} and Lemma \ref{lem:transfer}(3),(1) give 
\begin{equation*}
k+1=p_{\pi _{a}(y)}^{\ast \mathrm{ab}}(k)\leq p_{y}^{\ast \mathrm{ab}%
}(k)\leq p_{\alpha }^{\ast \mathrm{ab}}(k),
\end{equation*}%
which proves the conclusion.
\end{proof}

For $x\in \{0,1\}^{\mathbb{N}_{0}}$, $S\in \Sigma _{k}(\mathbb{N}_{0})$ and $%
n\in \mathbb{N}_{0},$ define%
\begin{equation*}
c_{x,S}(n):=\sum_{s\in S}x_{n+s}\text{, }\mathcal{C}_{x}(S)=\{c_{x,S}(n):n%
\geq 0\}.
\end{equation*}%
When $x$ is fixed, we write simply $c_{S}(n).$ This is the number of $1$'s
in the $S$-factor $x[n+S],$ hence its Parikh vector is $\Psi
(x[n+S])=(k-c_{S}(n))\mathbf{e}_{0}+c_{S}(n)\mathbf{e}_{1}.$ Thus the
Abelian class of $x[n+S]$ is uniquely determined by the integer $c_{S}(n).$
Consequently,%
\begin{equation}
p_{x}^{\mathrm{ab}}(S)=\sharp \mathcal{C}_{x}(S).
\label{eq:binary-abelian-count}
\end{equation}%
Since a binary word of length $k$ has at most $k+1$ distinct Parikh vectors,
one has 
\begin{equation}
p_{x}^{\ast \mathrm{ab}}(k)\leq k+1\ \text{for every }k\in \mathbb{N}.
\label{eq:binary-upper-bound}
\end{equation}

\begin{lemma}
\label{lem:zero-blocks} Let $x\in \{0,1\}^{\mathbb{N}_{0}}$ have infinitely
many occurrences of $1$, and suppose that $0^{\infty }\in \mathcal{O}(x)$.
Then the following statements hold.

\begin{enumerate}
\item For every $L\geq 1$ and $N\geq 0$, there exists $Y>\max \{N,L\}$ such
that 
\begin{equation*}
x_{Y-L}\cdots x_{Y-1}=0^{L},\text{\ }x_{Y}=1.
\end{equation*}

\item For every nonempty finite $S\subset \mathbb{N}_{0}$ and every $N\in 
\mathbb{N}_{0}$, there exist $n_{0},n_{1}>N$ such that $c_{S}(n_{0})=0$ and $%
c_{S}(n_{1})=1.$

\item In particular, for every nonempty finite $R\subset \mathbb{N}_{0}$,
there exist arbitrarily large $p$ such that $x[p+R]=0^{\sharp R}.$
\end{enumerate}
\end{lemma}

\begin{proof}
Since $0^{\infty }\in \mathcal{O}(x)$, the word $x$ contains arbitrarily
long zero blocks. Moreover, such blocks occur arbitrarily far to the right.
Indeed, otherwise their left endpoints would belong to a fixed finite set.
By the pigeonhole principle, one of these positions would be the left
endpoint of arbitrarily long zero blocks, which contradicts the assumption
that $1$ occurs infinitely often.

For (1), take a zero block of length at least $L$ whose left endpoint is
larger than $N$. Since $1$ occurs infinitely often, there is a first
occurrence $Y$ of $1$ following this block. Then $Y>\max \{N,L\}$ and $%
x_{Y-t}=0$ for all $1\leq t\leq L$.

For (2), fix a nonempty finite $S\subset \mathbb{N}_{0}$ and $N\in \mathbb{N}%
_{0}$. Let $L:=\max S-\min S+1$. By (1), there exists $Y>\max \{N+\max S,L\}$
such that $x_{Y-t}=0$ $(1\leq t\leq L)$ and $x_{Y}=1.$ Set $n_{1}:=Y-\max
S>N.$ Then $x_{n_{1}+\max S}=x_{Y}=1,$ while $x_{n_{1}+s}=x_{Y-(\max S-s)}=0$
for every $s\in S$ with $s<\max S.$ Hence $c_{S}(n_{1})=1.$

To obtain the value $0$, take a zero block $x_{a}x_{a+1}\cdots
x_{a+L-1}=0^{L}$ with $a>N+\min S.$ Set $n_{0}:=a-\min S>N.$ Then $%
n_{0}+S\subset \lbrack a,a+L-1],$ and hence $c_{S}(n_{0})=0.$

Finally, (3) follows from the zero-count conclusion of (2), with $S=R.$
\end{proof}

We shall use the following Ramsey theorem. Its statement is restricted to
the forms needed here.

\begin{theorem}[\protect\cite{Ramsey}]
\label{thm:Ramsey} If the two-element subsets of an infinite subset of $%
\mathbb{N}_{0}$ are assigned finitely many colors, then there is an infinite
subset all of whose two-element subsets have the same color.
\end{theorem}

\section{The lower bound for recurrent aperiodic words}

\label{sec:lower-bound-RAW}

In this section, we shall prove Theorem \ref{thm:main-lower-bound}, which
gives the lower bound for the Abelian maximal pattern complexity of
recurrent aperiodic words.

\subsection{Superstationarity and tail-freezing}

\ 

We recall the following three selection lemmas from \cite{KWZ2015}, which
will be used below. We also recall the notion of superstationarity from \cite%
{Kamae2011}.

\begin{lemma}[\protect\cite{KWZ2015}]
\label{lem:tail-selection}Let $\alpha \in \mathbb{A}^{\mathbb{N}_{0}}$ be
recurrent. Then there exists an infinite subset $\mathcal{N}%
=\{N_{0}<N_{1}<N_{2}<\cdots \}\subset \mathbb{N}_{0}$ such that for every $%
t,q\in \mathbb{N}_{0}$, 
\begin{equation}
\alpha _{t+N_{0}}\alpha _{t+N_{1}}\cdots \alpha _{t+N_{q-1}}\alpha
_{t+N_{q}}^{\infty }\in \mathcal{O}(\alpha )[\mathcal{N}].
\label{eq:tail-selection}
\end{equation}
\end{lemma}

\begin{lemma}[\protect\cite{KWZ2015}]
\label{lem:subsequence-tail}Let $\mathcal{N}=\{N_{0}<N_{1}<N_{2}<\cdots
\}\subset \mathbb{N}_{0}$ satisfy condition (\ref{eq:tail-selection}), and
let $\mathcal{N}^{\prime }$ be any infinite subset of $\mathcal{N}$. Then $%
\mathcal{N}^{\prime }$ also satisfies condition (\ref{eq:tail-selection}).
\end{lemma}

Let $\Omega \subset \mathbb{A}^{\mathbb{N}_{0}}$ be nonempty and let $j\geq
1 $. We say that $\Omega $ is $j$-\emph{superstationary }if $\Omega \lbrack
S]=\Omega \lbrack S^{\prime }]$ for every $S,S^{\prime }\in \Sigma _{j}(%
\mathbb{N}_{0})$.

\begin{lemma}[\protect\cite{KWZ2015}]
\label{lem:superstationary}Let $\Omega \subset \mathbb{A}^{\mathbb{N}_{0}}$
be nonempty and let $\mathcal{N}\subset \mathbb{N}_{0}$ be an infinite
subset. Then for every positive integer $k,$ there exists an infinite subset 
$\mathcal{N}^{\prime }$ of $\mathcal{N}$ (depending on $k$) such that $%
\Omega \lbrack \mathcal{N}^{\prime }]$ is $j$-superstationary for every $%
1\leq j\leq k$.
\end{lemma}

\begin{lemma}
\label{lem:tail-closure}Let $\alpha \in \mathbb{A}^{\mathbb{N}_{0}}$, and
let $\mathcal{N}=\{N_{0}<N_{1}<N_{2}<\cdots \}\subset \mathbb{N}_{0}$ be an
infinite set satisfying (\ref{eq:tail-selection}). Write $\Omega :=\mathcal{O%
}(\alpha )[\mathcal{N}].$ Then $\Omega $ has the tail-freezing property,
i.e., $\omega _{0}\omega _{1}\cdots \omega _{q-1}\omega _{q}^{\infty }\in
\Omega $ for every $\omega \in \Omega $ and every $q\in \mathbb{N}_{0}$.
\end{lemma}

\begin{proof}
Let $\omega \in \Omega $ and $q\geq 0$. Since $\omega _{0}\cdots \omega
_{q}\in \mathcal{O}(\alpha )[\{N_{0},\ldots ,N_{q}\}],$ it follows from (\ref%
{eq:rel}) that there exists$\ t\geq 0$ such that $\alpha _{t+N_{i}}=\omega
_{i}$\ $(0\leq i\leq q).$ Applying (\ref{eq:tail-selection}) with this $t$
and $q$, we obtain%
\begin{equation*}
\alpha _{t+N_{0}}\alpha _{t+N_{1}}\cdots \alpha _{t+N_{q-1}}\alpha
_{t+N_{q}}^{\infty }\in \mathcal{O}(\alpha )[\mathcal{N}]=\Omega .
\end{equation*}%
By the choice of $t$, this word is $\omega _{0}\omega _{1}\cdots \omega
_{q-1}\omega _{q}^{\infty },$ and hence $\Omega $ has the tail-freezing
property.
\end{proof}

\begin{lemma}
\label{lem:tail-hered}Let $\Omega \subset \mathbb{A}^{\mathbb{N}_{0}}$ be
nonempty and let $\mathcal{N}=\{N_{0}<N_{1}<N_{2}<\cdots \}\subset \mathbb{N}%
_{0}$ be infinite. If $\Omega $ is $k$-superstationary, then $\Omega \lbrack 
\mathcal{N}]$ is $k$-superstationary. Moreover, if $\Omega $ has the
tail-freezing property, then so does $\Omega \lbrack \mathcal{N}]$.
\end{lemma}

\begin{proof}
Suppose first that $\Omega $ is $k$-superstationary. Let $S=\{s_{1}<\cdots
<s_{k}\},S^{\prime }=\{s_{1}^{\prime }<\cdots <s_{k}^{\prime }\}\in \Sigma
_{k}(\mathbb{N}_{0}).$ Define 
\begin{equation*}
\mathcal{N}(S):=\{N_{s_{1}},\ldots ,N_{s_{k}}\}\text{, }\mathcal{N}%
(S^{\prime }):=\{N_{s_{1}^{\prime }},\ldots ,N_{s_{k}^{\prime }}\}.
\end{equation*}%
Then $\Omega \lbrack \mathcal{N}][S]=\Omega \lbrack \mathcal{N}(S)]$, $%
\Omega \lbrack \mathcal{N}][S^{\prime }]=\Omega \lbrack \mathcal{N}%
(S^{\prime })].$ Since $\Omega $ is $k$-superstationary and $\sharp \mathcal{%
N}(S)=\sharp \mathcal{N}(S^{\prime })=k,$ one has $\Omega \lbrack \mathcal{N}%
(S)]=\Omega \lbrack \mathcal{N}(S^{\prime })].$ Then we have $\Omega \lbrack 
\mathcal{N}][S]=\Omega \lbrack \mathcal{N}][S^{\prime }],$ and hence $\Omega
\lbrack \mathcal{N}]$ is $k$-superstationary.

Now suppose that $\Omega $ has the tail-freezing property. Let $\xi \in
\Omega \lbrack \mathcal{N}]$ and $q\in \mathbb{N}_{0}$. Take $\omega \in
\Omega $ such that $\xi =\omega \lbrack \mathcal{N}].$ By the tail-freezing
property of $\Omega $, we have%
\begin{equation*}
\widetilde{\omega }:=\omega _{0}\omega _{1}\cdots \omega _{N_{q}-1}\omega
_{N_{q}}^{\infty }\in \Omega .
\end{equation*}%
For $i<q$, $\widetilde{\omega }_{N_{i}}=\omega _{N_{i}}=\xi _{i},$ while for 
$i\geq q$, $\widetilde{\omega }_{N_{i}}=\omega _{N_{q}}=\xi _{q}.$ Therefore
one has%
\begin{equation*}
\xi _{0}\xi _{1}\cdots \xi _{q-1}\xi _{q}^{\infty }=\widetilde{\omega }[%
\mathcal{N}]\in \Omega \lbrack \mathcal{N}].
\end{equation*}%
Thus $\Omega \lbrack \mathcal{N}]$ has the tail-freezing property.
\end{proof}

\begin{proposition}
\label{prop:simultaneous-selection} Let $\alpha $ be recurrent and let $k\in 
\mathbb{N}$. Then there exists an infinite set $\mathcal{N}\subset \mathbb{N}%
_{0}$ such that $\mathcal{O}(\alpha )[\mathcal{N}]$ is $j$-superstationary
for every $1\leq j\leq k$ and has the tail-freezing property.
\end{proposition}

\begin{proof}
By Lemma \ref{lem:tail-selection}, there exists an infinite subset $\mathcal{%
N}_{0}\subset \mathbb{N}_{0}$ satisfying (\ref{eq:tail-selection}). Applying
Lemma \ref{lem:superstationary} to $\mathcal{O}(\alpha )$ and $\mathcal{N}%
_{0}$, we obtain an infinite subset $\mathcal{N}\subset \mathcal{N}_{0}$
such that $\mathcal{O}(\alpha )[\mathcal{N}]$ is $j$-superstationary for
every $1\leq j\leq k$.

By Lemma \ref{lem:subsequence-tail}, the subset $\mathcal{N}$ still
satisfies condition (\ref{eq:tail-selection}). Hence Lemma \ref%
{lem:tail-closure} implies that $\mathcal{O}(\alpha )[\mathcal{N}]$ has the
tail-freezing property.
\end{proof}

\begin{lemma}
\label{lem:projection-preservation}Let $\mathbb{A}$ and $\mathbb{B}$ be
finite alphabets, and let $\varphi :\mathbb{A}\longrightarrow \mathbb{B}$ be
a map. Then for every $\alpha \in \mathbb{A}^{\mathbb{N}_{0}}$ and every
infinite set $\mathcal{N}\subset \mathbb{N}_{0}$, one has 
\begin{equation*}
\varphi (\mathcal{O}(\alpha )[\mathcal{N}])=\mathcal{O}(\varphi (\alpha ))[%
\mathcal{N}].
\end{equation*}%
Moreover, if $\Omega \subset \mathbb{A}^{\mathbb{N}_{0}}$ is $k$%
-superstationary, then $\varphi (\Omega )$ is $k$-superstationary; if $%
\Omega $ has the tail-freezing property, then so does $\varphi (\Omega )$.
\end{lemma}

\begin{proof}
Since the coordinatewise extension of $\varphi $ is continuous and satisfies 
$\varphi \circ \mathsf{T}=\mathsf{T}\circ \varphi ,$ one has $\varphi (%
\mathcal{O}(\alpha ))\subset \mathcal{O}(\varphi (\alpha )).$ Conversely, $%
\mathcal{O}(\alpha )$ is compact, so $\varphi (\mathcal{O}(\alpha ))$ is
compact and therefore closed. Moreover, the set $\varphi (\mathcal{O}(\alpha
))$ contains 
\begin{equation*}
\{\varphi (\mathsf{T}^{n}\alpha ):n\in \mathbb{N}_{0}\}=\{\mathsf{T}%
^{n}\varphi (\alpha ):n\in \mathbb{N}_{0}\}.
\end{equation*}%
Thus $\mathcal{O}(\varphi (\alpha ))\subset \varphi (\mathcal{O}(\alpha )),$
and consequently $\varphi (\mathcal{O}(\alpha ))=\mathcal{O}(\varphi (\alpha
)).$ Since $\varphi $ commutes with restriction to $\mathcal{N}$, we have $%
\varphi (\mathcal{O}(\alpha )[\mathcal{N}])=\mathcal{O}(\varphi (\alpha ))[%
\mathcal{N}].$

Now suppose that $\Omega $ is $k$-superstationary. Then for $P,Q\in \Sigma
_{k}(\mathbb{N}_{0})$, one has $\Omega \lbrack P]=\Omega \lbrack Q].$ Since
coordinatewise application of $\varphi $ commutes with restriction, it
follows that 
\begin{equation*}
\varphi (\Omega )[P]=\varphi (\Omega \lbrack P])=\varphi (\Omega \lbrack
Q])=\varphi (\Omega )[Q].
\end{equation*}
Hence $\varphi (\Omega )$ is $k$-superstationary.

Finally, suppose that $\Omega $ has the tail-freezing property. Let $\xi \in
\varphi (\Omega )$ and $q\in \mathbb{N}_{0}$. Choose $\omega \in \Omega $
such that $\xi =\varphi (\omega ).$ Then $\omega _{0}\omega _{1}\cdots
\omega _{q-1}\omega _{q}^{\infty }\in \Omega .$ Applying $\varphi $
coordinatewise gives $\xi _{0}\xi _{1}\cdots \xi _{q-1}\xi _{q}^{\infty }\in
\varphi (\Omega ).$ Therefore $\varphi (\Omega )$ has the tail-freezing
property.
\end{proof}

\subsection{Proof of Theorem \protect\ref{thm:main-lower-bound}}

\begin{lemma}
\label{lem:binary-two-coordinate} Let $x\in \{0,1\}^{\mathbb{N}_{0}}$ be
aperiodic, and let $\mathcal{N}=\{N_{0}<N_{1}<N_{2}<\cdots \}\subset \mathbb{%
N}_{0}$ be infinite. Set $\Omega :=\mathcal{O}(x)[\mathcal{N}].$ Then for
every $i,j\in \mathbb{N}_{0}$ with $i<j$, $01,10\in \Omega \lbrack \{i,j\}].$
\end{lemma}

\begin{proof}
Fix $i<j$ and let $d=N_{j}-N_{i}>0$. Suppose first that $10\notin \Omega
\lbrack \{i,j\}]$. Then (\ref{eq:rel}) implies that $%
(x_{n+N_{i}},x_{n+N_{j}})\neq (1,0)$\ for all $n\in \mathbb{N}_{0},$
equivalently, 
\begin{equation}
x_{t}=1\Longrightarrow x_{t+d}=1\text{\ }(t\geq N_{i}).  \label{eq:no10}
\end{equation}%
For each $h\in \{0,\ldots ,d-1\}$, consider the residue word $%
(x_{h+md})_{m\in \mathbb{N}_{0}}$ from the first index $h+md\geq N_{i}$
onward. By (\ref{eq:no10}), this word is eventually constant. Since there
are only finitely many residue classes modulo $d$, there exists $t_{0}\in 
\mathbb{N}_{0}$ such that $x_{t+d}=x_{t}$ for all $t\geq t_{0}.$ Thus $x$ is
eventually periodic, a contradiction. Therefore $10\in \Omega \lbrack
\{i,j\}].$

The same argument, with $0$ and $1$ interchanged, gives $01\in \Omega
\lbrack \{i,j\}].$
\end{proof}

\begin{lemma}
\label{lem:binary}Let $x\in \{0,1\}^{\mathbb{N}_{0}}$ be aperiodic, and let $%
\mathcal{N}\subset \mathbb{N}_{0}$ be infinite. Write $\Omega :=\mathcal{O}%
(x)[\mathcal{N}].$ Suppose that $k\in \mathbb{N},$ $\Omega $ is $k$%
-superstationary and has the tail-freezing property. Then $\mathcal{\sharp }%
\Psi (\Omega \lbrack \{0,1,\ldots ,k-1\}])=k+1.$
\end{lemma}

\begin{proof}
For $k=1$, Lemma \ref{lem:binary-two-coordinate} shows that both letters $0$
and $1$ occur in $\Omega \lbrack \{0\}]$. Hence the assertion follows.

Assume that $k\geq 2.$ By Lemma \ref{lem:binary-two-coordinate}, there
exists $\omega \in \Omega $ such that $\omega \lbrack \{2k-2,2k-1\}]=01$.
Write $\xi :=\omega _{0}\omega _{1}\cdots \omega _{2k-3}.$ Applying the
tail-freezing property to $\omega $ at $q=2k-2$ and $q=2k-1,$ respectively,
we obtain 
\begin{equation}
\xi 0^{\infty }\in \Omega \text{ and }\xi 01^{\infty }\in \Omega .
\label{eq:two-tails}
\end{equation}%
Since $|\xi |=2k-2$, one of the letters in $\{0,1\}$ occurs in $\xi $ at
least $k-1$ times. Let $b\in \{0,1\}$ be such a letter, and choose a set $%
P\subset \{0,\ldots ,2k-3\}$ of cardinality $k-1$ such that $\omega \lbrack
P]=b^{k-1}.$ Then (\ref{eq:two-tails}) gives 
\begin{equation*}
\left\{ 
\begin{array}{cc}
0^{k-1}1\in \Omega \lbrack P\cup \{2k-1\}], & \text{if }b=0, \\ 
1^{k-1}0\in \Omega \lbrack P\cup \{2k-2\}], & \text{if }b=1.%
\end{array}%
\right.
\end{equation*}%
Thus, in either case, there exists a $k$-pattern $Q$ such that $%
b^{k-1}(1-b)\in \Omega \lbrack Q].$ As $\Omega $ is $k$-superstationary, we
have $b^{k-1}(1-b)\in \Omega \lbrack \{0,1,\ldots ,k-1\}].$ Applying the
tail-freezing property gives%
\begin{equation}
b^{k-1}(1-b)^{\infty }\in \Omega \text{ and }b^{\infty }\in \Omega .
\label{eq:step-word}
\end{equation}

For each $s\in \{0,1,\ldots ,k\}$, define $Q_{s}:=\{0,\ldots ,s-1\}\cup
\{k,\ldots ,2k-s-1\},$ with the usual convention that either set is empty
when its lower endpoint exceeds its upper endpoint. Then $\sharp Q_{s}=k.$
For $0\leq s\leq k-1$, the first word in (\ref{eq:step-word}) gives $%
b^{s}(1-b)^{k-s}\in \Omega \lbrack Q_{s}],$ while for $s=k$, the same
conclusion follows from $b^{\infty }\in \Omega $. Since $\Omega $ is $k$%
-superstationary, we have%
\begin{equation*}
b^{s}(1-b)^{k-s}\in \Omega \lbrack \{0,1,\ldots ,k-1\}]\text{ for every }%
0\leq s\leq k.
\end{equation*}%
These $k+1$ words have pairwise distinct Parikh vectors. Therefore 
\begin{equation*}
\mathcal{\sharp }\Psi (\Omega \lbrack \{0,1,\ldots ,k-1\}])\geq k+1.
\end{equation*}%
Since a binary word of length $k$ has at most $k+1$ possible Parikh vectors,
equality follows.
\end{proof}

\begin{proof}[Proof of Theorem \protect\ref{thm:main-lower-bound}]
Let $k\in \mathbb{N}$. By Lemma \ref{lem:projection}, there exists $a\in 
\mathbb{A}_{\alpha }$ such that $\pi _{a}(\alpha )$ is aperiodic. By
Proposition \ref{prop:simultaneous-selection}, there exists an infinite set $%
\mathcal{N}=\{N_{0}<N_{1}<N_{2}<\cdots \}\subset \mathbb{N}_{0}$ such that $%
\Omega :=\mathcal{O}(\alpha )[\mathcal{N}]$ is $k$-superstationary and has
the tail-freezing property. Set $P:=\{N_{0},N_{1},\ldots ,N_{k-1}\}.$

By Lemma \ref{lem:projection-preservation}, 
\begin{equation*}
\pi _{a}(\Omega )=\mathcal{O}(\pi _{a}(\alpha ))[\mathcal{N}]
\end{equation*}%
is also $k$-superstationary and has the tail-freezing property. Since $\pi
_{a}(\alpha )$ is aperiodic, Lemma \ref{lem:binary} implies that, for each $%
0\leq h\leq k$, there exists $u_{h}\in \Omega \lbrack \{0,1,\ldots ,k-1\}]$
such that $(\Psi (u_{h}))_{a}=h.$ Set 
\begin{equation*}
v_{h}:=\Psi (u_{h})\text{ }(0\leq h\leq k),\text{\ }\mathcal{V}%
:=\{v_{0},v_{1},\ldots ,v_{k}\}.
\end{equation*}%
Since the $a$-coordinates of $v_{0},\ldots ,v_{k}$ are $0,1,\ldots ,k$,
respectively, these vectors are pairwise distinct. Hence $\sharp \mathcal{V}%
=k+1.$

For each $c\in \mathbb{A}_{\alpha }$, recurrence of $\alpha $ implies that $%
c $ occurs infinitely often in $\alpha $. Hence there exists $t_{c}\in 
\mathbb{N}_{0}$ such that $\alpha _{t_{c}+N_{0}}=c.$ Therefore the word $(%
\mathsf{T}^{t_{c}}\alpha )[\mathcal{N}]\in \Omega $ has first letter $c$.
Applying the tail-freezing property at $q=0$ gives $c^{\infty }\in \Omega .$
Consequently, one has%
\begin{equation*}
\mathcal{C}:=\{k\mathbf{e}_{c}:c\in \mathbb{A}_{\alpha }\}\subset \Psi
(\Omega \lbrack \{0,1,\ldots ,k-1\}]),
\end{equation*}%
and $\sharp \mathcal{C}=\ell .$

For $1\leq h\leq k-1$, the $a$-coordinate of $v_{h}$ is $h$, while the $a$%
-coordinate of a vector in $\mathcal{C}$ is either $0$ or $k$. Therefore $%
v_{h}\notin \mathcal{C}$\ for all $1\leq h\leq k-1,$ and hence $\sharp (%
\mathcal{V}\cap \mathcal{C})\leq 2.$ It follows that 
\begin{equation*}
\sharp \Psi (\Omega \lbrack \{0,1,\ldots ,k-1\}])\geq \sharp (\mathcal{V}%
\cup \mathcal{C})=\sharp \mathcal{V}+\sharp \mathcal{C}-\sharp (\mathcal{V}%
\cap \mathcal{C})\geq (k+1)+\ell -2=k+\ell -1.
\end{equation*}

Finally, by (\ref{eq:rel}), we obtain that $\Omega \lbrack \{0,1,\ldots
,k-1\}]=\mathcal{O}(\alpha )[P]=\mathcal{F}_{\alpha }(P).$ Therefore one has%
\begin{equation*}
p_{\alpha }^{\ast \mathrm{ab}}(k)\geq p_{\alpha }^{\mathrm{ab}}(P)=\sharp
\Psi (\mathcal{F}_{\alpha }(P))\geq k+\ell -1.
\end{equation*}
\end{proof}

Since aperiodicity is equivalent to aperiodicity by projection in the binary
case, one has the following.

\begin{corollary}[\protect\cite{KWZ2015}]
\label{cor:binary}If $x\in \{0,1\}^{\mathbb{N}_{0}}$ is recurrent and
aperiodic, then $p_{x}^{\ast \mathrm{ab}}(k)=k+1$ for every $k\in \mathbb{N}$%
.
\end{corollary}

\section{Recurrent aperiodic words with minimal Abelian maximal pattern
complexity}

\label{sec:stru-RAW}

In this section, we study recurrent aperiodic words of minimal Abelian
maximal pattern complexity. We first develop the basic properties of
separated interleavings and then prove the structural characterization in
Theorem \ref{thm:main-interleaving}.

\subsection{Separated interleavings}

\ 

\begin{lemma}
\label{lem:interleaving-decomposition}Let $\alpha \in \mathbb{A}^{\mathbb{N}%
_{0}}$ be an $m$-interleaving of a word $x\in \mathbb{A}^{\mathbb{N}_{0}}$,
and let $S\in \Sigma _{k}(\mathbb{N}_{0}).$ For $0\leq r\leq m-1$, let 
\begin{equation*}
S_{r}:=\{s\in S:s\equiv r\text{ }(\func{mod}m)\},\text{\ }k_{r}:=\sharp
S_{r}.
\end{equation*}%
Then 
\begin{equation*}
p_{\alpha }^{\mathrm{ab}}(S)\leq \sharp \{r:k_{r}=0\}+\sum_{\substack{ 0\leq
r\leq m-1  \\ k_{r}>0}}p_{x}^{\ast \mathrm{ab}}(k_{r}).
\end{equation*}
\end{lemma}

\begin{proof}
Let $p\in \{0,\ldots ,m-1\}$ be such that $\alpha ^{(p)}=x,$ and write $%
\alpha ^{(j)}=c_{j}^{\infty }$\ $(j\neq p).$ For $0\leq h\leq m-1$, let 
\begin{equation*}
\Xi _{h}(S):=\{\Psi (\alpha \lbrack n+S]):n\in \mathbb{N}_{0},\ n\equiv h%
\text{ }(\func{mod}m)\}.
\end{equation*}%
Since every $n\in \mathbb{N}_{0}$ belongs to exactly one residue class
modulo $m$, one has%
\begin{equation*}
p_{\alpha }^{\mathrm{ab}}(S)\leq \sum_{h=0}^{m-1}\sharp \Xi _{h}(S).
\end{equation*}

Fix $h\in \{0,\ldots ,m-1\}$, and let $r_{h}\in \{0,\ldots ,m-1\}$ be the
unique integer satisfying $h+r_{h}\equiv p\pmod m.$ For $n=mq+h$ with $q\in 
\mathbb{N}_{0},$ the coordinates in $n+S_{r_{h}}$ lie in the residue $\alpha
^{(p)}=x$, while those in $n+(S\setminus S_{r_{h}})$ lie in constant
residues of $\alpha $.

If $k_{r_{h}}=0$, then $\Psi (\alpha \lbrack n+S])$ is independent of $q$,
and hence 
\begin{equation}
\sharp \Xi _{h}(S)\leq 1.  \label{eq:mcl-P<=1}
\end{equation}

Suppose that $k_{r_{h}}>0$, and define 
\begin{equation*}
L_{r_{h}}:=\left\{ \frac{s-r_{h}}{m}:s\in S_{r_{h}}\right\} .
\end{equation*}%
Then $L_{r_{h}}\in \Sigma _{k_{r_{h}}}(\mathbb{N}_{0}).$ Since $%
h+r_{h}\equiv p\pmod m$ and $0\leq h,r_{h},p<m$, we have $\varepsilon _{h}:=%
\frac{h+r_{h}-p}{m}\in \{0,1\}.$ If $s=mt+r_{h}\in S_{r_{h}}$ and $n=mq+h$,
then $n+s=m(q+\varepsilon _{h}+t)+p,$ and hence $\alpha
_{n+s}=x_{q+\varepsilon _{h}+t}.$

The coordinates in $n+(S\setminus S_{r_{h}})$ lie in constant residues of $%
\alpha $ and therefore contribute a fixed Parikh vector $C_{h}\in \mathbb{N}%
_{0}^{\mathbb{A}}$, independent of $q$. Hence we obtain that 
\begin{equation*}
\Psi (\alpha \lbrack n+S])=C_{h}+\Psi (x[(q+\varepsilon _{h})+L_{r_{h}}]).
\end{equation*}%
Consequently, 
\begin{equation}
\sharp \Xi _{h}(S)\leq p_{x}^{\mathrm{ab}}(L_{r_{h}})\leq p_{x}^{\ast 
\mathrm{ab}}(k_{r_{h}}).  \label{eq:mcl-P<=p_k}
\end{equation}

Since the map $h\longmapsto r_{h}$ is a permutation of $\{0,\ldots ,m-1\}$,
summing the preceding estimates (\ref{eq:mcl-P<=1}) and (\ref{eq:mcl-P<=p_k}%
) gives 
\begin{equation*}
p_{\alpha }^{\mathrm{ab}}(S)\leq \sharp \{r:k_{r}=0\}+\sum_{\substack{ 0\leq
r\leq m-1  \\ k_{r}>0}}p_{x}^{\ast \mathrm{ab}}(k_{r}).
\end{equation*}
\end{proof}

\begin{proposition}
\label{prop:interleaving-upper-bound} Let $x\in \{a,b\}^{\mathbb{N}_{0}}$ be
a binary word, and let $\alpha \in \mathbb{A}^{\mathbb{N}_{0}}$ be an $m$%
-interleaving of $x$. Then $p_{\alpha }^{\ast \mathrm{ab}}(k)\leq k+m$\ for
every $k\in \mathbb{N}.$
\end{proposition}

\begin{proof}
Let $S\in \Sigma _{k}(\mathbb{N}_{0}).$ For $0\leq r\leq m-1$, set 
\begin{equation*}
S_{r}:=\{s\in S:s\equiv r\text{ }(\func{mod}m)\},\text{\ }k_{r}:=\sharp
S_{r}.
\end{equation*}%
By Lemma \ref{lem:interleaving-decomposition} and (\ref%
{eq:binary-upper-bound}), one has 
\begin{equation*}
p_{\alpha }^{\mathrm{ab}}(S)\leq \sharp
\{r:k_{r}=0\}+\sum_{k_{r}>0}(k_{r}+1)=\sum_{r=0}^{m-1}k_{r}+m=k+m.
\end{equation*}%
By taking the supremum over $S\in \Sigma _{k}(\mathbb{N}_{0}),$ we have $%
p_{\alpha }^{\ast \mathrm{ab}}(k)\leq k+m.$
\end{proof}

\begin{theorem}
\label{thm:interleaving-exact} Let $x\in \{a,b\}^{\mathbb{N}_{0}}$ be a
recurrent aperiodic binary word, and let $\alpha $ be a separated $m$%
-interleaving of $x$. Then $p_{\alpha }^{\ast \mathrm{ab}}(k)=k+m$\ for
every $k\in \mathbb{N}.$
\end{theorem}

\begin{proof}
Let $p\in \{0,\ldots ,m-1\}$ be such that $\alpha ^{(p)}=x,$ and write $%
\alpha ^{(j)}=c_{j}^{\infty }$\ $(j\neq p).$ By Corollary \ref{cor:binary},
one has $p_{x}^{\ast \mathrm{ab}}(k)=k+1.$ Hence there exists a $k$-pattern $%
P=\{t_{1}<\cdots <t_{k}\}$ such that $p_{x}^{\mathrm{ab}}(P)=k+1.$ Set 
\begin{equation*}
S:=mP=\{mt_{1},\ldots ,mt_{k}\}.
\end{equation*}

Let $n\in \mathbb{N}_{0}$, and write uniquely 
\begin{equation*}
n=mq+j,\text{\ }q\in \mathbb{N}_{0},\text{\ }0\leq j\leq m-1.
\end{equation*}%
If $j=p$, then 
\begin{equation*}
\alpha \lbrack n+S]=\alpha \lbrack mq+p+S]=x[q+P].
\end{equation*}%
If $j\neq p$, then 
\begin{equation*}
\alpha \lbrack n+S]=c_{j}^{k}.
\end{equation*}%
Therefore we have%
\begin{equation*}
\Psi (\mathcal{F}_{\alpha }(S))=\Psi (\mathcal{F}_{x}(P))\cup \{k\mathbf{e}%
_{c_{j}}:0\leq j\leq m-1,j\neq p\}.
\end{equation*}

Since the interleaving is separated, the vectors $k\mathbf{e}_{c_{j}}$\ $%
(j\neq p)$ are pairwise distinct and do not belong to $\Psi (\mathcal{F}%
_{x}(P))$. Consequently, 
\begin{equation*}
p_{\alpha }^{\mathrm{ab}}(S)=p_{x}^{\mathrm{ab}}(P)+(m-1)=(k+1)+(m-1)=k+m.
\end{equation*}%
Thus $p_{\alpha }^{\ast \mathrm{ab}}(k)\geq k+m.$ Together with Proposition %
\ref{prop:interleaving-upper-bound}, which gives the reverse inequality, we
obtain 
\begin{equation*}
p_{\alpha }^{\ast \mathrm{ab}}(k)=k+m.
\end{equation*}
\end{proof}

\begin{lemma}
\label{lem:interleaving-recurrent-aperiodic}Let $\alpha $ be a separated $m$%
-interleaving of a binary word $x$. Then the following statements hold.

\begin{enumerate}
\item If $x$ is recurrent, then $\alpha $ is recurrent.

\item If $x$ is aperiodic, then $\alpha $ is aperiodic.
\end{enumerate}
\end{lemma}

\begin{proof}
Let $p\in \{0,\ldots ,m-1\}$ be such that $\alpha ^{(p)}=x,$ and write $%
\alpha ^{(j)}=c_{j}^{\infty }$\ $(j\neq p).$

We first prove (1). Assume that $x$ is recurrent, and let $w=\alpha
_{i}\alpha _{i+1}\cdots \alpha _{i+L-1}$ be a factor of $\alpha $.

Suppose first that no integer in $[i,i+L-1]$ is congruent to $p$ modulo $m$.
Since every residue of $\alpha $ other than $\alpha ^{(p)}$\ is constant,
one has 
\begin{equation*}
\alpha _{i+rm}\alpha _{i+rm+1}\cdots \alpha _{i+rm+L-1}=w\text{\ }(r\in 
\mathbb{N}_{0}).
\end{equation*}%
Thus $w$ occurs infinitely often in $\alpha $.

Now suppose that $[i,i+L-1]$ contains an integer congruent to $p$ modulo $m$%
. Let $q_{0}\leq q_{1}$ be such that the integers in this interval congruent
to $p$ are precisely 
\begin{equation*}
p+q_{0}m,\text{\ }p+(q_{0}+1)m,\ldots ,p+q_{1}m.
\end{equation*}%
The letters of $\alpha $ at these coordinates form the factor $%
x_{q_{0}}x_{q_{0}+1}\cdots x_{q_{1}}$ of $x$. Since $x$ is recurrent, there
exist infinitely many $t\in \mathbb{N}$ such that 
\begin{equation*}
x_{q_{0}+t}x_{q_{0}+1+t}\cdots x_{q_{1}+t}=x_{q_{0}}x_{q_{0}+1}\cdots
x_{q_{1}}.
\end{equation*}%
For each such $t$, translation by $mt$ preserves every residue class modulo $%
m$. On the residue $p$, the preceding equality preserves the corresponding
letters, while all other residues are constant. Hence 
\begin{equation*}
\alpha _{i+mt}\alpha _{i+mt+1}\cdots \alpha _{i+mt+L-1}=w.
\end{equation*}%
Thus $w$ occurs infinitely often in $\alpha $, and therefore $\alpha $ is
recurrent.

We now prove (2). If $m=1$, then $\alpha =x$, and there is nothing to prove.
Assume that $m\geq 2$ and, contrary to the assertion, that $\alpha $ is
eventually periodic with period $d\geq 1$.

Let $j\neq p$. Since the interleaving is separated, one has 
\begin{equation*}
\alpha _{n}=c_{j}\Longleftrightarrow n\equiv j\text{ }(\func{mod}m).
\end{equation*}%
For all sufficiently large $n\equiv j\pmod m$, eventual periodicity gives $%
\alpha _{n+d}=\alpha _{n}=c_{j}.$ It follows that $n+d\equiv j\pmod m,$ and
therefore $m\mid d.$ Write $d=mr$ with $r\geq 1$. Then for all sufficiently
large $q$, we have%
\begin{equation*}
x_{q+r}=\alpha _{p+m(q+r)}=\alpha _{p+mq+d}=\alpha _{p+mq}=x_{q}.
\end{equation*}%
Thus $x$ is eventually periodic, a contradiction. Hence $\alpha $ is
aperiodic.
\end{proof}

\begin{corollary}
\label{cor:sharpness}Let $\sharp \mathbb{A}=\ell \geq 2$. Then there exists
a recurrent aperiodic word $\alpha \in \mathbb{A}^{\mathbb{N}_{0}}$ such
that 
\begin{equation*}
p_{\alpha }^{\ast \mathrm{ab}}(k)=k+\ell -1
\end{equation*}%
for every $k\in \mathbb{N}$.
\end{corollary}

\begin{proof}
Let $\alpha \in \mathbb{A}^{\mathbb{N}_{0}}$ be a separated $(\ell -1)$%
-interleaving of a recurrent aperiodic binary word. By Lemma \ref%
{lem:interleaving-recurrent-aperiodic}, the word $\alpha $ is recurrent and
aperiodic. Theorem \ref{thm:interleaving-exact} then gives $p_{\alpha
}^{\ast \mathrm{ab}}(k)=k+\ell -1$ for every $k\in \mathbb{N}$.
\end{proof}

\subsection{Structural characterization}

\ 

We now prove Theorem \ref{thm:main-interleaving}. By Theorem \ref%
{thm:interleaving-exact}, it remains to prove the following structural
statement.

\begin{theorem}
\label{thm:interleaving}Let $\alpha $ be a recurrent aperiodic word with $%
\sharp \mathbb{A}_{\alpha }=\ell \geq 2.$ If $p_{\alpha }^{\ast \mathrm{ab}%
}(k)=k+\ell -1$\ for every $k\in \mathbb{N},$ then $\alpha $ is a separated $%
(\ell -1)$-interleaving of a recurrent aperiodic binary word.
\end{theorem}

We shall use the following simultaneous selection result.

\begin{proposition}
\label{prop:simultaneous-binary-factors}Let $\mathbb{D}$ be a finite
alphabet and let $z\in \mathbb{D}^{\mathbb{N}_{0}}$ be recurrent. Let $%
\varphi _{1},\ldots ,\varphi _{s}:\mathbb{D}\longrightarrow \{0,1\}$ be maps
such that $\varphi _{j}(z)$ is aperiodic for every $1\leq j\leq s$. Then for
every $L,q\in \mathbb{N}$, there exists an $L$-pattern $P=\{M_{0}<M_{1}<%
\cdots <M_{L-1}\}$ such that $M_{0}\equiv M_{1}\equiv \cdots \equiv M_{L-1}%
\pmod q$ and $p_{\varphi _{j}(z)}^{\mathrm{ab}}(P)=L+1$ for all $1\leq j\leq
s.$
\end{proposition}

\begin{proof}
Let $L,q\in \mathbb{N}$. By Proposition \ref{prop:simultaneous-selection},
there exists an infinite set $\mathcal{N}=\{N_{0}<N_{1}<N_{2}<\cdots
\}\subset \mathbb{N}_{0}$ such that $\Omega _{0}:=\mathcal{O}(z)[\mathcal{N}%
] $ is $L$-superstationary and has the tail-freezing property.

By the pigeonhole principle, there exists an infinite set $\mathcal{I}%
=\{i_{0}<i_{1}<i_{2}<\cdots \}\subset \mathbb{N}_{0}$ such that $%
N_{i_{0}}\equiv N_{i_{1}}\equiv N_{i_{2}}\equiv \cdots \pmod q$. Set $%
M_{r}:=N_{i_{r}}$ $(r\in \mathbb{N}_{0})$ and $\mathcal{M}:=\{M_{r}\}_{r\in 
\mathbb{N}_{0}}.$ Then $\mathcal{O}(z)[\mathcal{M}]=\Omega _{0}[\mathcal{I}%
]. $ By Lemma \ref{lem:tail-hered}, $\Omega :=\mathcal{O}(z)[\mathcal{M}]$
is $L $-superstationary and has the tail-freezing property.

Let $P:=\{M_{0},M_{1},\ldots ,M_{L-1}\}.$ For every $1\leq j\leq s$, Lemma %
\ref{lem:projection-preservation} implies that $\varphi _{j}(\Omega )=%
\mathcal{O}(\varphi _{j}(z))[\mathcal{M}],$ and $\varphi _{j}(\Omega )$ is $%
L $-superstationary and has the tail-freezing property. Since $\varphi
_{j}(z)$ is aperiodic, Lemma \ref{lem:binary} gives $\sharp \Psi \left( 
\mathcal{O}(\varphi _{j}(z))[P]\right) =L+1.$ By (\ref{eq:rel}), we obtain $%
p_{\varphi _{j}(z)}^{\mathrm{ab}}(P)=L+1.$ This holds for every $1\leq j\leq
s$, while the construction of $\mathcal{M}$ gives $M_{0}\equiv M_{1}\equiv
\cdots \equiv M_{L-1}\pmod q.$
\end{proof}

\begin{corollary}
\label{cor:simultaneous-all-counts} Let $\alpha \in \mathbb{A}^{\mathbb{N}%
_{0}}$ be recurrent, and let $\{a_{1},\ldots ,a_{s}\}\subset \mathbb{A}$ be
such that $\pi _{a_{i}}(\alpha )$ is aperiodic for every $1\leq i\leq s$.
Then for every $L,q\in \mathbb{N}$, there exists an $L$-pattern $%
P=\{M_{0}<M_{1}<\cdots <M_{L-1}\}$ such that $M_{0}\equiv M_{1}\equiv \cdots
\equiv M_{L-1}\pmod q$ and $p_{\pi _{a_{i}}(\alpha )}^{\mathrm{ab}}(P)=L+1$
for all $1\leq i\leq s.$
\end{corollary}

\begin{proof}
For $1\leq i\leq s$, define $\varphi _{i}:\mathbb{A}\longrightarrow \{0,1\}$
by\ $\varphi _{i}(a):=\mathbf{1}_{\{a_{i}\}}(a).$ Its coordinatewise
extension is $\pi _{a_{i}}$. The result now follows from Proposition \ref%
{prop:simultaneous-binary-factors} with $\mathbb{D}=\mathbb{A}$ and $%
z=\alpha .$
\end{proof}

Throughout the remainder of this subsection, let $\alpha $ satisfy the
hypotheses of Theorem\ \ref{thm:interleaving}.

\begin{lemma}
\label{lem:two-aperiodic-projections}Let $B:=\{c\in \mathbb{A}_{\alpha }:\pi
_{c}(\alpha )$ is aperiodic$\}.$ Then $\sharp B=2.$
\end{lemma}

\begin{proof}
We first show that $\sharp B\geq 2.$ Suppose that $\sharp B\leq 1$. Choose $%
b_{0}\in \mathbb{A}_{\alpha }$ such that $B\subset \{b_{0}\}$. Then $\pi
_{c}(\alpha )$ is eventually periodic for every $c\in \mathbb{A}_{\alpha
}\setminus \{b_{0}\}$. Taking a common eventual period and using 
\begin{equation*}
\sum_{c\in \mathbb{A}_{\alpha }}\mathbf{1}_{\{c\}}(\alpha _{n})=1\text{\ }%
(n\in \mathbb{N}_{0}),
\end{equation*}%
we see that $\pi _{b_{0}}(\alpha )$ is also eventually periodic. Thus
every $\pi _{c}(\alpha )$ is eventually periodic. Taking a common eventual
period gives $\alpha _{n+q}=\alpha _{n}$ for all sufficiently large $n$,
which contradicts the aperiodicity of $\alpha $. Hence $\sharp B\geq 2.$

For each $d\in \mathbb{A}_{\alpha }\setminus B$, the word $\pi _{d}(\alpha )$
is eventually periodic. By Lemma \ref{lem:projection}, the word $\pi
_{d}(\alpha )$ is also recurrent, and hence periodic. Choose a period $q_{d}$
of $\pi _{d}(\alpha )$ for each $d\in \mathbb{A}_{\alpha }\setminus B$, and
let 
\begin{equation*}
q:=\func{lcm}\{q_{d}:d\in \mathbb{A}_{\alpha }\setminus B\},
\end{equation*}%
with $q=1$ if $\mathbb{A}_{\alpha }\setminus B=\varnothing $.

Apply Corollary \ref{cor:simultaneous-all-counts} to the letters in $B$,
with $L=2$ and the above value of $q$. We obtain a $2$-pattern $%
P=\{M_{0}<M_{1}\}$ with $M_{0}\equiv M_{1}\pmod q$ such that $p_{\pi
_{c}(\alpha )}^{\mathrm{ab}}(P)=3$ for every $c\in B.$ By (\ref%
{eq:binary-abelian-count}), one has 
\begin{equation}
\mathcal{C}_{\pi _{c}(\alpha )}(P)=\{0,1,2\}\ (c\in B).  \label{eq:0,1,2}
\end{equation}

For each $c\in B$, (\ref{eq:0,1,2}) yields a $P$-factor of $\alpha $ whose
two letters are both $c$, and hence $2\mathbf{e}_{c}\in \Psi (\mathcal{F}%
_{\alpha }(P)).$ It also yields a $P$-factor containing exactly one
occurrence of $c$. Thus, for some $\delta (c)\in \mathbb{A}_{\alpha
}\setminus \{c\}$, 
\begin{equation*}
\mathbf{e}_{c}+\mathbf{e}_{\delta (c)}\in \Psi (\mathcal{F}_{\alpha }(P)).
\end{equation*}

Now let $d\in \mathbb{A}_{\alpha }\setminus B$. By recurrence, choose $j\geq
M_{0}$ such that $\alpha _{j}=d.$ Put $n:=j-M_{0}.$ Since $q_{d}\mid q$ and $%
q\mid (M_{1}-M_{0})$, one has%
\begin{equation*}
\pi _{d}(\alpha )_{n+M_{1}}=\pi _{d}(\alpha )_{n+M_{0}}=1.
\end{equation*}%
Hence $\alpha _{n+M_{0}}=\alpha _{n+M_{1}}=d,$ and therefore $2\mathbf{e}%
_{d}\in \Psi (\mathcal{F}_{\alpha }(P)).$

Consequently, we have%
\begin{equation*}
\{2\mathbf{e}_{c}:c\in \mathbb{A}_{\alpha }\}\cup \{\mathbf{e}_{c}+\mathbf{e}%
_{\delta (c)}:c\in B\}\subset \Psi (\mathcal{F}_{\alpha }(P)).
\end{equation*}%
The two sets on the left are disjoint. Thus 
\begin{equation*}
\ell +\sharp \{\mathbf{e}_{c}+\mathbf{e}_{\delta (c)}:c\in B\}\leq p_{\alpha
}^{\mathrm{ab}}(P)\leq p_{\alpha }^{\ast \mathrm{ab}}(2)=\ell +1.
\end{equation*}%
It follows that $\sharp \{\mathbf{e}_{c}+\mathbf{e}_{\delta (c)}:c\in
B\}\leq 1.$ Since $B\neq \varnothing $, all these vectors $\mathbf{e}_{c}+%
\mathbf{e}_{\delta (c)}$ $(c\in B)$ coincide. Write their common value as $%
\mathbf{e}_{a}+\mathbf{e}_{b}$ with $a\neq b.$ For every $c\in B$, $\mathbf{e%
}_{c}+\mathbf{e}_{\delta (c)}=\mathbf{e}_{a}+\mathbf{e}_{b},$ and hence $%
\{c,\delta (c)\}=\{a,b\}.$ Hence $B\subset \{a,b\}.$ Since $\sharp B\geq 2$,
we conclude that $B=\{a,b\}.$
\end{proof}

Write $B=\{a,b\},$ $C:=\mathbb{A}_{\alpha }\setminus B.$ Let $\ast $ be a
symbol not in $C$, and define $\rho :\mathbb{A}_{\alpha }\longrightarrow
C\cup \{\ast \}$ by 
\begin{equation*}
\rho (a)=\rho (b)=\ast ,\text{\ }\rho (c)=c\text{\ }(c\in C).
\end{equation*}%
Write $\sigma :=\rho (\alpha ).$

The word $\sigma $ is periodic. In fact, the proof of Lemma \ref%
{lem:two-aperiodic-projections} shows that $\pi _{c}(\alpha )$ is periodic
for every $c\in C$. Take a period $q_{c}$ of $\pi _{c}(\alpha ),$ and let 
\begin{equation*}
q:=\func{lcm}\{q_{c}:c\in C\},
\end{equation*}%
with $q:=1$ if $C=\varnothing $. Then $\mathbf{1}_{\{c\}}(\alpha _{n+q})=%
\mathbf{1}_{\{c\}}(\alpha _{n})$\ for every $c\in C$ and $n\in \mathbb{N}%
_{0}.$ It follows that $\rho (\alpha _{n+q})=\rho (\alpha _{n})$ for all $%
n\in \mathbb{N}_{0},$ so $\sigma $ is periodic with period $q$.

Let $m$ be the least period of $\sigma $. Then for $0\leq j\leq m-1$,%
\begin{equation*}
\left\{ 
\begin{array}{ll}
\alpha ^{(j)}=c^{\infty }, & \text{if }\sigma _{j}=c\in C, \\ 
\alpha ^{(j)}\in B^{\mathbb{N}_{0}}, & \text{if }\sigma _{j}=\ast .%
\end{array}%
\right.
\end{equation*}

\begin{lemma}
\label{lem:unique-aper}Every residue $\alpha ^{(j)}$ is recurrent. Moreover,
there exists a unique $p\in \{0,\ldots ,m-1\}$ such that $\alpha ^{(p)}$ is
aperiodic.
\end{lemma}

\begin{proof}
Define 
\begin{equation*}
\beta =\beta _{0}\beta _{1}\beta _{2}\cdots \in (\mathbb{A}_{\alpha }^{m})^{%
\mathbb{N}_{0}},\text{\ }\beta _{n}:=(\alpha _{mn},\alpha _{mn+1},\ldots
,\alpha _{mn+m-1}).
\end{equation*}%
We first show that $\beta $ is recurrent. Let $\beta _{s}\beta _{s+1}\cdots
\beta _{s+r-1}$ be a factor of $\beta $. It corresponds to the factor $%
w:=\alpha _{ms}\alpha _{ms+1}\cdots \alpha _{m(s+r)-1}$ of $\alpha $. Since $%
\alpha $ is recurrent, $w$ occurs infinitely often in $\alpha $. Suppose
that an occurrence of $w$ begins at $t$. Applying $\rho $ to the first $m$
letters of this occurrence gives 
\begin{equation*}
\sigma _{t}\sigma _{t+1}\cdots \sigma _{t+m-1}=\sigma _{ms}\sigma
_{ms+1}\cdots \sigma _{ms+m-1}=\sigma _{0}\sigma _{1}\cdots \sigma _{m-1}.
\end{equation*}%
Since $\sigma $ has least period $m$, it follows that $t\equiv 0\pmod m.$
Thus every occurrence of $w$ begins at a multiple of $m$, and hence induces
an occurrence of $\beta _{s}\beta _{s+1}\cdots \beta _{s+r-1}$ in $\beta $.
Therefore $\beta $ is recurrent.

For each $0\leq j<m$, the residue $\alpha ^{(j)}$ is the $j$th coordinate
projection of $\beta $, and hence is recurrent.

At least one residue $\alpha ^{(j)}$ is aperiodic. Otherwise, by recurrence,
all residues would be periodic. Then a common period of the finitely many
residues would give a period of $\alpha $, a contradiction. Since $\alpha
^{(j)}$ is constant whenever $\sigma _{j}\in C$, every aperiodic residue
satisfies $\sigma _{j}=\ast $. Fix one such residue and denote it by $\alpha
^{(p)}.$

Suppose, for a contradiction, that $\alpha ^{(i)}$ is aperiodic for some $%
i\in \{0,\ldots ,m-1\}\setminus \{p\}$. Then $\sigma _{p}=\sigma _{i}=\ast .$
Since $m$ is the least period of $\sigma $, there exists $1\leq h<m$ such
that $\sigma _{p+h}\neq \sigma _{i+h}.$ Interchanging $p$ and $i$ if
necessary, we may assume that $\sigma _{p+h}\in C.$ Set 
\begin{equation*}
c_{0}:=\sigma _{p+h}\in C,\text{\ }d_{0}:=\sigma _{i+h}\in C\cup \{\ast \},%
\text{\ }d_{0}\neq c_{0}.
\end{equation*}

For $j\in \{p,i\}$, define 
\begin{equation*}
\varphi _{j}:\mathbb{A}_{\alpha }^{m}\longrightarrow \{0,1\},\text{ }\varphi
_{j}(u_{0},\ldots ,u_{m-1}):=\mathbf{1}_{\{a\}}(u_{j}).
\end{equation*}%
Then $\varphi _{j}(\beta )=\pi _{a}(\alpha ^{(j)}).$ Since $\alpha ^{(p)}$
and $\alpha ^{(i)}$ are aperiodic binary words over $\{a,b\}$, both $\pi
_{a}(\alpha ^{(p)})$ and\ $\pi _{a}(\alpha ^{(i)})$ are aperiodic.

Applying Proposition \ref{prop:simultaneous-binary-factors} to $\beta $ and
the maps $\varphi _{p},\varphi _{i},$ with $L:=2\ell -2$ and $q=1$, we
obtain an $L$-pattern $P=\{M_{0}<M_{1}<\cdots <M_{L-1}\}$ such that 
\begin{equation}
p_{\pi _{a}(\alpha ^{(j)})}^{\mathrm{ab}}(P)=L+1\text{\ }(j\in \{p,i\}).
\label{eq:L+1-i,p}
\end{equation}%
Set 
\begin{equation*}
S:=mP\cup \{h\}.
\end{equation*}%
Since $1\leq h<m$, we have $h\notin mP$, and hence $\sharp S=L+1=2\ell -1.$

It follows from (\ref{eq:L+1-i,p}) that for every $0\leq t\leq L,$ there
exists $q_{t}\in \mathbb{N}_{0}$ such that $\sharp \{s\in P:\alpha
_{q_{t}+s}^{(p)}=a\}=t.$ Since $\sigma _{p+h}=c_{0}\in C$, we have $\alpha
_{p+mq_{t}+h}=c_{0}.$ Hence%
\begin{equation*}
V_{p}:=\left\{ t\mathbf{e}_{a}+(L-t)\mathbf{e}_{b}+\mathbf{e}_{c_{0}}:0\leq
t\leq L\right\} \subset \Psi (\mathcal{F}_{\alpha }(S)).
\end{equation*}%
In particular, $\sharp V_{p}=L+1.$

If $d_{0}\in C$, the same argument applied to $\alpha ^{(i)}$ gives 
\begin{equation*}
V_{i}:=\left\{ t\mathbf{e}_{a}+(L-t)\mathbf{e}_{b}+\mathbf{e}_{d_{0}}:0\leq
t\leq L\right\} \subset \Psi (\mathcal{F}_{\alpha }(S)),
\end{equation*}%
and hence $\sharp V_{i}=L+1.$

Suppose now that $d_{0}=\ast $. For each $0\leq t\leq L$, choose $r_{t}\in 
\mathbb{N}_{0}$ such that $\sharp \{s\in P:\alpha _{r_{t}+s}^{(i)}=a\}=t.$
Since $\sigma _{i+h}=\ast $, we have $\alpha _{i+mr_{t}+h}\in \{a,b\}.$ Let $%
\varepsilon _{t}:=\mathbf{1}_{\{a\}}(\alpha _{i+mr_{t}+h}).$ Hence%
\begin{equation*}
V_{i}:=\left\{ (t+\varepsilon _{t})\mathbf{e}_{a}+(L-t+1-\varepsilon _{t})%
\mathbf{e}_{b}:0\leq t\leq L\right\} \subset \Psi (\mathcal{F}_{\alpha }(S)).
\end{equation*}%
Since $t+\varepsilon _{t}\in \{t,t+1\},$ each $a$-coordinate arises from at
most two values of $t$. Therefore $\sharp V_{i}\geq \left\lceil \frac{L+1}{2}%
\right\rceil =\ell .$

In either case, $V_{p}\cap V_{i}=\varnothing .$ Indeed, every vector in $%
V_{p}$ has $c_{0}$-coordinate $1$, while every vector in $V_{i}$ has $c_{0}$%
-coordinate $0$. Hence 
\begin{equation*}
p_{\alpha }^{\mathrm{ab}}(S)\geq \sharp V_{p}+\sharp V_{i}\geq (L+1)+\ell
=3\ell -1,
\end{equation*}%
which contradicts the bound
\begin{equation*}
p_{\alpha }^{\mathrm{ab}}(S)\leq p_{\alpha }^{\ast \mathrm{ab}}(2\ell
-1)=(2\ell -1)+\ell -1=3\ell -2.
\end{equation*}%
Thus no such $i$ exists, and $\alpha ^{(p)}$ is the unique aperiodic residue.
\end{proof}

\begin{lemma}
\label{lem:else-constant}Assume that $\alpha ^{(p)}$ is aperiodic for some $%
p\in \{0,\ldots ,m-1\}$. Then $\alpha ^{(j)}$ is constant for every $j\in
\{0,\ldots ,m-1\}\setminus \{p\}$.
\end{lemma}

\begin{proof}
Fix $i\in \{0,\ldots ,m-1\}\setminus \{p\}$. If $\sigma _{i}\in C$, then $%
\alpha ^{(i)}=\sigma _{i}^{\infty }$ is constant. Suppose that $\sigma
_{i}=\ast .$ By Lemma \ref{lem:unique-aper}, the residue $\alpha ^{(i)}$ is
recurrent and eventually periodic. Then $\alpha ^{(i)}$ is periodic. Let $q$
be its least period. If $q=1$, then $\alpha ^{(i)}$ is constant, so suppose
that $q\geq 2.$

We first claim that there exist $1\leq h<m$ and $c_{0}\in C$ such that 
\begin{equation}
\sigma _{p+h}=c_{0},\text{ }\sigma _{i+h}\neq c_{0}.  \label{eq:+h-2}
\end{equation}%
Indeed, suppose otherwise. Let 
\begin{equation*}
D_{p}:=\{0\leq h<m:\sigma _{p+h}\in C\},\text{\ }D_{i}:=\{0\leq h<m:\sigma
_{i+h}\in C\}.
\end{equation*}%
Our assumption implies that $\sigma _{i+h}=\sigma _{p+h}$ for every $h\in
D_{p},$ and hence $D_{p}\subset D_{i}.$ Since addition by $p$ and by $i$
permutes the residue classes modulo $m$, the sets $D_{p}$ and $D_{i}$ have
the same cardinality. Thus $D_{p}=D_{i}.$ It follows that $\sigma
_{p+h}=\sigma _{i+h}$ for all $0\leq h<m.$ If $\delta \in \{1,\ldots ,m-1\}$
satisfies $\delta \equiv i-p\pmod m$, then $\delta $ is a period of $\sigma $%
, which contradicts the minimality of $m$.

Since $\alpha ^{(p)}$ is a recurrent aperiodic word over $B=\{a,b\}$, the
word $\pi _{a}(\alpha ^{(p)})$ is aperiodic. Applying Corollary \ref%
{cor:simultaneous-all-counts} to $\alpha ^{(p)}$, with $L=2$ and the above
value of $q$, we obtain a $2$-pattern $P=\{M_{0}<M_{1}\}$ such that $%
M_{0}\equiv M_{1}\pmod q$ and 
\begin{equation}
p_{\pi _{a}(\alpha ^{(p)})}^{\mathrm{ab}}(P)=3.  \label{eq:3-p}
\end{equation}%
Set 
\begin{equation*}
S:=mP\cup \{h\}.
\end{equation*}%
Since $1\leq h<m$, we have $\sharp S=3.$

By (\ref{eq:3-p}), for every $0\leq t\leq 2$, there exists $r_{t}\in \mathbb{%
N}_{0}$ such that 
\begin{equation*}
\left\vert \alpha ^{(p)}[r_{t}+P]\right\vert _{a}=t.
\end{equation*}%
Since $\sigma _{p+h}=c_{0}$, we have $\alpha _{p+mr_{t}+h}=c_{0}.$ It
follows that 
\begin{equation*}
V_{p}:=\left\{ t\mathbf{e}_{a}+(2-t)\mathbf{e}_{b}+\mathbf{e}_{c_{0}}:0\leq
t\leq 2\right\} \subset \Psi (\mathcal{F}_{\alpha }(S)).
\end{equation*}

Now consider the residue $\alpha ^{(i)}$. Since $\alpha ^{(i)}$ is a
nonconstant periodic word over $\{a,b\}$, both $a$ and $b$ occur in it. Take 
$r_{a},r_{b}\in \mathbb{N}_{0}$ such that $\alpha _{r_{a}+M_{0}}^{(i)}=a$
and $\alpha _{r_{b}+M_{0}}^{(i)}=b.$ Since $q\mid (M_{1}-M_{0})$, one has $%
\alpha _{r_{a}+M_{1}}^{(i)}=a$ and $\alpha _{r_{b}+M_{1}}^{(i)}=b.$ Let $%
\gamma _{a}:=\alpha _{i+mr_{a}+h}$ and $\gamma _{b}:=\alpha _{i+mr_{b}+h}.$
By (\ref{eq:+h-2}), we have $\gamma _{a},\gamma _{b}\neq c_{0}.$
Consequently, 
\begin{equation*}
V_{i}:=\left\{ 2\mathbf{e}_{a}+\mathbf{e}_{\gamma _{a}},\,2\mathbf{e}_{b}+%
\mathbf{e}_{\gamma _{b}}\right\} \subset \Psi (\mathcal{F}_{\alpha }(S)),
\end{equation*}%
where the two vectors in $V_{i}$ are distinct. Moreover, every vector in $%
V_{p}$ has $c_{0}$-coordinate $1$, while every vector in $V_{i}$ has $c_{0}$%
-coordinate $0$. Hence $V_{p}\cap V_{i}=\varnothing .$

For each $c\in C$, choose $0\leq j_{c}<m$ such that $\sigma _{j_{c}}=c.$
Then $\alpha ^{(j_{c})}=c^{\infty }.$ Therefore the two coordinates
corresponding to $mP$ in the $S$-factor $\alpha \lbrack j_{c}+S]$ are both $%
c $. Let $w_{c}:=\Psi (\alpha \lbrack j_{c}+S]).$ Then $(w_{c})_{c}\geq 2.$
Thus the vectors in $W:=\{w_{c}:c\in C\}$ are pairwise distinct, and $\sharp
W=\sharp C=\ell -2.$ Furthermore, one has $W\cap (V_{p}\cup
V_{i})=\varnothing ,$ since every $C$-coordinate of a vector in $V_{p}\cup
V_{i}$ is at most $1$.

Since $V_{p},$ $V_{i}$ and $W$ are pairwise disjoint, one has 
\begin{equation*}
\ell +3=\sharp V_{p}+\sharp V_{i}+\sharp W\leq p_{\alpha }^{\mathrm{ab}%
}(S)\leq p_{\alpha }^{\ast \mathrm{ab}}(3)=\ell +2,
\end{equation*}%
a contradiction. Hence $\alpha ^{(i)}$ is constant.
\end{proof}

\begin{proof}[Proof of Theorem~\protect\ref{thm:interleaving}]
If $\ell=2$, the conclusion is immediate. Assume that $\ell\geq3$.

Recall that $m$ is the least period of $\sigma $. By Lemmas \ref%
{lem:unique-aper} and \ref{lem:else-constant}, there exists a unique $0\leq
p<m$ such that $x:=\alpha ^{(p)}$ is a recurrent aperiodic binary word,
while every other residue of $\alpha $ is constant. Thus $\alpha $ is an $m$%
-interleaving of $x$.

For each $0\leq j<m$, Corollary \ref{cor:binary} gives $p_{x}^{\ast \mathrm{%
ab}}(2^{j})=2^{j}+1.$ Choose a $2^{j}$-pattern $P_{j}$ such that $p_{x}^{%
\mathrm{ab}}(P_{j})=2^{j}+1,$ and set 
\begin{equation*}
S:=\bigcup_{j=0}^{m-1}(mP_{j}+j).
\end{equation*}%
Since the sets $j+mP_{j}$ lie in distinct residue classes modulo $m$, they
are pairwise disjoint. Hence $K:=\sharp S=\sum_{j=0}^{m-1}2^{j}=2^{m}-1.$

For $0\leq t<m$, define 
\begin{equation*}
V_{t}:=\{\Psi (\alpha \lbrack t+mr+S]):r\in \mathbb{N}_{0}\}.
\end{equation*}%
Let $0\leq j_{t}<m$ be the unique integer satisfying $t+j_{t}\equiv p\pmod m%
, $ and define $\varepsilon _{t}:=(t+j_{t}-p)/m\in \{0,1\}.$ Then for $r\in 
\mathbb{N}_{0}$ and $q\in P_{j_{t}}$, we have%
\begin{equation*}
\alpha _{t+mr+mq+j_{t}}=x_{r+q+\varepsilon _{t}}.
\end{equation*}%
If $j\neq j_{t}$, then all coordinates $t+mr+mq+j$\ $(q\in P_{j})$ lie in
the same constant residue of $\alpha $. Hence there exists $c_{t,j}\in 
\mathbb{A}_{\alpha }$, independent of $r$ and $q$, such that 
\begin{equation*}
\alpha _{t+mr+mq+j}=c_{t,j}\text{\ }(q\in P_{j}).
\end{equation*}%
Consequently, 
\begin{equation*}
\Psi (\alpha \lbrack t+mr+S])=\sum_{j\neq j_{t}}2^{j}\mathbf{e}%
_{c_{t,j}}+\Psi (x[(r+\varepsilon _{t})+P_{j_{t}}]).
\end{equation*}

Since $x$ is recurrent, every $P_{j_{t}}$-factor of $x$ occurs infinitely
often. Hence restricting the index to $r+\varepsilon _{t},$ $r\in \mathbb{N}%
_{0},$ does not remove any Abelian $P_{j_{t}}$-class. Therefore one has 
\begin{equation*}
\sharp V_{t}=p_{x}^{\mathrm{ab}}(P_{j_{t}})=2^{j_{t}}+1.
\end{equation*}

We next show that the sets $V_{t}$ $(0\leq t<m)$ are pairwise disjoint. For $%
c\in C$, every vector in $V_{t}$ has $c$-coordinate 
\begin{equation*}
E_{t}(c):=\sum_{\substack{ 0\leq j<m  \\ \sigma _{t+j}=c}}2^{j}.
\end{equation*}%
Suppose that $V_{t}\cap V_{t^{\prime }}\neq \varnothing .$ Then $%
E_{t}(c)=E_{t^{\prime }}(c)$ for all $c\in C.$ By uniqueness of binary
expansion, we obtain that%
\begin{equation*}
\{0\leq j<m:\sigma _{t+j}=c\}=\{0\leq j<m:\sigma _{t^{\prime }+j}=c\}\text{\ 
}(c\in C).
\end{equation*}%
It follows that $\sigma _{t+j}=\sigma _{t^{\prime }+j}$\ $(0\leq j<m).$
Since $m$ is the least period of $\sigma $, we must have $t=t^{\prime }.$
Thus the sets $V_{t}$ are pairwise disjoint.

Every $n\in \mathbb{N}_{0}$ has a unique representation $n=t+mr$ $(0\leq
t<m,r\in \mathbb{N}_{0}).$ Hence $\Psi (\mathcal{F}_{\alpha
}(S))=\bigcup_{t=0}^{m-1}V_{t}.$ Since the union is disjoint and the map $%
t\longmapsto j_{t}$ is a permutation of $\{0,\ldots ,m-1\}$, one has 
\begin{equation*}
p_{\alpha }^{\mathrm{ab}}(S)=\sum_{t=0}^{m-1}\sharp
V_{t}=\sum_{t=0}^{m-1}(2^{j_{t}}+1)=K+m.
\end{equation*}%
Hence 
\begin{equation*}
K+m=p_{\alpha }^{\mathrm{ab}}(S)\leq p_{\alpha }^{\ast \mathrm{ab}%
}(K)=K+\ell -1,
\end{equation*}%
and therefore $m\leq \ell -1.$

On the other hand, every letter of $C$ occurs in $\alpha $ and hence in $%
\sigma $, while $\sigma _{p}=\ast .$ Thus a period block of $\sigma $
contains all symbols of $C\cup \{\ast \}$ whose cardinality is $\ell -1$.
Thus $m\geq \ell -1.$ Together with the reverse inequality, this gives $%
m=\ell -1.$

Therefore each symbol of $C\cup \{\ast \}$ occurs exactly once in the period
block of $\sigma $. In particular, $\ast $ occurs only at $p$, while the
remaining residues are the constant words $c^{\infty }$ $(c\in C)$, each
occurring exactly once. Since $\mathbb{A}_{x}=B$ and $C\cap B=\varnothing $, 
$\alpha $ is a separated $(\ell -1)$-interleaving of $x$.
\end{proof}

\bigskip

\section{The binary bound and the periodicity criterion}

\label{sec:basic-lower-AW}

In this section, we turn to aperiodic words without any recurrence
assumption and prove Theorem \ref{thm:threshold}.

We shall use the following periodic approximation lemma. A bi-infinite word $%
\gamma \in \mathbb{A}^{\mathbb{Z}}$ is $T$-\emph{periodic} if $\gamma
_{n+T}=\gamma _{n}$ for every $n\in \mathbb{Z}$.

\begin{lemma}
\label{lem:alignment} Suppose that $\alpha \in \mathbb{A}^{\mathbb{N}_{0}}$
is aperiodic and that $\beta \in \mathcal{O}(\alpha )$ is periodic with
period $T$. Then there exist a $T$-periodic word $\gamma \in \mathbb{A}^{%
\mathbb{Z}}$ and integers $b_{j}\in \mathbb{N}_{0}$, $L_{j}\in \mathbb{N}$
such that 
\begin{equation*}
b_{j}\longrightarrow \infty ,\text{\ }L_{j}\longrightarrow \infty ,\text{\
and }\alpha _{t}=\gamma _{t}\text{\ for all }b_{j}\leq t<b_{j}+L_{j}.
\end{equation*}%
In particular, $\gamma _{s}^{\infty }\in \mathcal{O}(\alpha ^{(s)})$ for
each $s\in \{0,\ldots ,T-1\}$.
\end{lemma}

\begin{proof}
Since $\beta \in \mathcal{O}(\alpha )$, there exist integers $b_{j}\in 
\mathbb{N}_{0}$ and $L_{j}\in \mathbb{N}$ with $L_{j}\rightarrow \infty $
such that 
\begin{equation*}
\alpha _{b_{j}+k}=\beta _{k}\text{\ }(0\leq k<L_{j}).
\end{equation*}%
Necessarily $b_{j}\rightarrow \infty $. Otherwise, some subsequence would
satisfy $b_{j}=b$ for a fixed $b$, and hence $\alpha _{b+k}=\beta _{k}$ for
every $k\in \mathbb{N}_{0}.$ This implies that $\alpha $ is eventually
periodic.

Passing to a subsequence, we may assume that $b_{j}\equiv \nu \pmod T$ for
some $\nu \in \{0,\ldots ,T-1\}$. Let $\widetilde{\beta }\in \mathbb{A}^{%
\mathbb{Z}}$ be the $T$-periodic extension of $\beta $, and define $\gamma
_{t}:=\widetilde{\beta }_{t-\nu }$\ $(t\in \mathbb{Z}).$ Then $\gamma $ is $%
T $-periodic. If $b_{j}\leq t<b_{j}+L_{j}$, then $\alpha _{t}=\beta
_{t-b_{j}}=\widetilde{\beta }_{t-\nu }=\gamma _{t}.$

Fix $s\in \{0,\ldots ,T-1\}$. For every $n$ satisfying $b_{j}\leq
Tn+s<b_{j}+L_{j}$, we have $\alpha _{n}^{(s)}=\gamma _{s}.$ These values of $%
n$ form intervals of consecutive integers whose lengths tend to infinity and
whose initial points tend to infinity.\ Hence $\gamma _{s}^{\infty }\in 
\mathcal{O}(\alpha ^{(s)})$.
\end{proof}

For each integer $k\geq 1,$ write 
\begin{equation*}
f(k):=\min \left\{ m\in \mathbb{N}:\binom{m}{2}\geq k\right\} .
\end{equation*}%
Define $\kappa (1):=0,$ $\kappa (m):=1+\binom{m-1}{2}\ (m\geq 2).$ For
integers $k\geq 1$ and $m\geq 2$, the definition of $f$ gives 
\begin{equation}
f(k)=m\Longleftrightarrow \binom{m-1}{2}<k\leq \binom{m}{2}.
\label{eq:threshold-integers}
\end{equation}%
Consequently, $f(k)=\max \{m\geq 1:\kappa (m)\leq k\}$\ and $2\leq f(k)\leq
k+1$ for $k\geq 1.$

\begin{lemma}
\label{lem:separation} Let $x\in \{0,1\}^{\mathbb{N}_{0}}$ have infinitely
many occurrences of $1$, and suppose that $0^{\infty }\in \mathcal{O}(x)$.
Let $u_{0}<\cdots <u_{r-1}$ be elements of$\ \mathbb{N}_{0}$, and let $S_{0}$
be a pattern such that $c_{S_{0}}(u_{i})\geq b\geq 1$ for all $i$. Then
there is a pattern $S_{1}\supset S_{0}$ such that $\sharp \left(
S_{1}\setminus S_{0}\right) \leq \binom{r}{2},$ the counts $%
c_{S_{1}}(u_{0}),\ldots ,c_{S_{1}}(u_{r-1})$ are pairwise distinct and $%
c_{S_{1}}(u_{i})\geq b\geq 1$ for all $i$. Moreover, for every $k\geq \sharp
S_{1}$, the pattern $S_{1}$ can be enlarged to a $k$-pattern without
changing any of these counts.
\end{lemma}

\begin{proof}
The case $r=1$ is immediate. Assume $r\ge2$.

Set $S^{(0)}=S_{0}$. We construct inductively $S^{(0)}\subset S^{(1)}\subset
\cdots \subset S^{(r-1)}$ such that, for each $1\leq j<r$, the counts $%
c_{S^{(j)}}(u_{0}),\ldots ,c_{S^{(j)}}(u_{j})$ are pairwise distinct, while
the counts at $u_{0},\ldots ,u_{j-1}$ are unchanged in passing from $%
S^{(j-1)}$ to $S^{(j)}$.

Suppose that $S^{(j-1)}$ has been constructed. By Lemma \ref{lem:zero-blocks}%
(1), applied with $L=u_{j}-u_{0},$ there are arbitrarily large $p\in \mathbb{%
N}_{0}$ such that 
\begin{equation*}
x_{u_{i}+p}=0\text{\ }(0\leq i<j)\text{\ and }x_{u_{j}+p}=1.
\end{equation*}

Choose distinct such integers $p_{1},\ldots ,p_{j}$ outside $S^{(j-1)}$, and
put $a_{i}:=c_{S^{(j-1)}}(u_{i})$\ $(0\leq i\leq j).$ For $0\leq \lambda
\leq j$, let 
\begin{equation*}
S_{\lambda }^{(j-1)}:=S^{(j-1)}\cup \{p_{1},\ldots ,p_{\lambda }\},
\end{equation*}%
with $S_{0}^{(j-1)}=S^{(j-1)}$. Then $c_{S_{\lambda }^{(j-1)}}(u_{i})=a_{i}$%
\ for all $0\leq i<j,$ while $c_{S_{\lambda }^{(j-1)}}(u_{j})=a_{j}+\lambda
. $ For each $i<j$, the equality $a_{j}+\lambda =a_{i}$ excludes at most one
value of $\lambda $. Since there are $j+1$ choices $\lambda \in \{0,\ldots
,j\}$, some choice makes $a_{j}+\lambda $ distinct from $a_{0},\ldots
,a_{j-1}$. Define $S^{(j)}=S_{\lambda }^{(j-1)}$ for such a choice.

Set $S_{1}:=S^{(r-1)}.$ At the $j$th step at most $j$ elements were added,
and hence $\sharp \left( S_{1}\setminus S_{0}\right) \leq \sum_{j=1}^{r-1}j=%
\binom{r}{2}.$ By construction, the counts $c_{S_{1}}(u_{0}),\ldots
,c_{S_{1}}(u_{r-1})$ are pairwise distinct. It follows that $%
c_{S_{1}}(u_{i})\geq c_{S_{0}}(u_{i})\geq b$\ for all $0\leq i<r.$

Finally, let $k\geq \sharp S_{1}$. By Lemma \ref{lem:zero-blocks}(3), there
are arbitrarily large $p$ such that $x_{u_{i}+p}=0$\ for all $0\leq i<r.$
Choose $k-\sharp S_{1}$ distinct such integers outside $S_{1}$ and adjoin
them to $S_{1}$. This produces a $k$-pattern and leaves all the above counts
unchanged.
\end{proof}

\begin{proposition}
\label{prop:zerolower} If a binary word $x$ has infinitely many occurrences
of $1$ and $0^{\infty }\in \mathcal{O}(x)$, then $p_{x}^{\ast \mathrm{ab}%
}(k)\geq f(k)$ for every $k\in \mathbb{N}$.
\end{proposition}

\begin{proof}
Fix $k\in \mathbb{N}$, and put $m:=f(k)$, $r:=m-1\geq 1$. Take $u_{0}<\cdots
<u_{r-1}$ such that $x_{u_{i}}=1$ for all $0\leq i<r,$ and let $S_{0}:=\{0\}$%
. Thus $c_{S_{0}}(u_{i})=1$\ for all $0\leq i<r.$ Applying Lemma \ref%
{lem:separation} with $b=1$, we obtain a pattern $S_{1}\supset S_{0}$ such
that the counts $c_{S_{1}}(u_{0}),\ldots ,c_{S_{1}}(u_{r-1})$ are pairwise
distinct and positive, and $\sharp S_{1}\leq 1+\binom{r}{2}=\kappa (m)\leq
k. $ By the last assertion of Lemma \ref{lem:separation}, $S_{1}$ can be
extended to a $k$-pattern $S$ without changing these counts. By Lemma \ref%
{lem:zero-blocks}(2), the value $0$ is also attained by $c_{S}$. Therefore $%
p_{x}^{\ast \mathrm{ab}}(k)\geq p_{x}^{\mathrm{ab}}(S)\geq f(k).$
\end{proof}

\begin{proposition}
\label{prop:binarylower} If $x\in \{0,1\}^{\mathbb{N}_{0}}$ is aperiodic,
then $p_{x}^{\ast \mathrm{ab}}(k)\geq f(k)$\ for all $k\in \mathbb{N}.$
\end{proposition}

\begin{proof}
If $\mathcal{O}(x)$ contains no periodic point, then Lemma \ref%
{lem:noperiodic} gives $p_{x}^{\ast \mathrm{ab}}(k)\geq k+1\geq f(k).$

Suppose that $\mathcal{O}(x)$ contains a periodic point with period $T$. By
Lemma \ref{lem:alignment}, each residue $x^{(s)}$ $(0\leq s<T)$ has a
constant point in its orbit closure. At least one $x^{(s)}$ is aperiodic;
otherwise, if $q$ is a common eventual period of the finitely many residue
words $x^{(s)}$, then $Tq$ is an eventual period of $x$. Fix such an $s$ and
put $z:=x^{(s)}$. Without loss of generality we can assume that $0^{\infty
}\in \mathcal{O}(z).$ Since $z$ is aperiodic, both letters occur infinitely
often in $z$. For every finite pattern $Q$ and every $n\in \mathbb{N}_{0}$,
one has $z[n+Q]=x[Tn+s+TQ]$ and $\sharp (TQ)=\sharp Q$. Hence Proposition %
\ref{prop:zerolower} gives $p_{x}^{\ast \mathrm{ab}}(k)\geq p_{z}^{\ast 
\mathrm{ab}}(k)\geq f(k)$ for all $k\in \mathbb{N}$.
\end{proof}

\begin{corollary}
\label{cor:universalbinary} If $\alpha \in \mathbb{A}^{\mathbb{N}_{0}}$ is
aperiodic, then $p_{\alpha }^{\ast \mathrm{ab}}(k)\geq f(k)$\ for all $k\in 
\mathbb{N}.$
\end{corollary}

\begin{proof}
By Lemma \ref{lem:projection}, $\pi _{a}(\alpha )$ is aperiodic for some $%
a\in \mathbb{A}$. By Proposition \ref{prop:binarylower} and Lemma \ref%
{lem:transfer}(3), one has $p_{\alpha }^{\ast \mathrm{ab}}(k)\geq p_{\pi
_{a}(\alpha )}^{\ast \mathrm{ab}}(k)\geq f(k)$\ for all $k\in \mathbb{N}.$
\end{proof}

We next construct a binary word giving the matching upper bound for $%
\mathcal{L}_{2}^{\mathrm{ab}}$. This construction will be revisited in
Section \ref{sec:APSW}.

Let 
\begin{equation}
E:=\{3^{j}:j\in \mathbb{N}_{0}\}\text{\ and }x:=\mathbf{1}_{E}\in \{0,1\}^{%
\mathbb{N}_{0}}.  \label{eq:def-sparse-bi-core}
\end{equation}%
Both $0$ and $1$ occur infinitely often in $x$. Moreover, the gaps between
successive occurrences of $1$ are $3^{j+1}-3^{j}=2\cdot 3^{j}$ and hence
tend to infinity. A periodic tail containing $1$ would have bounded gaps
between occurrences of $1$. Thus $x$ is aperiodic, and $x\in \mathcal{A}_{2}$%
.

\begin{lemma}
\label{lem:forest} Let $R,S\subset \mathbb{N}_{0}$ be finite. Then $\Gamma
_{E}[R,S]$ is a forest.
\end{lemma}

\begin{proof}
Suppose that $\Gamma _{E}[R,S]$ contains a cycle. For each edge $(r,s)$ of
the cycle, the integer $r+s$ is a power of $3$. Choose an edge $(r,s)$ for
which $M:=r+s$ is maximal among the edge sums on the cycle.

Let $(r,s^{\prime })$ be the other cycle edge incident with $r$. By $%
s^{\prime }\neq s$ and maximality of $M$, one has $r+s^{\prime }<M$. Both $%
r+s^{\prime }$ and $M$ are powers of $3$, and therefore $r+s^{\prime }\leq
M/3$. In particular, $r\leq M/3$. Similarly, if $(r^{\prime },s)$ is the
other cycle edge incident with $s$, then $r^{\prime }+s\leq M/3$, and hence $%
s\leq M/3$. Consequently, $M=r+s\leq \frac{2M}{3},$ a contradiction. Thus $%
\Gamma _{E}[R,S]$ is a forest.
\end{proof}

\begin{proposition}
\label{prop:forest-upper} Let $D\subset \mathbb{N}_{0}$ be infinite. If $%
\Gamma _{D}$ is a forest, then $p_{\mathbf{1}_{D}}^{\ast \mathrm{ab}}(k)\leq
f(k)$ for every $k\in \mathbb{N}.$
\end{proposition}

\begin{proof}
Fix a $k$-pattern $S$. If $c_{S}$ has no positive value, then $p_{\mathbf{1}%
_{D}}^{\mathrm{ab}}(S)=1\leq f(k).$ Assume that $r\geq 1$. Let $1\leq
b_{1}<\cdots <b_{r}$ be the distinct positive values of $c_{S}$, and choose $%
n_{i}\in \mathbb{N}_{0}$ such that $c_{S}(n_{i})=b_{i}$\ for $1\leq i\leq r.$
Let $R:=\{n_{1},\ldots ,n_{r}\}.$ Since $\Gamma _{D}$ is a forest, so is $%
\Gamma _{D}[R,S]$. Its number of edges is $\sharp E(\Gamma
_{D}[R,S])=\sum_{i=1}^{r}b_{i}.$ Moreover, a forest on $r+k$ vertices has at
most $r+k-1$ edges. Hence%
\begin{equation*}
\frac{r(r+1)}{2}\leq \sum_{i=1}^{r}b_{i}=\sharp E(\Gamma _{D}[R,S])\leq
r+k-1.
\end{equation*}%
Thus $1+\binom{r}{2}\leq k.$ By (\ref{eq:threshold-integers}), one has $%
r+1\leq f(k).$

The positive counts contribute exactly $r$ Abelian classes, and the count $0$%
, if it occurs, contributes at most one further class. Therefore $p_{\mathbf{%
1}_{D}}^{\mathrm{ab}}(S)\leq f(k).$ Since $S$ was arbitrary, it follows that 
$p_{\mathbf{1}_{D}}^{\ast \mathrm{ab}}(k)\leq f(k).$
\end{proof}

\begin{corollary}
\label{cor:sparse-upper} Let $x$ be the word defined in (\ref%
{eq:def-sparse-bi-core}). Then $p_{x}^{\ast \mathrm{ab}}(k)=f(k)$ for all $%
k\in \mathbb{N}$. In particular, $\mathcal{L}_{2}^{\mathrm{ab}}(k)=f(k)\leq 
\sqrt{2k}+2$.
\end{corollary}

\begin{proof}
By Lemma \ref{lem:forest}, $\Gamma _{E}$ is a forest. Hence Proposition \ref%
{prop:forest-upper} gives $p_{x}^{\ast \mathrm{ab}}(k)\leq f(k)$. By
Corollary \ref{cor:universalbinary}, we have $f(k)\leq \mathcal{L}_{2}^{%
\mathrm{ab}}(k)\leq p_{x}^{\ast \mathrm{ab}}(k).$ Therefore $p_{x}^{\ast 
\mathrm{ab}}(k)=\mathcal{L}_{2}^{\mathrm{ab}}(k)=f(k)$ for all $k\in \mathbb{%
N}$.
\end{proof}

\begin{proof}[Proof of Theorem~\protect\ref{thm:threshold}]
Suppose that $\alpha _{n+q}=\alpha _{n}$ for all $n\geq N$, where $N\in 
\mathbb{N}_{0}$ and $q\in \mathbb{N}$. By Lemma \ref{lem:transfer}(4) one
has $p_{\alpha }^{\ast \mathrm{ab}}(k)\leq N+q$ for every $k$. Thus $%
(1)\Rightarrow (2)$. Suppose next that $M:=\sup_{k\in \mathbb{N}}p_{\alpha
}^{\ast \mathrm{ab}}(k)<\infty .$ If $k>\binom{M}{2}$, then $\binom{%
p_{\alpha }^{\ast \mathrm{ab}}(k)}{2}\leq \binom{M}{2}<k,$ so $%
(2)\Rightarrow (3)$. And $(3)\Rightarrow (1)$ is the contrapositive of
Corollary\ \ref{cor:universalbinary} together with the definition of $f.$
\end{proof}

\section{Abelian pattern Sturmian words}

\label{sec:APSW}

In this section, we characterize Abelian pattern Sturmian words in terms of
incidence graphs and examine the relation between the forest condition and
the Sidon property. In the first subsection, we prove Theorem \ref%
{thm:main-APS}.

\subsection{The incidence graph characterization}

\ 

\begin{lemma}
\label{lem:binarySturmian} Every Abelian pattern Sturmian word uses exactly
two letters, both occurring infinitely often, and is aperiodic.
\end{lemma}

\begin{proof}
Let $x$ be an Abelian pattern Sturmian word. Since $p_{x}^{\ast \mathrm{ab}%
}(1)=f(1)=2,$ the word $x$ uses exactly two letters. As $p_{x}^{\ast \mathrm{%
ab}}(k)=f(k)\rightarrow \infty ,$ Theorem \ref{thm:threshold} implies that $%
x $ is aperiodic. If one of its two letters occurred only finitely often,
then $x$ would be eventually constant. Thus both letters occur infinitely
often.
\end{proof}

\begin{lemma}
\label{lem:nonconstant-periodic} Let $x\in \{0,1\}^{\mathbb{N}_{0}}$ be
aperiodic. If $\mathcal{O}(x)$ contains a nonconstant periodic point, then $%
p_{x}^{\ast \mathrm{ab}}(3)=4.$
\end{lemma}

\begin{proof}
Let $\beta \in \mathcal{O}(x)$ be a nonconstant periodic point, and let $T$
be a period of $\beta $. By Lemma \ref{lem:alignment}, there exists a $T$%
-periodic word $\gamma \in \{0,1\}^{\mathbb{Z}}$ such that $\gamma
_{s}^{\infty }\in \mathcal{O}(x^{(s)})$ for every $0\leq s<T$. At least one
of the words $x^{(s)}$ is aperiodic. Fix such an $s$ and put $z=x^{(s)}$.
After interchanging the two letters if necessary, we may assume that $%
0^{\infty }\in \mathcal{O}(z).$ Since $z$ is aperiodic, both letters occur
infinitely often in $z$.

Choose $u<v$ such that $z_{u}=z_{v}=1.$ By Lemma \ref{lem:zero-blocks}(1),
there exists $Y>v$ such that $z_{Y}=1$ and $z_{Y-t}=0$ for $(1\leq t\leq
v-u).$ Set $p=Y-v.$ Then one has $z_{u+p}=0$ and $z_{v+p}=1.$ By Lemma \ref%
{lem:zero-blocks}(3), take $r>p$ such that $z_{u+r}=z_{v+r}=0.$ Hence, for $%
Q:=\{0,p,r\},$ the numbers of $1$'s in $z[u+Q]$ and $z[v+Q]$ are
respectively $1$ and $2$.

Now set $S:=TQ=\{0,Tp,Tr\}.$ Since $z_{n}=x_{Tn+s}$, the numbers of $1$'s in 
$x[(Tu+s)+S]$ and $x[(Tv+s)+S]$ are $1$ and $2$, respectively.

Since $\beta $ is nonconstant and $T$-periodic, there are indices $a,b$ such
that $\beta _{a}=0$ and $\beta _{b}=1.$ As $\beta \in \mathcal{O}(x)$, these
finite patterns also occur in $x$. Thus $\sharp \mathcal{C}_{x}(S)=4$. By (%
\ref{eq:binary-abelian-count})--(\ref{eq:binary-upper-bound}) one has $%
p_{x}^{\ast \mathrm{ab}}(3)=4.$
\end{proof}

\begin{proposition}
\label{prop:background} Let $x$ be binary and aperiodic. If $p_{x}^{\ast 
\mathrm{ab}}(3)\leq 3$, then, after exchanging the letters if necessary, $%
0^{\infty }$ is the unique periodic point of $\mathcal{O}(x)$.
\end{proposition}

\begin{proof}
By Lemma \ref{lem:noperiodic} and the hypothesis $p_{x}^{\ast \mathrm{ab}%
}(3)\leq 3$, the orbit closure $\mathcal{O}(x)$ contains a periodic point.
Lemma \ref{lem:nonconstant-periodic} implies that this point is constant.
Exchanging the letters if necessary, we may assume that $0^{\infty }\in 
\mathcal{O}(x)$.

Suppose that $1^{\infty }\in \mathcal{O}(x)$ as well. Then both $000$ and $%
111$ occur in $x$. Let $c(n):=x_{n}+x_{n+1}+x_{n+2}.$ Since $%
c(n+1)-c(n)=x_{n+3}-x_{n}$, we have $|c(n+1)-c(n)|\leq 1.$ As $c$ attains
both $0$ and $3$, it also attains $1$ and $2$. Therefore $p_{x}^{\ast 
\mathrm{ab}}(3)=4,$ a contradiction. Hence $1^{\infty }\notin \mathcal{O}(x)$%
. By Lemma \ref{lem:nonconstant-periodic}, $\mathcal{O}(x)$ contains no
nonconstant periodic point. Hence $0^{\infty }$ is its unique periodic point.
\end{proof}

\begin{theorem}
\label{thm:cycle} Let $x=\mathbf{1}_{D}$, where $D\subset \mathbb{N}_{0}$ is
infinite and $0^{\infty }\in \mathcal{O}(x).$ Suppose that $\Gamma _{D}$
contains a cycle of length $2r$, where $r\geq 2$. Then 
\begin{equation*}
p_{x}^{\ast \mathrm{ab}}\left( \binom{r+1}{2}\right) \geq r+2>f\left( \binom{%
r+1}{2}\right) =r+1.
\end{equation*}
\end{theorem}

\begin{proof}
Let $R$ and $S_{0}$ be the sets of left and right vertices, respectively, of
such a cycle. Then $\sharp R=\sharp S_{0}=r.$ Each vertex of $R$ is incident
with two edges of the cycle. Hence 
\begin{equation*}
c_{S_{0}}(u)=\sharp \left( (u+S_{0})\cap D\right) \geq 2\text{\ }(u\in R).
\end{equation*}

Applying Lemma \ref{lem:separation} to $R$ with $b=2$, we obtain a pattern $%
S_{1}\supset S_{0}$ such that the $r$ counts $c_{S_{1}}(u)$\ $(u\in R)$ are
pairwise distinct and at least $2$, and 
\begin{equation*}
\sharp S_{1}\leq \sharp S_{0}+\binom{r}{2}=r+\binom{r}{2}=\binom{r+1}{2}.
\end{equation*}%
By the last assertion of Lemma \ref{lem:separation}, $S_{1}$ can be extended
to a pattern $S$ of cardinality $\sharp S=\binom{r+1}{2}$ without changing
these $r$ counts.

By Lemma \ref{lem:zero-blocks}(2), the values $0$ and $1$ are also attained
by $c_{S}$. Therefore one has 
\begin{equation*}
p_{x}^{\ast \mathrm{ab}}\left( \binom{r+1}{2}\right) \geq p_{x}^{\mathrm{ab}%
}(S)=\sharp \mathcal{C}_{x}(S)\geq r+2>f\left( \binom{r+1}{2}\right) =r+1,
\end{equation*}%
which proves the result.
\end{proof}

A set $D\subset \mathbb{N}_{0}$\ is called a \emph{Sidon set} if all
positive differences $b-a$ with $a,b\in D$ and $a<b$ are distinct.

\begin{lemma}
\label{lem:sidon} If $\Gamma _{D}$ is a forest, then $D$ is a Sidon set.
Moreover, if $D=\{d_{0}<d_{1}<\cdots \}$ is an infinite Sidon set, then 
\begin{equation}
d_{n+1}-d_{n}\longrightarrow \infty .  \label{eq:gaps}
\end{equation}%
For every $M\in \mathbb{N}_{0}$ and $L\in \mathbb{N}$, if $h=\mathcal{\sharp 
}(D\cap \lbrack M,M+L)),$ then 
\begin{equation}
\binom{h}{2}\leq L-1.  \label{eq:density}
\end{equation}%
Consequently, $\mathbf{1}_{D}$ is aperiodic and $0^{\infty }\in \mathcal{O}(%
\mathbf{1}_{D}).$
\end{lemma}

\begin{proof}
Suppose that $D$ is not Sidon. Then there exist $t>0$ and distinct $a,b\in D$
such that $a+t,\ b+t\in D.$ The left vertices $0,t$ and the right vertices $%
a,b$ then form a $4$-cycle in $\Gamma _{D}$, since the corresponding edge
sums are $a,\ a+t,\ b,\ b+t\in D.$ This contradicts the assumption that $%
\Gamma _{D}$ is a forest.

Now suppose that $D=\{d_{0}<d_{1}<\cdots \}$ is infinite. Fix $L\in \mathbb{N%
}$. If $d_{n+1}-d_{n}\leq L,$ then $d_{n+1}-d_{n}\in \{1,\ldots ,L\}$. Since 
$D$ is Sidon, each of these positive differences can occur for at most one
index $n$. Thus there are at most $L$ indices for which $d_{n+1}-d_{n}\leq
L. $ Since $L$ is arbitrary, this proves\ (\ref{eq:gaps}).

Let $M\in \mathbb{N}_{0}$ and $L\in \mathbb{N}$, and write $D\cap \lbrack
M,M+L)=\{e_{1}<\cdots <e_{h}\}.$ The $\binom{h}{2}$ positive differences $%
e_{j}-e_{i}$ $(1\leq i<j\leq h)$ are pairwise distinct by the Sidon
property. Since $1\leq e_{j}-e_{i}\leq L-1,$ they all belong to $\{1,\ldots
,L-1\}$. Hence $\binom{h}{2}\leq L-1,$ which proves\ (\ref{eq:density}).

Finally,\ (\ref{eq:gaps}) implies that the intervals $[d_{n}+1,d_{n+1})$ are
zero blocks of $\mathbf{1}_{D}$ whose lengths tend to infinity, and hence $%
0^{\infty }\in \mathcal{O}(\mathbf{1}_{D}).$ Since $D$ is infinite, $\mathbf{%
1}_{D}$ contains infinitely many $1$'s. An eventually periodic binary word
with infinitely many $1$'s has bounded gaps between successive occurrences
of $1$, which contradicts (\ref{eq:gaps}). Thus $\mathbf{1}_{D}$ is
aperiodic.
\end{proof}

\begin{proof}[Proof of Theorem~\protect\ref{thm:main-APS}]
Suppose first that $x$ is Abelian pattern Sturmian. By Lemma\ \ref%
{lem:binarySturmian}, $x$ is binary and aperiodic. Since $p_{x}^{\ast 
\mathrm{ab}}(3)=f(3)=3,$ Proposition \ref{prop:background} allows us, after
interchanging the two letters if necessary, to assume that $0^{\infty }\in 
\mathcal{O}(x).$ Write $x=\mathbf{1}_{D}.$ Since both letters occur
infinitely often in $x$, the set $D$ is infinite.

Suppose that $\Gamma _{D}$ contains a cycle of length $2r$, with $r\geq 2$.
Then Theorem\ \ref{thm:cycle} gives 
\begin{equation*}
p_{x}^{\ast \mathrm{ab}}\left( \binom{r+1}{2}\right) \geq r+2>r+1=f\left( 
\binom{r+1}{2}\right) ,
\end{equation*}%
which contradicts the defining equality $p_{x}^{\ast \mathrm{ab}}(k)=f(k)$
for all $k\in \mathbb{N}.$ Hence $\Gamma _{D}$ is a forest.

Conversely, let $D\subset \mathbb{N}_{0}$ be infinite and suppose that $%
\Gamma _{D}$ is a forest. By Lemma\ \ref{lem:sidon}, one has $0^{\infty }\in 
\mathcal{O}(\mathbf{1}_{D}),$ and $\mathbf{1}_{D}$ is aperiodic. Since $D$
is infinite, $1$ occurs infinitely often in $\mathbf{1}_{D}$. Therefore
Proposition \ref{prop:zerolower} gives $p_{\mathbf{1}_{D}}^{\ast \mathrm{ab}%
}(k)\geq f(k)$\ for all $k\in \mathbb{N}.$ On the other hand, Proposition %
\ref{prop:forest-upper} gives the upper bound $p_{\mathbf{1}_{D}}^{\ast 
\mathrm{ab}}(k)\leq f(k)$\ for all $k\in \mathbb{N}.$ Thus $p_{\mathbf{1}%
_{D}}^{\ast \mathrm{ab}}(k)=f(k)$\ for all $k\in \mathbb{N}.$
\end{proof}

\begin{corollary}
\label{cor:ordinary} Every Abelian pattern Sturmian word is a pattern
Sturmian word.
\end{corollary}

\begin{proof}
Let $x$ be an Abelian pattern Sturmian word. By Theorem \ref{thm:main-APS},
we may assume that $x=\mathbf{1}_{D},$ where $\Gamma _{D}$ is a forest.

Fix a $k$-pattern $S$. For $n\in \mathbb{N}_{0}$, define $N_{S}(n):=\{s\in
S:n+s\in D\}.$ Then $x[n+S]$ is uniquely determined by $N_{S}(n)$, and hence 
$p_{x}(S)=\sharp \{N_{S}(n):n\in \mathbb{N}_{0}\}.$

Let $R\subset \mathbb{N}_{0}$ be a subset such that the map $n\longmapsto
N_{S}(n)$ restricts to a bijection from $R$ onto $\{N_{S}(n):n\in \mathbb{N}%
_{0},\ N_{S}(n)\neq \varnothing \}.$ Let $a$ be the number of indices $n\in
R $ for which $\sharp N_{S}(n)=1$, and let $b$ be the number of those for
which $\sharp N_{S}(n)\geq 2$. Since the sets $N_{S}(n)$ of cardinality $1$
are distinct subsets of $S$, one has $a\leq k.$ Moreover, the degree of a
left vertex $n\in R$ in $\Gamma _{D}[R,S]$ is $\mathcal{\sharp }N_{S}(n)$.
Therefore $\sharp E(\Gamma _{D}[R,S])\geq a+2b.$ Since $\Gamma _{D}[R,S]$ is
a forest on $\sharp R+\sharp S=a+b+k$ vertices, it follows that $\sharp
E(\Gamma _{D}[R,S])\leq a+b+k-1.$ Consequently, we have $b\leq k-1.$

There is at most one further pattern, corresponding to the empty set $%
N_{S}(n)=\varnothing $. Thus $p_{x}(S)\leq a+b+1\leq 2k.$ Since $S$ was
arbitrary, one has $p_{x}^{\ast }(k)\leq 2k.$ On the other hand, $x$ is
aperiodic, and hence (\ref{KZ2002}) gives $p_{x}^{\ast }(k)\geq 2k$ for
every $k\in \mathbb{N}$. Consequently, $p_{x}^{\ast }(k)=2k$ for all $k\in 
\mathbb{N},$ and therefore $x$ is a pattern Sturmian word.
\end{proof}

\begin{corollary}
No Abelian pattern Sturmian word is recurrent.
\end{corollary}

\begin{proof}
If an Abelian pattern Sturmian word $x$ were recurrent, then by Lemma \ref%
{lem:binarySturmian} it would be binary and aperiodic. Theorem \ref%
{thm:KWZ2015} would then give $p_{x}^{\ast \mathrm{ab}}(3)=4.$ This
contradicts $p_{x}^{\ast \mathrm{ab}}(3)=f(3)=3.$
\end{proof}

\subsection{Sidon sets and examples}

\ 

\begin{proposition}
\label{prop:three} Let $x\in \{0,1\}^{\mathbb{N}_{0}}$ be aperiodic. Then $%
p_{x}^{\ast \mathrm{ab}}(3)=3$ if and only if, after exchanging the letters
if necessary, $x=\mathbf{1}_{D}$ for an infinite Sidon set $D\subset \mathbb{%
N}_{0}$.
\end{proposition}

\begin{proof}
Suppose first that $p_{x}^{\ast \mathrm{ab}}(3)=3$. After interchanging the
two letters if necessary, Proposition \ref{prop:background} gives $0^{\infty
}\in \mathcal{O}(x)$. Let $D:=\{n:x_{n}=1\}$. Since $x$ is aperiodic, $D$ is
infinite. If $D$ is not Sidon, then there exists a $4$-cycle, as in the
proof of Lemma\ \ref{lem:sidon}. Then Theorem \ref{thm:cycle} with $r=2$
gives $p_{x}^{\ast \mathrm{ab}}(3)\geq 4,$ a contradiction. Thus $D$ is
Sidon.

Conversely, let $x=\mathbf{1}_{D},$\ where $D$ is infinite and Sidon. By
Lemma\ \ref{lem:sidon}, one has $0^{\infty }\in \mathcal{O}(x).$ Let $%
S:=\{s_{1}<s_{2}<s_{3}\}$ be a $3$-pattern. We claim that the counts $2$ and 
$3$ cannot both occur on $S$. Indeed, suppose $c_{S}(n)=3$\ and\ $c_{S}(m)=2$
for some $n,m\in \mathbb{N}_{0}.$ Take $i<j$ such that $%
x_{m+s_{i}}=x_{m+s_{j}}=1.$ Since $c_{S}(n)=3$, we also have $%
x_{n+s_{i}}=x_{n+s_{j}}=1.$ Thus 
\begin{equation*}
s_{j}-s_{i}=(n+s_{j})-(n+s_{i})=(m+s_{j})-(m+s_{i})
\end{equation*}%
gives two distinct representations of the same positive difference by
elements of $D$, which contradicts the Sidon property. Hence $p_{x}^{\mathrm{%
ab}}(S)\leq 3$ for every $S\in \Sigma _{3}(\mathbb{N}_{0})$, and therefore $%
p_{x}^{\ast \mathrm{ab}}(3)\leq 3.$

Let $S\subset D$ with $\sharp S=3$. Then $|x[S]|_{1}=3,$ while Lemma \ref%
{lem:zero-blocks}(2) gives starting positions on the same pattern $S$ with
counts $0$ and $1$. Thus we have $p_{x}^{\mathrm{ab}}(S)=3,$ and
consequently $p_{x}^{\ast \mathrm{ab}}(3)=3.$
\end{proof}

Define the \emph{upper Banach density} of $D\subset \mathbb{N}_{0}$ by 
\begin{equation*}
d^{\ast }(D)=\limsup_{L\rightarrow \infty }\sup_{M\in \mathbb{N}_{0}}\frac{%
\sharp (D\cap \lbrack M,M+L))}{L}.
\end{equation*}%
Equation\ (\ref{eq:density}) implies $d^{\ast }(D)=0$ for every Sidon set $D$%
. Consequently, if $x$ is an Abelian pattern Sturmian word, then, after
interchanging the two letters if necessary, one has $d^{\ast }(\{n\in 
\mathbb{N}_{0}:x_{n}=1\})=0.$ The Sidon condition is strictly weaker than
the forest condition in Theorem \ref{thm:main-APS}, as the examples below
show.

\begin{proposition}
\label{prop:lacunary} Let $D=\{d_{0}<d_{1}<\cdots \}\subset \mathbb{N}_{0}$
be infinite. If $d_{n+1}>2d_{n}$ for every $n$, then $\Gamma _{D}$ is a
forest. Consequently, $\mathbf{1}_{D}$ is Abelian pattern Sturmian.
\end{proposition}

\begin{proof}
Suppose that $\Gamma _{D}$ contains a cycle. Take an edge $(r,s)$ of the
cycle such that $M:=r+s$ is maximal among the sums corresponding to its
edges. Let $(r,s^{\prime })$ and $(r^{\prime },s)$ be the other two cycle
edges incident to $r$ and $s$, respectively. By maximality, it follows that $%
r+s^{\prime }<M$ and $r^{\prime }+s<M.$ Since all three sums belong to $D$,
the hypothesis implies $r+s^{\prime }<\frac{M}{2}$ and $r^{\prime }+s<\frac{M%
}{2}.$ It follows that $r<\frac{M}{2}$ and $s<\frac{M}{2},$ which
contradicts $r+s=M$. Hence $\Gamma _{D}$ is a forest. The last assertion
follows from Theorem\ \ref{thm:main-APS}.
\end{proof}

\begin{example}
\label{ex:powers} The word $\mathbf{1}_{\{3^{j}:j\geq 0\}}$ is Abelian
pattern Sturmian by Proposition \ref{prop:lacunary} or Corollary \ref%
{cor:sparse-upper}. In contrast, let $F:=\{2^{j}:j\geq 0\}$ and $y:=\mathbf{1%
}_{F}$. Suppose that $2^{b}-2^{a}=2^{d}-2^{c}$ with $b>a$ and $d>c$. Then $%
2^{a}(2^{b-a}-1)=2^{c}(2^{d-c}-1).$ Since both parenthesized factors are
odd, we obtain $a=c$, and hence $b=d$. Thus $F$ is a Sidon set. Since $y$ is
aperiodic, Proposition \ref{prop:three} gives $p_{y}^{\ast \mathrm{ab}}(3)=3$%
.

However, in $\Gamma _{F}[R,Q]$, where $R=Q=\{0,2,4\},$ the edges 
\begin{equation*}
(0,2),\text{\ }(2,2),\text{\ }(2,0),\text{\ }(4,0),\text{\ }(4,4),\text{\ }%
(0,4)
\end{equation*}%
form a $6$-cycle. Hence Theorem\ \ref{thm:cycle} yields $p_{y}^{\ast \mathrm{%
ab}}(6)\geq 5>f(6)=4$. Therefore, $y$ is not Abelian pattern Sturmian.
\end{example}

\section{General lower bounds for aperiodic words}

\label{sec:general-lower-AW}

We now establish a general lower bound for aperiodic words that reflects the
size of the alphabet. Together with a matching upper bound, this yields the
sharp square-root asymptotic of the extremal complexity.

\subsection{Canonical partition and residue structure}

\label{sec:partition}

\subsubsection{Eventually periodic alphabet partitions}

\ 

A \emph{partition} of $\mathbb{A}$ is a family of pairwise disjoint nonempty
subsets whose union is $\mathbb{A}$. Its members are called \emph{blocks}.

\begin{definition}
Let $\mathcal{P}$ be a partition of $\mathbb{A}$. Denote by 
\begin{equation*}
\pi_{\mathcal{P}}:\mathbb{A}\longrightarrow\mathcal{P}
\end{equation*}
the canonical projection, where $\pi_{\mathcal{P}}(a)$ is the unique block
of $\mathcal{P}$ containing $a$. The \emph{quotient word} of $\alpha\in%
\mathbb{A}^{\mathbb{N}_0}$ induced by $\mathcal{P}$ is 
\begin{equation*}
\pi_{\mathcal{P}}(\alpha):= \pi_{\mathcal{P}}(\alpha_0)\pi_{\mathcal{P}%
}(\alpha_1)\pi_{\mathcal{P}}(\alpha_2)\cdots \in\mathcal{P}^{\mathbb{N}_0}.
\end{equation*}
We call $\mathcal{P}$ an \emph{eventually periodic alphabet partition} of $%
\alpha$ if $\pi_{\mathcal{P}}(\alpha)$ is eventually periodic.
\end{definition}

For two partitions $\mathcal{P}$ and $\mathcal{Q}$ of $\mathbb{A}$, we say
that $\mathcal{P}$ \emph{refines} $\mathcal{Q}$, and write $\mathcal{P}%
\preceq \mathcal{Q}$, if every block of $\mathcal{P}$ is contained in a
block of $\mathcal{Q}$. The common refinement of $\mathcal{P}$ and $\mathcal{%
Q}$ is 
\begin{equation*}
\mathcal{P}\wedge \mathcal{Q}:=\{B\cap C:B\in \mathcal{P},\text{\ }C\in 
\mathcal{Q},\text{\ }B\cap C\neq \varnothing \}.
\end{equation*}

\begin{lemma}
\label{lem:finest-partition} Let $\alpha\in\mathbb{A}^{\mathbb{N}_0}$. Then
there exists a unique eventually periodic alphabet partition $\mathcal{P}%
_{*} $ of $\alpha$ such that $\mathcal{P}_{*}\preceq\mathcal{P}$ for every
eventually periodic alphabet partition $\mathcal{P}$ of $\alpha$.
\end{lemma}

\begin{proof}
Let $\mathfrak{E}$ denote the family of eventually periodic alphabet
partitions of $\alpha$. The family $\mathfrak{E}$ is nonempty, as the
trivial partition $\{\mathbb{A}\}$ is in $\mathfrak{E}$.

We first show that $\mathfrak{E}$ is closed under common refinement. Let $%
\mathcal{P},\mathcal{Q}\in \mathfrak{E}$. Suppose that $\pi _{\mathcal{P}%
}(\alpha )$ and $\pi _{\mathcal{Q}}(\alpha )$ are eventually periodic with
eventual periods $p$ and $q$, respectively. Set $r:=\func{lcm}(p,q)$. Then
for all sufficiently large $n$, we have 
\begin{equation*}
\pi _{\mathcal{P}}(\alpha _{n+r})=\pi _{\mathcal{P}}(\alpha _{n})\text{ and\ 
}\pi _{\mathcal{Q}}(\alpha _{n+r})=\pi _{\mathcal{Q}}(\alpha _{n}).
\end{equation*}%
Hence $\alpha _{n+r}$ and $\alpha _{n}$ lie in the same block of $\mathcal{P}
$ and in the same block of $\mathcal{Q}$. Therefore they lie in the same
block of $\mathcal{P}\wedge \mathcal{Q}$, and thus 
\begin{equation*}
\pi _{\mathcal{P}\wedge \mathcal{Q}}(\alpha _{n+r})=\pi _{\mathcal{P}\wedge 
\mathcal{Q}}(\alpha _{n})
\end{equation*}%
for all sufficiently large $n$. Hence $\mathcal{P}\wedge \mathcal{Q}\in 
\mathfrak{E}$.

Since $\mathbb{A}$ is finite, $\mathfrak{E}$ is finite. Let $\mathcal{P}%
_{\ast }$ be the common refinement of all partitions in $\mathfrak{E}$. By
the preceding argument, $\mathcal{P}_{\ast }\in \mathfrak{E}$. By the
definition of common refinement, one has $\mathcal{P}_{\ast }\preceq 
\mathcal{P}\ $for each $\mathcal{P}\in \mathfrak{E}.$

Finally, suppose that $\mathcal{Q}_{\ast }\in \mathfrak{E}$ also satisfies $%
\mathcal{Q}_{\ast }\preceq \mathcal{P}$ for every $\mathcal{P}\in \mathfrak{E%
}$. Then, since $\mathcal{P}_{\ast },\mathcal{Q}_{\ast }\in \mathfrak{E}$,
we obtain $\mathcal{P}_{\ast }\preceq \mathcal{Q}_{\ast }$\ and\ $\mathcal{Q}%
_{\ast }\preceq \mathcal{P}_{\ast },$ which implies $\mathcal{P}_{\ast }=%
\mathcal{Q}_{\ast }$.
\end{proof}

Let $\mathcal{P}_{\ast }$ be the partition given by Lemma \ref%
{lem:finest-partition}. Write 
\begin{equation*}
c_{\ast }:=\sharp \mathcal{P}_{\ast },\text{\ }\eta :=\pi _{\mathcal{P}%
_{\ast }}(\alpha ).
\end{equation*}%
Let $q_{\ast }\geq 1$ be the least positive eventual period of $\eta $.

\begin{lemma}
\label{lem:cq} If $\alpha\in\mathcal{A}_\ell$, then $1\le c_{*}\le\ell-1$
and $q_{*}\ge c_{*}$.
\end{lemma}

\begin{proof}
Since $\mathcal{P}_{\ast }$ is a partition of the $\ell $-letter alphabet $%
\mathbb{A}$, we have $1\leq c_{\ast }\leq \ell $. If $c_{\ast }=\ell $, then
every block of $\mathcal{P}_{\ast }$ is a singleton, so $\pi _{\mathcal{P}%
_{\ast }}:\mathbb{A}\rightarrow \mathcal{P}_{\ast }$ is bijective. Since $%
\eta $ is eventually periodic, so is $\alpha .$ This contradicts $\alpha \in 
\mathcal{A}_{\ell }$. Hence $1\leq c_{\ast }\leq \ell -1$.

For each $B\in\mathcal{P}_*$, choose $a\in B$. Since every letter of $%
\mathbb{A}$ occurs infinitely often in $\alpha$, the symbol $B$ occurs
infinitely often in $\eta$. Thus every symbol of $\mathcal{P}_*$ occurs
infinitely often in $\eta$.

Take $N$ such that $\eta _{n+q_{\ast }}=\eta _{n}$ for all $n\geq N$. Every
symbol occurring infinitely often in $\eta $ must occur among $\eta
_{N},\eta _{N+1},\ldots ,\eta _{N+q_{\ast }-1}$. Hence these $q_{\ast }$
positions contain all $c_{\ast }$ symbols of $\mathcal{P}_{\ast }$, and
therefore $q_{\ast }\geq c_{\ast }$.
\end{proof}

Put 
\begin{equation}
r_{\ast }:=\ell -c_{\ast },\text{\ }g:=\ell -1=c_{\ast }+r_{\ast }-1.
\label{eq:canonical}
\end{equation}%
Thus $r_{\ast }\geq 1$. We retain these notations throughout the
construction below.

\subsubsection{Residue alphabets}

\ 

\label{subsec:Residue-alph}

Assume throughout Subsections \ref{subsec:Residue-alph}--\ref{subsec:grouped}
that $\mathcal{O}(\alpha )$ contains a periodic point. Let $T$ be a common
multiple of $q_{\ast }$ and a period of such a periodic point, and let $%
\gamma $ be the $T$-periodic word given by Lemma \ref{lem:alignment}. For $%
0\leq s<T$, let 
\begin{equation*}
\mathbb{A}_{s}:=\{a\in \mathbb{A}:a\text{ occurs infinitely often in }\alpha
^{(s)}\}.
\end{equation*}%
Delete a sufficiently long prefix whose length is a multiple of $T$. By
Lemma \ref{lem:transfer}(2), any lower bound for this tail also holds for
the original word. We continue to write $\alpha $ for the resulting tail.
Thus we may assume, throughout Subsections \ref{subsec:Residue-alph}--\ref%
{subsec:grouped}, that the quotient $\pi _{\mathcal{P}_{\ast }}(\alpha )$ is
periodic, and every letter of $\alpha ^{(s)}$ belongs to $\mathbb{A}_{s}$.
The partition $\mathcal{P}_{\ast }$ and the least eventual period $q_{\ast }$%
\ are unchanged by deleting a finite prefix. Moreover, because the deleted
length is a multiple of $T,$ the residue classes modulo $T$ are unchanged,
and hence no rephasing of $\gamma $ is needed.

By Lemma \ref{lem:alignment}, the matching intervals between $\alpha $ and $%
\gamma $ have arbitrarily large lengths and occur arbitrarily far to the
right. Hence, after the above finite deletion, there is still a matching
interval of length at least $T.$ The quotient words $\pi _{\mathcal{P}_{\ast
}}(\gamma )$ and $\pi _{\mathcal{P}_{\ast }}(\alpha )$ are both $T$-periodic
and agree on this interval. They therefore agree at every nonnegative
position. In particular, every block of $\mathcal{P}_{\ast }$ occurs in $\pi
_{\mathcal{P}_{\ast }}(\gamma )$.

\begin{lemma}
\label{lem:residuegraph}Let $G$ be the simple undirected graph with vertex
set $\mathbb{A}$ in which two distinct letters are adjacent whenever they
belong to a common set $\mathbb{A}_{s}$. Then the connected components of $G$
are precisely the blocks of $\mathcal{P}_{\ast }$.
\end{lemma}

\begin{proof}
Every letter of $\mathbb{A}$ belongs to some $\mathbb{A}_{s}$, since it
occurs infinitely often in $\alpha $ and there are only finitely many
residue classes modulo $T$. Let $\mathcal{Q}$ be the partition of $\mathbb{A}
$ into the connected components of $G.$

Since $q_{\ast }\mid T$, each $\mathbb{A}_{s}$ is contained in a single
block of $\mathcal{P}_{\ast }$, and hence $\mathcal{Q}\preceq \mathcal{P}%
_{\ast }$. On the other hand, each $\mathbb{A}_{s}$ lies in a single
component of $G$, and every letter of $\alpha ^{(s)}$ belongs to $\mathbb{A}%
_{s}$. Thus $\pi _{\mathcal{Q}}(\alpha _{Tn+s})$ is independent of $n$ for
each $s$. It follows that $\pi _{\mathcal{Q}}(\alpha )$ is $T$-periodic. By
the defining property of $\mathcal{P}_{\ast },$ one has $\mathcal{P}_{\ast
}\preceq \mathcal{Q}$. Consequently, $\mathcal{Q}=\mathcal{P}_{\ast }$.
\end{proof}

A \emph{left-infinite word} over $\mathbb{A}$ is a word $v\in \mathbb{A}^{%
\mathbb{Z}_{\leq 0}}.$ Let $\mathcal{N}\subset \mathbb{N}_{0}$ be infinite.
We say that $v$ is \emph{realized along $\mathcal{N}$ in }$\alpha $ if, for
every $L\geq 0$, 
\begin{equation*}
\alpha _{n+t}=v(t)\text{\ }(-L\leq t\leq 0)
\end{equation*}%
for all sufficiently large $n\in \mathcal{N}$.

\begin{lemma}
\label{lem:leftlimits} For each $0\leq s<T$, write $t_{s}:=\sharp \mathbb{A}%
_{s}-1.$ The letters of $\mathbb{A}_{s}$ can be ordered as 
\begin{equation*}
\mathbb{A}_{s}=\{a_{s,0},a_{s,1},\ldots ,a_{s,t_{s}}\},\text{\ }%
a_{s,0}=\gamma _{s},
\end{equation*}%
so that, for every $1\leq j\leq t_{s}$, there exist a left-infinite word $%
v_{s,j}$ and an infinite set $\mathcal{N}_{s,j}\subset s+T\mathbb{N}_{0}$
such that $v_{s,j}$ is realized along $\mathcal{N}_{s,j}$ in $\alpha $ and 
\begin{equation}
v_{s,j}(0)=a_{s,j},\text{\ }v_{s,j}(Tu)\in \{a_{s,0},\ldots ,a_{s,j-1}\}%
\text{\ }(u\in \mathbb{Z}_{<0}).  \label{eq:leftchain}
\end{equation}
\end{lemma}

\begin{proof}
Fix $s$. By Lemma \ref{lem:alignment}, the residue word $\alpha ^{(s)}$
contains arbitrarily long constant $\gamma _{s}$-blocks arbitrarily far to
the right. In particular, $\gamma _{s}\in \mathbb{A}_{s}.$ Set $%
a_{s,0}:=\gamma _{s}$ and $B:=\{a_{s,0}\}.$

Suppose inductively that $B=\{a_{s,0},\ldots ,a_{s,j-1}\}\varsubsetneq 
\mathbb{A}_{s}$. Take intervals $[p_{\nu },q_{\nu })\subset \mathbb{N}_{0}$
such that $p_{\nu }\longrightarrow \infty ,$\ $q_{\nu }-p_{\nu
}\longrightarrow \infty ,$ and $\alpha _{n}^{(s)}=\gamma _{s}$\ $(p_{\nu
}\leq n<q_{\nu }).$ Since every letter of $\mathbb{A}_{s}\setminus B$ occurs
infinitely often in $\alpha ^{(s)}$, the integer 
\begin{equation*}
y_{\nu }:=\min \{n\geq q_{\nu }:\alpha _{n}^{(s)}\notin B\}
\end{equation*}%
is well defined. Then $\alpha _{n}^{(s)}\in B$\ for all $p_{\nu }\leq
n<y_{\nu },$ and $y_{\nu }-p_{\nu }\longrightarrow \infty .$

After passing to a subsequence we may assume that $\alpha _{y_{\nu
}}^{(s)}=a $ for some fixed letter $a\in \mathbb{A}_{s}\setminus B$. Set $%
a_{s,j}:=a$ and $N_{\nu }:=Ty_{\nu }+s.$ Then $N_{\nu }\rightarrow \infty $
and $N_{\nu }\equiv s\pmod T$. Passing to a further subsequence if
necessary, we may assume that $N_1<N_2<\cdots$.

Since $\mathbb{A}$ is finite, a diagonal argument yields a subsequence,
still indexed by $\nu $, and a word $v_{s,j}\in \mathbb{A}^{\mathbb{Z}_{\leq
0}}$ such that for every $L\geq 0$, one has%
\begin{equation}
\alpha _{N_{\nu }+t}=v_{s,j}(t)\ (-L\leq t\leq 0)
\label{eq:realization-property}
\end{equation}%
for all sufficiently large $\nu $. Thus $v_{s,j}$ is realized along the
infinite set 
\begin{equation*}
\mathcal{N}_{s,j}:=\{N_{\nu }:\nu \geq 1\}\subset s+T\mathbb{N}_{0},
\end{equation*}%
and $v_{s,j}(0)=a_{s,j}.$

Fix $u\in \mathbb{Z}_{<0}$. Since $y_{\nu }-p_{\nu }\rightarrow \infty $, we
have $p_{\nu }\leq y_{\nu }+u<y_{\nu }$ for all sufficiently large $\nu $.
Hence $\alpha _{N_{\nu }+Tu}=\alpha _{y_{\nu }+u}^{(s)}\in B.$ By (\ref%
{eq:realization-property}), one has $v_{s,j}(Tu)\in B.$ This proves (\ref%
{eq:leftchain}).

Adjoin $a_{s,j}$ to $B$ and continue. Since each step adds one new letter,
the process terminates with $B=\mathbb{A}_{s}.$
\end{proof}

\subsubsection{Affine separation}

\ 

Vectors $v_1,\ldots,v_c$ are \emph{linearly independent} if $\sum_i z_i
v_i=0 $ implies $z_i=0$ for every $i$. They are \emph{affinely independent}
if the conditions $\sum_i z_i v_i=0$ and $\sum_i z_i=0$ together imply $%
z_i=0 $ for every $i$. A directed graph is \emph{strongly connected} if
there is a directed path from each vertex to every other vertex. Loops are
allowed below.

\begin{lemma}
\label{lem:affine} Let $v_1,\ldots,v_c$ be affinely independent vectors, and
let $H_1,\ldots,H_c$ be arbitrary vectors in the same real vector space.
There are at most $c-1$ real numbers $t$ for which two of the vectors $H_i+t
v_i$ coincide.
\end{lemma}

\begin{proof}
Let $\mathcal{T}:=\left\{ t\in \mathbb{R}:H_{i}+tv_{i}=H_{j}+tv_{j}\text{
for some }i\neq j\right\} .$ Affine independence implies $v_{i}\neq v_{j}$
whenever $i\neq j$. Thus the equality $H_{i}+tv_{i}=H_{j}+tv_{j}$ can hold
for at most one value of $t$ for each pair of distinct indices. For each $%
t\in \mathcal{T}$, choose one corresponding pair and join its indices by an
undirected edge labelled by $t$. The resulting graph is simple, and its edge
labels are pairwise distinct.

Suppose that this graph contains a cycle $i_{1},i_{2},\ldots ,i_{s},i_{1}$,
and let $t_{j}$ be the label of the edge $\{i_{j},i_{j+1}\}$, where $%
i_{s+1}=i_{1}$. Summing the corresponding equalities gives 
\begin{equation*}
\sum_{j=1}^{s}(t_{j}-t_{j-1})v_{i_{j}}=0,\text{\ }t_{0}=t_{s}.
\end{equation*}%
Since the coefficients sum to zero, affine independence implies $%
t_{j}=t_{j-1}$ for every $j$, which contradicts the distinctness of the edge
labels. The graph is therefore a forest and has at most $c-1$ edges. Hence $%
\sharp \mathcal{T}\leq c-1.$
\end{proof}

\begin{lemma}
\label{lem:function} Let $G$ be a strongly connected directed graph on $%
\{1,\ldots ,c\}$ in which every vertex has an outgoing edge. Then there
exists a map $\sigma :\{1,\ldots ,c\}\rightarrow \{1,\ldots ,c\}$ such that $%
i\rightarrow \sigma (i)$ is an edge of $G$ for every $i$, and the vectors 
\begin{equation*}
v_{i}:=\mathbf{e}_{i}-\mathbf{e}_{\sigma (i)}\in \mathbb{R}^{c}\text{\ }%
(1\leq i\leq c)
\end{equation*}%
are affinely independent.
\end{lemma}

\begin{proof}
Choose a directed cycle in $G$, allowing a loop as a cycle of length one.
For vertices on the cycle, let $\sigma (i)$ be the next vertex on the cycle.
For each vertex $i$ outside the cycle, choose a shortest directed path from $%
i$ to the cycle and let $\sigma (i)$ be the next vertex on this path. Then $%
i\rightarrow \sigma (i)$ is an edge of $G$ for every $i$, and the selected
edges form the chosen cycle together with directed trees oriented towards it.

Suppose that 
\begin{equation*}
\sum_{i=1}^{c}z_{i}(\mathbf{e}_{i}-\mathbf{e}_{\sigma (i)})=0\text{\ and }%
\sum_{i=1}^{c}z_{i}=0.
\end{equation*}%
For a vertex $j,$ the $j$th coordinate equation is $z_{j}-\sum_{\sigma
(i)=j}z_{i}=0.$ Hence, at any vertex outside the cycle with indegree $0$,
one has $z_{j}=0.$ Removing such vertices successively shows that all
coefficients outside the cycle vanish. The remaining coordinate equations
imply that the coefficients on the cycle are all equal. Since all other
coefficients vanish and $\sum_{i=1}^{c}z_{i}=0$, this common value is zero.
Thus the vectors $v_{i}$ are affinely independent.
\end{proof}

\subsubsection{A grouped lower bound}

\ 

\label{subsec:grouped}

Write the blocks of $\mathcal{P}_{\ast }$ as $B_{1},\ldots ,B_{c_{\ast }}$.
Consider the directed graph on $\{1,\ldots ,c_{\ast }\}$ in which $%
h\rightarrow l$ is an edge if the factor $B_{h}B_{l}$ occurs in the periodic
word $\pi _{\mathcal{P}_{\ast }}(\gamma )$.$\ $Since every block occurs in
this word, it follows that following a complete period gives a directed path
from any vertex to every other vertex. Thus the graph is strongly connected,
and every vertex has an outgoing edge.

Choose $\sigma $ as in Lemma \ref{lem:function}. For each $h$, choose $%
s_{h}\in \{0,\ldots ,T-1\}$ such that 
\begin{equation}
\pi _{\mathcal{P}_{\ast }}(\gamma _{s_{h}})=B_{h}\text{\ and }\pi _{\mathcal{%
P}_{\ast }}(\gamma _{s_{h}+1})=B_{\sigma (h)}.  \label{eq:representatives}
\end{equation}%
The $s_{h}$ are pairwise distinct.

Let $\mathcal{E}$ denote the index set of the left-infinite words in Lemma %
\ref{lem:leftlimits}, that is, 
\begin{equation*}
\mathcal{E}=\{(s,j):0\leq s<T,1\leq j\leq t_{s}\},
\end{equation*}%
where $t_{s}:=\sharp \mathbb{A}_{s}-1$. For $e=(s,j)\in \mathcal{E}$, write $%
v_{e}:=v_{s,j}$, $a_{e}:=a_{s,j}$ and $\mathcal{N}_{e}:=\mathcal{N}_{s,j}.$

\begin{lemma}
\label{lem:columns} There exist an infinite set $\mathcal{N}%
=\{n_{0}<n_{1}<\cdots \}\subset \mathbb{N}_{0},$ $r_{\ast }=\ell -c_{\ast }$
distinct indices from $\mathcal{E}$, relabeled as $e=1,\ldots ,r_{\ast }$,
and letters $b_{e}(h,l)$ $(1\leq e\leq r_{\ast },1\leq h,l\leq c_{\ast })$
such that the following hold.

\begin{enumerate}
\item For every $0\leq \mu <\nu $, $1\leq e\leq r_{\ast }$ and $1\leq
h,l\leq c_{\ast }$%
\begin{equation*}
v_{e}(T(n_{\mu }-n_{\nu })+s_{h}-s_{l})=b_{e}(h,l).
\end{equation*}

\item For each $e$, the letter $b_{e}(h,h)$ is independent of $h$. Denote
its common value by $b_{e}$. The vectors $u_{e}:=\mathbf{e}_{a_{e}}-\mathbf{e%
}_{b_{e}}$ $(1\leq e\leq r_{\ast })$ are linearly independent.

\item Let $m\in \mathbb{N}$ and $h_{0},\ldots ,h_{m-1}\in \{1,\ldots
,c_{\ast }\},$ and set $z_{\nu }:=Tn_{\nu }+s_{h_{\nu }}$\ $(0\leq \nu <m).$
Then $z_{0}<z_{1}<\cdots <z_{m-1}.$ Moreover, for every $0\leq \nu <m$ and
every $1\leq e\leq r_{\ast }$, there exist arbitrarily large $p\in \mathbb{N}%
_{0}$ such that 
\begin{equation*}
\alpha _{z_{\nu }+p}=a_{e},\text{\ }\alpha _{z_{\mu }+p}=b_{e}(h_{\mu
},h_{\nu })\text{\ }(0\leq \mu <\nu ).
\end{equation*}
\end{enumerate}
\end{lemma}

\begin{proof}
For $n<n^{\prime }$, color the pair $\{n,n^{\prime }\}$ by 
\begin{equation}
(v_{e}(T(n-n^{\prime })+s_{h}-s_{l}))_{e\in \mathcal{E},\text{ }1\leq
h,l\leq c_{\ast }}.  \label{eq:color}
\end{equation}%
Since $T(n-n^{\prime })+s_{h}-s_{l}\leq -T+(T-1)=-1$ for $n<n^{\prime }$ and 
$0\leq s_{h},s_{l}<T,$ it follows that every coordinate in (\ref{eq:color})
is well defined. Since $\mathcal{E}$ and $\mathbb{A}$ are finite, (\ref%
{eq:color}) defines a finite coloring. By Theorem \ref{thm:Ramsey}, there is
an infinite set $\mathcal{N}=\{n_{0}<n_{1}<\cdots \}\subset \mathbb{N}_{0}$
on which this coloring is constant. Denote the value of its $(e,h,l)$
coordinate by $b_{e}(h,l)$. This proves (1), initially for every $e\in 
\mathcal{E}$.

If $h=l$, then the argument 
\begin{equation*}
T(n_{\mu }-n_{\nu })+s_{h}-s_{h}=T(n_{\mu }-n_{\nu })
\end{equation*}%
is independent of $h$. Hence $b_{e}(h,h)$ has a common value, which we
denote by $b_{e}$.

For $e=(s,j)$, (\ref{eq:leftchain}) implies $b_{e}\in \{a_{s,0},\ldots
,a_{s,j-1}\}.$ Consequently, for each fixed $s$, the edges $\{b_{e},a_{e}\}$ 
$(e=(s,j))$ form a tree on $\mathbb{A}_{s}$: at the $j$th step the new
vertex $a_{s,j}$ is joined to one of $a_{s,0},\ldots ,a_{s,j-1}$. The union
of these trees has the same connected components as the graph in Lemma \ref%
{lem:residuegraph}, and hence its connected components are exactly the
blocks of $\mathcal{P}_{\ast }$. For each block $B\in \mathcal{P}_{\ast },$
choose a spanning tree of the corresponding connected component. The total
number of selected edges is 
\begin{equation*}
\sum_{B\in \mathcal{P}_{\ast }}(\sharp B-1)=\ell -c_{\ast }=r_{\ast }.
\end{equation*}%
For each selected edge choose one index $e\in \mathcal{E}$ realizing it, and
relabel the chosen indices, together with their associated data, as $%
e=1,\ldots ,r_{\ast }$.

The corresponding vectors $u_{e}=\mathbf{e}_{a_{e}}-\mathbf{e}_{b_{e}}$ are
linearly independent. Indeed, in a linear relation among these vectors, the
coordinate corresponding to a leaf of one of the selected trees forces the
coefficient of its unique incident edge to vanish. Removing the leaf and its
incident edge, and repeating the argument shows that all coefficients
vanish. This proves (2). Since (1) was obtained simultaneously for all $e\in 
\mathcal{E}$, it remains valid for the selected indices.

For (3), if $\mu <\nu $, then 
\begin{equation*}
z_{\nu }-z_{\mu }=T(n_{\nu }-n_{\mu })+s_{h_{\nu }}-s_{h_{\mu }}\geq
T-(T-1)=1,
\end{equation*}%
so the $z_{\nu }$ are strictly increasing.

Fix $0\leq \nu <m$ and $1\leq e\leq r_{\ast }$. Since $v_{e}$ is realized
along $\mathcal{N}_{e}$, we may choose $Z\in \mathcal{N}_{e}$ sufficiently
large such that $Z\geq z_{\nu }$ and 
\begin{equation*}
\alpha _{Z+t}=v_{e}(t)\text{ }(-(z_{\nu }-z_{0})\leq t\leq 0).
\end{equation*}%
Set $p:=Z-z_{\nu }.$ Then $\alpha _{z_{\nu }+p}=v_{e}(0)=a_{e}.$ For $\mu
<\nu $, one has%
\begin{equation*}
z_{\mu }-z_{\nu }=T(n_{\mu }-n_{\nu })+s_{h_{\mu }}-s_{h_{\nu }},
\end{equation*}%
and hence, by (1), 
\begin{equation*}
\alpha _{z_{\mu }+p}=v_{e}(z_{\mu }-z_{\nu })=b_{e}(h_{\mu },h_{\nu }).
\end{equation*}%
Since $Z$ may be chosen arbitrarily large in $\mathcal{N}_e$, so may $p$.
\end{proof}

For $\lambda =(\lambda _{1},\ldots ,\lambda _{r})\in \mathbb{N}_{0}^{r},$
write $\left\Vert \lambda \right\Vert _{1}:=\sum_{e=1}^{r}\lambda _{e}.$ For 
$r,t\geq 1$, define 
\begin{equation*}
\rho _{r}(t):=\min \left\{ L\in \mathbb{N}_{0}:\binom{L+r}{r}\geq
t+1\right\} ,\text{\ }W_{r}(m):=\sum_{t=1}^{m-1}\rho _{r}(t)\text{\ }(m\geq
1).
\end{equation*}%
Thus $W_{r}(1)=0$. The quantity $W_{r}(m)$ bounds the number of coordinates
used below; it does not bound the diameter of the resulting pattern.

\begin{theorem}
\label{thm:grouped} Suppose that $\alpha \in \mathcal{A}_{\ell }$ has a
periodic point in its orbit closure, and let $c_{\ast },r_{\ast }$ be as in (%
\ref{eq:canonical}). Let $m_{1},\ldots ,m_{c_{\ast }}$ be positive integers
with sum $m$. If $k\in \mathbb{N}$ satisfies 
\begin{equation}
k\geq \sum_{h=1}^{c_{\ast }}W_{r_{\ast }}(m_{h})+c_{\ast }-1,
\label{eq:groupcost}
\end{equation}%
then $p_{\alpha }^{\ast \mathrm{ab}}(k)\geq m.$
\end{theorem}

\begin{proof}
We retain the notation and choices made above, in particular $T,$ $s_{h},$ $%
\sigma .$ We also fix the infinite set $\mathcal{N}=\{n_{0}<n_{1}<\cdots \},$
the indices $e=1,\ldots ,r_{\ast }$, and the associated data $a_{e},$ $%
b_{e}(h,l)$ and $u_{e}$ given by Lemma \ref{lem:columns}.

Take $h_{0},\ldots ,h_{m-1}\in \{1,\ldots ,c_{\ast }\}$ such that each $h\in
\{1,\ldots ,c_{\ast }\}$ occurs exactly $m_{h}$ times, and set $z_{\nu
}:=Tn_{\nu }+s_{h_{\nu }}$\ $(0\leq \nu <m).$ By Lemma \ref{lem:columns}(3),
one has $z_{0}<z_{1}<\cdots <z_{m-1}.$

For each $h\in \{1,\ldots ,c_{\ast }\},$ call the set of indices $\nu $ with 
$h_{\nu }=h$ the group $h.$ We first separate the positions belonging to the
same group. Set $S_{0}:=\varnothing .$ Suppose inductively that $0\leq \nu
<m $ and that $S_{\nu }$ has been chosen so that, for every $h$, the vectors 
$\Psi (\alpha \lbrack z_{\mu }+S_{\nu }])$\ $(0\leq \mu <\nu ,$\ $h_{\mu
}=h) $ are pairwise distinct.

Let 
\begin{equation*}
I_{\nu }:=\{\mu <\nu :h_{\mu }=h_{\nu }\},\text{\ }t_{\nu }:=\sharp I_{\nu }.
\end{equation*}%
If $t_{\nu }=0$, set $S_{\nu +1}:=S_{\nu }.$ Suppose that $t_{\nu }\geq 1$,
and put $L_{\nu }:=\rho _{r_{\ast }}(t_{\nu }).$ Using Lemma \ref%
{lem:columns}(3), we may choose pairwise distinct integers $p_{e,\tau }\in 
\mathbb{N}_{0}\setminus S_{\nu }$ $(1\leq e\leq r_{\ast },1\leq \tau \leq
L_{\nu })$ such that 
\begin{equation}
\alpha _{z_{\nu }+p_{e,\tau }}=a_{e}\text{ and }\alpha _{z_{\mu }+p_{e,\tau
}}=b_{e}(h_{\mu },h_{\nu })\ (0\leq \mu <\nu )  \label{eq:vector-comp}
\end{equation}%
for every $1\leq \tau \leq L_{\nu }$.

For $\lambda =(\lambda _{1},\ldots ,\lambda _{r_{\ast }})\in \mathbb{N}%
_{0}^{r_{\ast }}$\ with $\left\Vert \lambda \right\Vert _{1}\leq L_{\nu },$
set 
\begin{equation*}
P_{\lambda }:=\bigcup_{e=1}^{r_{\ast }}\{p_{e,\tau }:1\leq \tau \leq \lambda
_{e}\}.
\end{equation*}%
If $\mu \in I_{\nu }$, then $h_{\mu }=h_{\nu }$, so $b_{e}(h_{\mu },h_{\nu
})=b_{e}.$ It follows from (\ref{eq:vector-comp}) that 
\begin{equation}
\Psi (\alpha \lbrack z_{\nu }+\left( S_{\nu }\cup P_{\lambda }\right)
])-\Psi (\alpha \lbrack z_{\mu }+\left( S_{\nu }\cup P_{\lambda }\right)
])=D_{\mu }+\sum_{e=1}^{r_{\ast }}\lambda _{e}u_{e},
\label{eq:affinedifference}
\end{equation}%
where $D_{\mu }:=\Psi (\alpha \lbrack z_{\nu }+S_{\nu }])-\Psi (\alpha
\lbrack z_{\mu }+S_{\nu }]).$ By Lemma \ref{lem:columns}(2), $u_{1},\ldots
,u_{r_{\ast }}$ are linearly independent, and hence for each $\mu \in I_{\nu
},$ the equation 
\begin{equation*}
D_{\mu }+\sum_{e=1}^{r_{\ast }}\lambda _{e}u_{e}=0
\end{equation*}%
has at most one solution $\lambda $. On the other hand, one has%
\begin{equation*}
\sharp \left\{ \lambda \in \mathbb{N}_{0}^{r_{\ast }}:\left\Vert \lambda
\right\Vert _{1}\leq L_{\nu }\right\} =\binom{L_{\nu }+r_{\ast }}{r_{\ast }}%
\geq t_{\nu }+1.
\end{equation*}%
Hence there is a choice of $\lambda $ for which the difference in (\ref%
{eq:affinedifference}) is nonzero for every $\mu \in I_{\nu }$. Fix such a $%
\lambda $ and set $S_{\nu +1}:=S_{\nu }\cup P_{\lambda }.$

Now we verify that the separations obtained at earlier stages are preserved.
If $\mu ,\mu ^{\prime }<\nu $ and $h_{\mu }=h_{\mu ^{\prime }}$, then every
new coordinate $p_{e,\tau }$ satisfies 
\begin{equation*}
\alpha _{z_{\mu }+p_{e,\tau }}=b_{e}(h_{\mu },h_{\nu })=b_{e}(h_{\mu
^{\prime }},h_{\nu })=\alpha _{z_{\mu ^{\prime }}+p_{e,\tau }}.
\end{equation*}%
Thus 
\begin{equation*}
\Psi (\alpha \lbrack z_{\mu }+S_{\nu +1}])-\Psi (\alpha \lbrack z_{\mu
^{\prime }}+S_{\nu +1}])=\Psi (\alpha \lbrack z_{\mu }+S_{\nu }])-\Psi
(\alpha \lbrack z_{\mu ^{\prime }}+S_{\nu }]).
\end{equation*}%
The induction therefore continues.

Set $\widetilde{S}:=S_{m}$ and $w:=\sharp \widetilde{S}.$ By the induction
above, for each $h$, the vectors 
\begin{equation*}
\Psi (\alpha \lbrack z_{\nu }+\widetilde{S}])\text{\ }(0\leq \nu <m,\text{\ }%
h_{\nu }=h)
\end{equation*}%
are pairwise distinct. For a group containing $m_{h}$ indices, the
successive values of $t_{\nu }$ are $0,1,\ldots ,m_{h}-1.$ Since the step
corresponding to $t_{\nu }=t\geq 1$ adds at most $\rho _{r_{\ast }}(t)$
coordinates, one has%
\begin{equation}
w\leq \sum_{h=1}^{c_{\ast }}\sum_{t=1}^{m_{h}-1}\rho _{r_{\ast
}}(t)=\sum_{h=1}^{c_{\ast }}W_{r_{\ast }}(m_{h}).
\label{eq:within-group-cost}
\end{equation}

It remains to separate different groups and, at the same time, enlarge the
pattern to cardinality $k$. Define the linear map $Q:\mathbb{R}^{\mathbb{A}%
}\longrightarrow \mathbb{R}^{c_{\ast }}$ by $Q(\mathbf{e}_{a}):=\mathbf{e}%
_{h}$\ $(a\in B_{h}).$ Since $\pi _{\mathcal{P}_{\ast }}(\alpha )$ is $T$%
-periodic and $z_{\nu }\equiv s_{h_{\nu }}\pmod T$, for each fixed $h$ the
vector $Q(\Psi (\alpha \lbrack z_{\nu }+\widetilde{S}]))$ is independent of $%
\nu $ among the indices satisfying $h_{\nu }=h$. Denote this common vector
by $H_{h}$. Thus 
\begin{equation*}
Q(\Psi (\alpha \lbrack z_{\nu }+\widetilde{S}]))=H_{h_{\nu }}\text{\ }(0\leq
\nu <m).
\end{equation*}

Let $L:=k-w.$ By (\ref{eq:groupcost}) and (\ref{eq:within-group-cost}), we
have $L\geq c_{\ast }-1.$

For each $\varepsilon \in \{0,1\}$, there are arbitrarily large $p\in 
\mathbb{N}_{0}$ satisfying $p\equiv \varepsilon \pmod T$ and 
\begin{equation}
\alpha _{z_{\nu }+p}=\gamma _{s_{h_{\nu }}+\varepsilon }\text{\ }(0\leq \nu
<m).  \label{eq:periodiccolumns}
\end{equation}%
Indeed, by Lemma \ref{lem:alignment} we may choose an arbitrarily far
interval $J=[b,b+D)$ such that $\alpha \lbrack J]=\gamma \lbrack J]$ and $%
D>z_{m-1}-z_{0}+T.$ There is a unique $p\in \lbrack b-z_{0},b-z_{0}+T)$ such
that $p\equiv \varepsilon \pmod T.$ Then $z_{\nu }+p\in J$\ $(0\leq \nu <m),$
and $z_{\nu }+p\equiv s_{h_{\nu }}+\varepsilon \pmod T.$ Since $\gamma $ is $%
T$-periodic, (\ref{eq:periodiccolumns}) follows. As the intervals $J$ may be
chosen arbitrarily far to the right, so may $p$.

Let $q$ be an integer with $0\leq q\leq L$. Choose $L$\ pairwise distinct
integers outside $\widetilde{S},$ of which $q$ satisfy (\ref%
{eq:periodiccolumns}) with $\varepsilon =0$ and $L-q$ satisfy it with $%
\varepsilon =1$. By (\ref{eq:representatives}), the resulting quotient
vector of group $h$ is 
\begin{equation}
H_{h}+q\mathbf{e}_{h}+(L-q)\mathbf{e}_{\sigma (h)}=H_{h}+L\mathbf{e}_{\sigma
(h)}+q(\mathbf{e}_{h}-\mathbf{e}_{\sigma (h)}).  \label{eq:groupaffine}
\end{equation}%
By Lemma \ref{lem:function}, the vectors $\mathbf{e}_{h}-\mathbf{e}_{\sigma
(h)}$ $(1\leq h\leq c_{\ast })$ are affinely independent. Hence Lemma \ref%
{lem:affine} implies that at most $c_{\ast }-1$ real values of $q$ produce a
coincidence between two vectors in (\ref{eq:groupaffine}). Therefore some $%
q\in \{0,1,\ldots ,c_{\ast }-1\}$ makes all $c_{\ast }$ quotient vectors
pairwise distinct. Since $L\geq c_{\ast }-1$, this choice satisfies $0\leq
q\leq L$.

Fix such a $q$ and adjoin the corresponding $L$ integers to $\widetilde{S}$,
obtaining a pattern $S$. By (\ref{eq:periodiccolumns}), every added element
contributes the same letter at all $z_{\nu }$ belonging to a fixed group.
Hence all within-group Parikh-vector differences remain unchanged. By (\ref%
{eq:groupaffine}), different groups have distinct images under $Q$ . Thus
the $m$ vectors $\Psi (\alpha \lbrack z_{0}+S]),\ldots ,\Psi (\alpha \lbrack
z_{m-1}+S])$ are pairwise distinct. Therefore $p_{\alpha }^{\ast \mathrm{ab}%
}(k)\geq p_{\alpha }^{\mathrm{ab}}(S)\geq m.$
\end{proof}

\begin{remark}
At each inductive step, the newly added coordinates need only contribute the
same letters within each previously treated group. Thus the within-group
separations are preserved, while different groups are separated later by the
quotient argument in (\ref{eq:groupaffine}). The parameters $T$ and $%
\mathcal{E}$ affect only the locations of the chosen coordinates and do not
enter the bound (\ref{eq:groupcost}).
\end{remark}

\begin{lemma}
\label{lem:fewgroups} Suppose that $\alpha \in \mathcal{A}_{\ell }$ has a
periodic point in its orbit closure. Then $p_{\alpha }^{\ast \mathrm{ab}%
}(k)\geq c_{\ast }$\ for every $k\in \mathbb{N}.$
\end{lemma}

\begin{proof}
Retain the notation above, in particular $T$, $\gamma $, and $s_{1},\ldots
,s_{c_{\ast }}$. Fix $k\geq 1$ and set $S:=\{0,T,\ldots ,(k-1)T\}.$ By Lemma %
\ref{lem:alignment}, choose an interval $J=[b,b+D)$ with $D>kT$ such that $%
\alpha \lbrack J]=\gamma \lbrack J].$ For each $1\leq h\leq c_{\ast }$, let $%
t_{h}\in \lbrack b,b+T)$ be the unique integer satisfying $t_{h}\equiv s_{h}%
\pmod T.$ Then $t_{h}+S\subset J,$ and hence $\Psi (\alpha \lbrack
t_{h}+S])=k\mathbf{e}_{\gamma _{s_{h}}}.$

By (\ref{eq:representatives}), one has $\pi _{\mathcal{P}_{\ast }}(\gamma
_{s_{h}})=B_{h}.$ Thus the letters $\gamma _{s_{1}},\ldots ,\gamma
_{s_{c_{\ast }}}$ are pairwise distinct. Hence the vectors 
\begin{equation*}
\Psi (\alpha \lbrack t_{h}+S])=k\mathbf{e}_{\gamma _{s_{h}}},\text{\ }1\leq
h\leq c_{\ast },
\end{equation*}%
are pairwise distinct. Since $\sharp S=k$, it follows that $p_{\alpha
}^{\ast \mathrm{ab}}(k)\geq p_{\alpha }^{\mathrm{ab}}(S)\geq c_{\ast }.$
\end{proof}

\subsection{Proof of Theorem~\protect\ref{thm:finite}}

\ 

\label{sec:integers}

\begin{lemma}
\label{lem:rhoestimate} For $r\geq 2$ and $t\geq 2r$, one has $\rho
_{r}(t)\leq t/r$.
\end{lemma}

\begin{proof}
We claim that for $r\geq 2$ and $L\geq 3$, 
\begin{equation}
\binom{L+r-1}{r}\geq rL.  \label{eq:binomialestimate}
\end{equation}%
At $L=3$ the difference between the two sides of (\ref{eq:binomialestimate})
is $(r-1)(r-2)/2\geq 0$. Increasing $L$ by one increases the left side by $%
\binom{L+r-1}{r-1}\geq r$, so induction proves (\ref{eq:binomialestimate}).

Put $L=\rho _{r}(t)$. If $L\leq 2$, the conclusion follows from $t\geq 2r$.
If $L\geq 3$, minimality gives $\binom{L+r-1}{r}<t+1$, and hence $\binom{%
L+r-1}{r}\leq t$. By (\ref{eq:binomialestimate}), one has $rL\leq \binom{%
L+r-1}{r}\leq t,$ so $L\leq t/r$.
\end{proof}

\begin{lemma}
\label{lem:Westimate}For all integers $r,n\geq 1$, $W_{r}(n)\leq \frac{1}{r}%
\binom{n}{2}+r-1.$ For $r\geq 2$, equality holds when $n=2r$.
\end{lemma}

\begin{proof}
For $r=1$, $\rho _{1}(t)=t$ and $W_{1}(n)=\binom{n}{2}$. Suppose $r\geq 2$
and write $D_{r}(n)=W_{r}(n)-\binom{n}{2}/r$. The definition gives $\rho
_{r}(t)=1$ for $1\leq t\leq r$, and $\rho _{r}(t)=2$ for $r<t\leq 2r$. Hence
one has%
\begin{equation}
W_{r}(n)=n-1\text{\ }(1\leq n\leq r+1)\text{ and }W_{r}(n)=2n-r-2\text{\ }%
(r+1\leq n\leq 2r+1).  \label{eq:W_r-n}
\end{equation}%
In the first range, 
\begin{equation*}
D_{r}(n)=\frac{r-1}{2}-\frac{(n-r)(n-r-1)}{2r}\leq \frac{r-1}{2};
\end{equation*}%
in the second, 
\begin{equation*}
D_{r}(n)=r-1-\frac{(n-2r)(n-2r-1)}{2r}\leq r-1.
\end{equation*}%
Both inequalities use $z(z-1)\geq 0$ for integers $z$. In particular, $%
D_{r}(2r)=r-1$. For $n\geq 2r$, Lemma \ref{lem:rhoestimate} gives 
\begin{equation*}
D_{r}(n+1)-D_{r}(n)=\rho _{r}(n)-n/r\leq 0.
\end{equation*}%
These initial estimates and monotonicity prove the assertion for all $n$.
\end{proof}

For $c\geq 2$, $r\geq 1$, and $m\geq c$, write $m=ca+b$, where $a\geq 1$ and 
$0\leq b<c$. Define 
\begin{equation}
B_{c,r}(m):=(c-b)W_{r}(a)+bW_{r}(a+1)+c-1.  \label{eq:balancedcost}
\end{equation}%
This is the bound in Theorem \ref{thm:grouped} for groups whose sizes differ
by at most one.

\begin{lemma}
\label{lem:balanced} Let $g=c+r-1$, with $c\geq 2$ and $r\geq 1$. Then one
has%
\begin{align}
B_{c,r}(m) &= m-1, \text{ if } c\leq m\leq 2g,
\label{eq:smallgroups}\\
B_{c,r}(m) &\leq \frac{1}{g}\binom{m}{2},
\text{ if } m\geq 2g.
\label{eq:largegroups}
\end{align}
\end{lemma}

\begin{proof}
Since $c(r+1)-2g=(c-2)(r-1)\geq 0$, if $m\leq 2g,$ then every group in a
partition of $m$ into $c$ groups whose sizes differ by at most one has size
at most $r+1$. It follows from (\ref{eq:W_r-n}) that $W_{r}(n)=n-1$ for all
such group sizes. Hence (\ref{eq:smallgroups}) follows. At $m=2g,$ (\ref%
{eq:smallgroups}) gives $B_{c,r}(2g)=2g-1=\binom{2g}{2}/g$.

From (\ref{eq:balancedcost}), we have 
\begin{equation*}
B_{c,r}(m+1)-B_{c,r}(m)=\rho _{r}(\lfloor m/c\rfloor ).
\end{equation*}%
Set $a=\lfloor m/c\rfloor $ and assume $m\geq 2g$. If $r=1$, then $g=c$ and $%
\rho _{1}(a)=a\leq m/g$. For $r\geq 2$, if $a\leq r,$ then $\rho _{r}(a)=1,$
while if $r<a<2r,$ then $\rho _{r}(a)=2$. In either case, $\rho _{r}(a)\leq
2\leq m/g$. If $a\geq 2r$, then by Lemma \ref{lem:rhoestimate} one has%
\begin{equation*}
\rho _{r}(a)\leq a/r\leq ca/g\leq m/g,
\end{equation*}%
where the middle inequality follows from $cr-g=(c-1)(r-1)\geq 0,$ and the
last from\ $ca\leq m$. Thus%
\begin{equation*}
B_{c,r}(m+1)-B_{c,r}(m)\leq \frac{1}{g}\binom{m+1}{2}-\frac{1}{g}\binom{m}{2}%
=\frac{m}{g}.
\end{equation*}%
Induction from $m=2g$ proves (\ref{eq:largegroups}).
\end{proof}

\begin{proof}[Proof of Theorem~\protect\ref{thm:finite}]
The case $m=1$ is immediate. For $\ell =2$, the assertion follows from
Corollary \ref{cor:universalbinary}. Assume $\ell \geq 3$ and put $g=\ell -1$%
. For integers $g,m\geq 1$, one has%
\begin{equation}
\frac{1}{g}\binom{m}{2}+g-1-(m-1)=\frac{(m-g)(m-g-1)+g(g-1)}{2g}\geq 0.
\label{eq:smallidentity}
\end{equation}%
If $\mathcal{O}(\alpha )$ has no periodic point, then (\ref{eq:smallidentity}) and the hypothesis on $k$
imply $k\geq m-1.$ Hence Lemma \ref%
{lem:noperiodic} implies $p_{\alpha }^{\ast \mathrm{ab}}(k)\geq k+1\geq m$.

Suppose that $\mathcal{O}(\alpha )$ contains a periodic point, and retain
the notation of Subsection \ref{sec:partition}. If $c_{\ast }=1$, then $%
r_{\ast }=g$, and Lemma \ref{lem:Westimate} gives 
\begin{equation*}
W_{r_{\ast }}(m)\leq \frac{1}{g}\binom{m}{2}+g-1\leq k.
\end{equation*}%
Theorem \ref{thm:grouped}, applied with the single group of size $m$, proves
the result.

Assume $c_{\ast }\geq 2$. Recall that $r_{\ast }=\ell -c_{\ast }\geq 1$ and $%
c_{\ast }+r_{\ast }-1=g.$ If $m<c_{\ast }$, the conclusion follows from
Lemma \ref{lem:fewgroups}. For $c_{\ast }\leq m<2g$, Lemma \ref{lem:balanced}
and (\ref{eq:smallidentity}) give 
\begin{equation*}
B_{c_{\ast },r_{\ast }}(m)=m-1\leq \frac{1}{g}\binom{m}{2}+g-1\leq k.
\end{equation*}%
For $m\geq 2g$, Lemma \ref{lem:balanced} also gives $B_{c_{\ast },r_{\ast
}}(m)\leq \frac{1}{g}\binom{m}{2}\leq k$. In both cases, write $m=c_{\ast
}a+b$ with $0\leq b<c_{\ast }$. Taking $c_{\ast }-b$ groups of size $a$ and $%
b$ groups of size $a+1$, the right-hand side of (\ref{eq:groupcost}) is $%
B_{c_*,r_*}(m)\le k$. Theorem \ref{thm:grouped} therefore gives $p_\alpha^{*%
\mathrm{ab}}(k)\ge m.$
\end{proof}

\begin{remark}
Theorem\ \ref{thm:finite} remains valid if the word contains letters
occurring only finitely often, provided that $\ell $ denotes the number of
letters occurring infinitely often. Indeed, after deleting a finite prefix
containing all other occurrences, Theorem\ \ref{thm:finite} applies to the
resulting tail, and Lemma \ref{lem:transfer}(2) transfers the lower bound to
the original word. Theorem\ \ref{thm:finite} need not hold if $\ell $ is
instead taken to be $\sharp \mathbb{A}_{\alpha }$; see Remark\ \ref%
{rem:finite-letters}.
\end{remark}

\subsection{An upper bound and the extremal square-root law}

\ 

\label{subsec:square-root} Fix $\ell \geq 2$ and put $g:=\ell -1$. Let $x$
be the sparse binary word in (\ref{eq:def-sparse-bi-core}). Take distinct
letters $c_{1},\ldots ,c_{g-1}$ outside $\{0,1\}$, and let $\varpi _{g}$ be
the separated $g$-interleaving of $x$ determined by 
\begin{equation*}
\varpi _{g}^{(0)}=x,\text{\ }\varpi _{g}^{(r)}=c_{r}^{\infty }\text{\ }%
(1\leq r\leq g-1).
\end{equation*}%
When $g=1$, this is just $x$. The word $\varpi _{g}$ has $\ell =g+1$
letters, each occurring infinitely often. Since $x$ is aperiodic, Lemma \ref%
{lem:interleaving-recurrent-aperiodic}(2) implies that $\varpi _{g}$ is
aperiodic. Hence $\varpi _{g}\in \mathcal{A}_{\ell }.$

\begin{proposition}
\label{prop:interleaving} The word $\varpi _{g}$ satisfies 
\begin{equation*}
p_{\varpi _{g}}^{\ast \mathrm{ab}}(k)\leq \sqrt{2gk}+2g=\sqrt{2(\ell -1)k}%
+2(\ell -1)\text{\ }(k\in \mathbb{N}).
\end{equation*}
\end{proposition}

\begin{proof}
Let $S\in \Sigma _{k}(\mathbb{N}_{0}).$ For $0\leq r\leq g-1$, set 
\begin{equation*}
S_{r}:=\{s\in S:s\equiv r\text{ }(\func{mod}g)\},\text{\ }k_{r}:=\sharp
S_{r}.
\end{equation*}%
Then $\sum_{r=0}^{g-1}k_{r}=k.$

Since $\varpi _{g}$ is a $g$-interleaving of $x$, Lemma \ref%
{lem:interleaving-decomposition} gives 
\begin{equation}
p_{\varpi _{g}}^{\mathrm{ab}}(S)\leq \sharp \{r:k_{r}=0\}+\sum_{\substack{ %
0\leq r\leq g-1  \\ k_{r}>0}}p_{x}^{\ast \mathrm{ab}}(k_{r}).
\label{eq:pab(S)<=k_r-g-x}
\end{equation}%
By Corollary \ref{cor:sparse-upper}, one has $p_{x}^{\ast \mathrm{ab}%
}(t)\leq \sqrt{2t}+2$\ for all $t\in \mathbb{N}.$ For $k_{r}=0,$ the
corresponding contribution in (\ref{eq:pab(S)<=k_r-g-x}) is $1\leq 2=\sqrt{%
2k_{r}}+2.$ Hence one has 
\begin{equation*}
p_{\varpi _{g}}^{\mathrm{ab}}(S)\leq \sum_{r=0}^{g-1}\left( \sqrt{2k_{r}}%
+2\right) =\sqrt{2}\sum_{r=0}^{g-1}\sqrt{k_{r}}+2g\leq \sqrt{%
2g\sum_{r=0}^{g-1}k_{r}}+2g=\sqrt{2gk}+2g,
\end{equation*}%
where the last inequality follows from the Cauchy--Schwarz inequality.
Taking the supremum over $S\in \Sigma _{k}(\mathbb{N}_{0})$ gives $p_{\varpi
_{g}}^{\ast \mathrm{ab}}(k)\leq \sqrt{2gk}+2g.$ Since $g=\ell -1$, the
assertion follows.
\end{proof}

We now combine the lower and upper bounds to determine the asymptotic
behavior of $\mathcal{L}_{\ell }^{\mathrm{ab}}(k)$.

\begin{theorem}
\label{thm:extremal}For every $\ell \geq 2$, $\mathcal{L}_{\ell }^{\mathrm{ab%
}}(k)=\sqrt{2(\ell -1)k}+O_{\ell }(1).$ In particular, $\lim_{k\rightarrow
\infty }\frac{\mathcal{L}_{\ell }^{\mathrm{ab}}(k)}{\sqrt{k}}=\sqrt{2(\ell
-1)}.$ Here $O_{\ell }(1)$ denotes a quantity bounded in absolute value by a
constant depending only on $\ell $.
\end{theorem}

\begin{proof}
Put $g:=\ell -1$. The function $t\longmapsto \frac{t(t-1)}{2g}+g-1$ is
strictly increasing on $[1,\infty )$. For $k\geq g-1$, let 
\begin{equation*}
r_{k}:=\frac{1+\sqrt{1+8g(k-g+1)}}{2}.
\end{equation*}%
Then $r_{k}\geq 1$ and $r_{k}(r_{k}-1)/(2g)+g-1=k$. Put $m_{k}:=\lfloor
r_{k}\rfloor $. Since $m_{k}\leq r_{k},$ monotonicity gives $\frac{1}{g}%
\binom{m_{k}}{2}+g-1\leq k$. Hence Theorem \ref{thm:finite} implies 
\begin{equation*}
\mathcal{L}_{\ell }^{\mathrm{ab}}(k)\geq m_{k}\geq r_{k}-1.
\end{equation*}%
As $k\rightarrow \infty $, $r_{k}=\sqrt{2gk}+O_{g}(1)$. Hence 
\begin{equation}
\mathcal{L}_{\ell }^{\mathrm{ab}}(k)\geq \sqrt{2gk}-O_{\ell }(1).
\label{eq:extremallower}
\end{equation}%
After enlarging the implicit constant, the same bound holds for the finitely
many remaining positive values of $k$.

For the reverse inequality, Proposition \ref{prop:interleaving} provides a
fixed word $\varpi _{g}\in \mathcal{A}_{\ell }$ satisfying $p_{\varpi
_{g}}^{\ast \mathrm{ab}}(k)\leq \sqrt{2gk}+2g$. Thus 
\begin{equation}
\mathcal{L}_{\ell }^{\mathrm{ab}}(k)\leq \sqrt{2gk}+O_{\ell }(1).
\label{eq:extremalupper}
\end{equation}%
Combining (\ref{eq:extremallower}) and (\ref{eq:extremalupper}) gives $%
\mathcal{L}_{\ell }^{\mathrm{ab}}(k)=\sqrt{2(\ell -1)k}+O_{\ell }(1).$
Division by $\sqrt{k}$ yields the asserted limit.
\end{proof}

\begin{remark}
The coefficient $1/(2(\ell -1))$ of $m^{2}$ in Theorem \ref{thm:finite} is
optimal among sufficient bounds with a lower-order remainder. Indeed, if a
bound of the form $k\geq am^{2}+o(m^{2})$ were sufficient with $a<1/(2(\ell
-1))$, then applying it to $\varpi _{\ell -1}$ would contradict Proposition %
\ref{prop:interleaving} for large $m$. This does not assert that the
additive term $\ell -2$ in Theorem \ref{thm:finite} is optimal.
\end{remark}

\begin{remark}
\label{rem:finite-letters}For $\ell \geq 3$, the condition that every letter
occur infinitely often in the definition of $\mathcal{A}_{\ell }$ cannot be
replaced by the requirement that every letter occur at least once. Let $x$
be the sparse binary word in (\ref{eq:def-sparse-bi-core}), choose a finite
word $u$ containing $\ell -2$ new letters, and put $\widetilde{\alpha }:=ux,$%
\ $M:=|u|.$ Then $\widetilde{\alpha }$ is aperiodic and $\sharp \mathbb{A}_{%
\widetilde{\alpha }}=\ell $, while $\mathsf{T}^{M}\widetilde{\alpha }=x$. By
Lemma\ 2.2(2), one has 
\begin{equation*}
p_{\widetilde{\alpha }}^{\ast \mathrm{ab}}(k)\leq p_{x}^{\ast \mathrm{ab}%
}(k)+M=f(k)+M\leq \sqrt{2k}+M+2.
\end{equation*}
If Theorem \ref{thm:finite} remained valid with $\ell =\sharp \mathbb{A}_{%
\widetilde{\alpha }}$, it would imply $p_{\widetilde{\alpha }}^{\ast \mathrm{%
ab}}(k)\geq \sqrt{2(\ell -1)k}-O_{\ell }(1),$ which is impossible for large $%
k$.
\end{remark}

For words in $\mathcal{A}_{\ell }$ with $\ell \geq 3$, the lower bounds and
their attainment lead to the following questions.

\begin{question}
Can the additive term $\ell -2$ in Theorem \ref{thm:finite} be removed?
Equivalently, for each $\ell \geq 3$, does the implication 
\begin{equation*}
\binom{m}{2}\leq (\ell -1)k\Longrightarrow p_{\alpha }^{\ast \mathrm{ab}%
}(k)\geq m
\end{equation*}%
hold for every $\alpha \in \mathcal{A}_{\ell }$ and all $m,k\geq 1$?
\end{question}

\begin{question}
For each $\ell \geq 3$, determine $\mathcal{L}_{\ell }^{\mathrm{ab}}(k)$ for
all $k\geq 1$. Moreover, does there exist a single word $\alpha _{\ell }\in 
\mathcal{A}_{\ell }$ such that $p_{\alpha _{\ell }}^{\ast \mathrm{ab}}(k)=%
\mathcal{L}_{\ell }^{\mathrm{ab}}(k)$\ for every $k\geq 1$?
\end{question}

\section*{Acknowledgments}

This work was partially supported by the National Natural Science Foundation
of China (No. 12301107) and the Shandong Provincial Natural Science
Foundation, China (No. ZR202209010046).

\section*{Conflict of Interest}
The authors declare that they have no financial or non-financial interests relevant to the content of this article.


\begin{thebibliography}{99}
\bibitem{Bruin2022} H. Bruin, Topological and Ergodic Theory of Symbolic
Dynamics, Graduate Studies in Mathematics, vol. 228, American Mathematical
Society, 2022.

\bibitem{CH1973} E. M. Coven, G. A. Hedlund, Sequences with minimal block
growth, Math. Syst. Theory \textbf{7} (1973) 138--153.

\bibitem{FP2023} G. Fici, S. Puzynina, Abelian combinatorics on words: A
survey, Comput. Sci. Rev. \textbf{47} (2023) 100532.

\bibitem{Kamae2011} T. Kamae, Uniform sets and super-stationary sets over
general alphabets, Ergodic Theory Dynam. Systems \textbf{31} (2011)
1445--1461.

\bibitem{KR2006} T. Kamae, H. Rao, Maximal pattern complexity of words over $%
\ell $ letters, European J. Combin. \textbf{27} (2006) 125--137.

\bibitem{KRTX2006} T. Kamae, H. Rao, B. Tan, Y.-M. Xue, Language structure
of pattern Sturmian words, Discrete Math. \textbf{306} (2006) 1651--1668.

\bibitem{KWZ2015} T. Kamae, S. Widmer, L. Q. Zamboni, Abelian maximal
pattern complexity of words, Ergodic Theory Dynam. Systems \textbf{35}
(2015) 142--151.

\bibitem{KZ2002a} T. Kamae, L. Q. Zamboni, Sequence entropy and the maximal
pattern complexity of infinite words, Ergodic Theory Dynam. Systems \textbf{%
22} (2002) 1191--1199.

\bibitem{KZ2002b} T. Kamae, L. Q. Zamboni, Maximal pattern complexity for
discrete systems, Ergodic Theory Dynam. Systems \textbf{22} (2002)
1201--1214.


\bibitem{LPS} A. N. Le, R. Pavlov, C. Schlortt, On subshifts with low
maximal pattern complexity, Trans. Amer. Math. Soc., to appear.
doi:10.1090/tran/9786.

\bibitem{MH1940} M. Morse, G. A. Hedlund, Symbolic dynamics II. Sturmian
trajectories, Amer. J. Math. \textbf{62}(1) (1940) 1--42.

\bibitem{Ramsey} F. P. Ramsey, On a problem of formal logic, Proc. London
Math. Soc. (2) \textbf{30} (1930) 264--286.

\bibitem{RSZ2011} G. Richomme, K. Saari, L. Q. Zamboni, Abelian complexity
of minimal subshifts, J. Lond. Math. Soc. (2) \textbf{83}(1) (2011) 79--95.

\bibitem{Schlortt2026} C. Schlortt, On the structure of sequences with
minimal maximal pattern complexity, Ergodic Theory Dynam. Systems \textbf{46}%
(8) (2026) 2131--2148.
\end{thebibliography}
\end{document}